\documentclass[10pt]{amsart}
\usepackage{amsmath}
\usepackage{amssymb, amsfonts, amsmath, amsthm, amscd}
\usepackage{mathrsfs}
\usepackage{color, graphics}
\usepackage[all]{xy}
\usepackage{hyperref}
\usepackage{amsxtra} 
\usepackage{pgf,tikz}

\newtheorem{theorem}{Theorem}[section]

\newtheorem*{theorem*}{Main Theorem}

\newtheorem{proposition}[theorem]{Proposition}
\newtheorem{lemma}[theorem]{Lemma}
\newtheorem{claim}[theorem]{Claim}

\newtheorem{remark}[theorem]{Remark}

\newtheorem{definition}[theorem]{Definition}
\theoremstyle{remark}

\numberwithin{equation}{section}

\def\deg{\operatorname{deg}}%
\def\dim{\operatorname{dim}}%
\def\max{\operatorname{max}}%
\def\ker{\operatorname{ker}}%

\def\pic{\hbox{\rm Pic}}

\def\Pic{\operatorname{Pic}}%
\def\deg{\operatorname{deg}}%
\def\dim{\operatorname{dim}}%
\def\max{\operatorname{max}}%
\def\ker{\operatorname{ker}}%

\newcommand \im   {\ensuremath{\mathrm{im}}}

\newcommand \ext {\ensuremath{\mathrm{Ext}}}
\newcommand \Sec {\ensuremath{\mathrm{Sec}}}

\newcommand \coker {\ensuremath{\mathrm{coker}}}

\newcommand \lra {\rightarrow}

\def\deg{\mbox{deg}}

\def\Oc{\mathcal O}

\def\N{\mathcal N}

\def\Ff{\mathcal F}

\def\Pp{\mathbb P}

\begin{document}
\title[Universal extension spaces, modular maps: Brill-Noether loci]{Universal extension spaces and modular maps: unveiling irreducible components of 
Brill-Noether loci of stable bundles on a general $\nu$-gonal curve}
\author{Youngook Choi}
\address{Department of Mathematics Education, Yeungnam University, 280 Daehak-Ro, Gyeongsan, Gyeongbuk 38541,
 Republic of Korea }
\email{ychoi824@yu.ac.kr}
\author{Flaminio Flamini}
\address{Dipartimento di Matematica, Universita' degli Studi di Roma Tor Vergata, Via della Ricerca Scientifica-00133 Roma, Italy}
\email{flamini@mat.uniroma2.it}
\author{Seonja Kim}
\address{Department of  Electronic Engineering,
Chungwoon University, Sukgol-ro, Nam-gu, Incheon, 22100, Republic of Korea}
\email{sjkim@chungwoon.ac.kr}

\subjclass[2020]{14H60, 14H10, 14D20, 14D22, 14J26, 14E05, 14B10, 14N05}

\keywords{Stable vector bundles; Moduli spaces; Degeneracy loci; Birational geometry; Gonality of curves}

\begin{abstract} We investigate the Brill–Noether theory of rank-two, degree-$d$, stable vector bundles of speciality $3$ on a general $\nu$-gonal curve $C$ of genus $g$. Our approach leverages universal extension spaces, modular maps, and recent advancements in rank-one Brill-Noether theory over Hurwitz spaces.

We establish existence criteria for the corresponding Brill-Noether loci and provide a comprehensive description of their irreducible components. We moreover prove that these components exhibit diverse geometric behaviors, categorized by their regularity, superabundance, and the properties of their general points. Notably, for specific degrees $d$, we prove the coexistence of multiple superabundant components alongside a regular one. Using specialization techniques, we uncover a stratification into locally closed subschemes within these components and provide insights into their birational geometry and local structure. Furthermore, our results yield also consequences for Brill-Noether loci of stable, rank-two bundles with a fixed general determinant.
\end{abstract}

\maketitle

\tableofcontents

\section*{Introduction} Brill-Noether theory on smooth, complex projective curves $C$ of genus $g$ explores the geometry of maps $C \to \mathbb{P}^r$ of degree $d$. These maps correspond to the Brill-Noether loci:
$$W_d^r(C) := \left\{ L \in \text{Pic}^d(C) \mid h^0(C, L) \geq r+1\right\} \subseteq \text{Pic}^d(C).$$For a {\em general curve} $C$ (i.e. when $C$ corresponds to a general point of the {\em moduli space} $\mathcal{M}_g$ of curves of genus $g$) the properties of $W_d^r(C)$ — such as non-emptiness, dimension, and smoothness — are well-established  by main theorems from the 1970-80's (due to several authors as D. Eisenbud-J. Harris,\; W. Fulton-R. Lazarsfeld, D. Gieseker, P. Griffiths-J. Harris, G. Kempf, S. Kleiman-D. Laksov, K. Petri, just to mention a few) which give good descriptions of the geometric properties of such $W^r_d (C)$'s, which turn out to be determined by the invariants $(g, r, d)$, cf. e.g. \cite{ACGH}. 

In contrast, the theory for higher-rank vector bundles remains partially open, even for $C$ a general curve. Precisley, let $U_C(n, d)$ be the moduli space of {\em (semi)stable vector bundles} of rank $n$ and degree $d$, and $U^s_C(n, d)$ be the open dense subset of {\em stable bundles}. For any non-negative integer $k$, one can define the {\em Brill-Noether} locus via (locally) determinantal subschemes as:$$B^k_{n,d}  := \left\{ [\mathcal{F}] \in U_C(n, d) \mid h^0(\mathcal{F}) \geq k \right\} \subseteq U_C(n,d).$$Focusing on the rank $n=2$ case, we simply denote $U_C(d) := U_C(2,d)$ and $B^k_d := B^k_{2,d}$. These loci are central to e.g. applications in birational geometry \cite{Beau, Muk}, Hilbert schemes of curves and surfaces \cite{CCFM, CIK} and are also related to Bridgeland stability \cite{FFR}.

Using Serre duality and stability considerations, one may focus in the range $2g-2 \leq d \leq 4g-4$ with $k = k_i := d-2g+2+i$, for any integer $i \geq 1$ which is called the {\em speciality} of the vector bundles. 
The expected dimension of $B_d^{k_i} \cap U^s_C(d)$ is the so called {\em rank-two Brill-Noether number}:$$\rho_d^{k_i} := 4g-3-ik_i.$$An irreducible component $\mathcal{B} \subseteq B_d^{k_i} \cap U_C^s(d)$ is said to be {\em regular} if it is generically smooth of dimension $\rho_d^{k_i}$, and {\em superabundant} otherwise. Unlike the line bundle case, $B^{k_i}_d$ may exhibit unexpected behavior even for $C$ a general curve, i.e. with {\em general moduli}  (cf. e.g.\,\cite{BeFe,BGN,GT} and also Remark \ref{rem:BNloci}-(b) below). 

Nevertheless, for $i=1$ notably $B^{k_1}_d \cap U_C^s(d)$ is known to be irreducible and of the expected dimension for any smooth $C$ of genus $g$, cf. \cite{L,Sun}.

For speciality $i=2$, N. Sundaram \cite{Sun} proved  that $B^{k_2}_d\cap U_C^s(d)$ is not empty for any  smooth, irreducible projective curve $C$ and for odd degree $2g-1 \le d\le 3g-4$; later on M. Teixidor I Bigas generalized Sundaram's result as follows:

\medskip

\noindent
{\bf Result} (\cite{Teixidor1,Teixidor2}) {\em For any $2g-2 \leq d \leq 4g-4$ and for any smooth, irreducible, projective curve $C$ of genus $g$ the locus $B^{k_2}_d \cap U^s_C(d)$ is either empty or it contains a component $\mathcal B$ of (expected)  dimension $\rho^{k_2}_d=8g-2d-11$. 

When in particular $C$ is assumed to be a {\em general} curve, i.e. $C$ with {\em general moduli}, then such a $\mathcal B$  is the only component of $B^{k_2}_d \cap U^s_C(d)$, hence $B^{k_2}_d \cap U^s_C(d)$ is irreducible, moreover it is {\em regular}.}

\medskip

It is clear that, when $C$ of genus $g$ is assumed to be with {\em general moduli}, one can be supported by appropriate (and very delicate) degeneration arguments together with the theory of {\em limit linear series} in the sense of D. Eisenbud and J. Harris (cf. \cite{EH}). This is exactly the approach that e.g. M. Teixidor I Bigas used to get theorem below for  $B^{k_i}_d\cap U_C^s(d)$, $i \geq 2$, letting $C$ degenerate to a chain of $g$ elliptic curves.

\medskip

\noindent
{\bf Result} (\cite{Teixidor3}) {\em If $i\geq 2$ and $d-2g+2+i\geq 2$, then $B^{k_i}_d\cap U_C^s(d)$ is non-empty on $C$ a {\em general} curve of genus $g$, i.e. $C$ with {\em general moduli}, for $\rho_d^{k_i}\geq 1$ and $d$ odd or for $\rho_d^{k_i}\geq 5$ and $d$ even. In all these cases there is one component of $B^{k_i}_d\cap U_C^s(d)$ of the expected dimension.} 

\medskip

In the present paper, we study rank-two Brill-Noether theory on smooth curves $C$ of genus $g$ equipped with some {\em ``atypical"} maps $f$ which force $C$  to be not general in moduli (i.e. curves with {\em special moduli}). The basic and most natural case occurs when 
$f :C \to \mathbb P^1$ and $\deg(f) = \nu$ with $\nu < \lfloor \frac{g+3}{2}\rfloor$ (the integer $\lfloor \frac{g+3}{2}\rfloor$ being the {\em general gonality}, i.e. the gonality of a curve of genus $g$ with general moduli), and when $C$ is {\em general} among genus-$g$ curves possessing such a degree-$\nu < \lfloor \frac{g+3}{2}\rfloor$ map $f$ onto $\mathbb P^1$ (the case $\nu = 2$, resp. $3$, being the classical case of {\em hyperelliptic} curves studied by Clifford \cite{Cli} in 1878, resp. of {\em trigonal} curves studied by Maroni \cite{Maro} in 1946). The term ``general" in this context can be made more precise by considering the {\em Hurwitz scheme} $\mathcal H_{g,\nu}$ which is irreducible, parametrizing branched covers $ C \stackrel{f}{\longrightarrow} \mathbb P^1$ of degree $\nu$ and genus $g$, and which admits a natural {\em modular map} $$\mathcal H_{g,\nu} \stackrel{\Phi}{\longrightarrow}\mathcal M_g,$$given by forgetting maps $f$. When $\nu \geq \lfloor \frac{g+3}{2}\rfloor$, the map $\Phi$ is dominant and $C$ is with general moduli, as above. This is the reason why we restrict our attention to cases $\nu < \lfloor \frac{g+3}{2}\rfloor$ and refer to $C$ to be a {\em general $\nu$-gonal} curve of genus $g$  when it corresponds to a general point $[C]$ of ${\rm Im}(\Phi)$.

\medskip

In the forthcoming paper \cite{CFK3} we prove:

\medskip

\noindent
{\bf Result} (\cite{CFK3}) {\em For any $2g-2 \leq d \leq 4g-4$ and on a general $\nu$-gonal curve $C$ of genus $g$, according to the choice of $d$, the Brill-Noether locus $B^{k_2}_d \cap U^s_C(d)$ is 

\noindent
$\bullet$ either empty, or 

\noindent
$\bullet$ it consists of only one irreducible component $B_{\rm reg,2}$ (which coincides with the component $\mathcal B$ mentioned in Teixidor's result above), which is also {\em regular} and  {\em uniruled}, providing also an explicit description of the stable vector bundle corresponding to the general point of $B_{\rm reg,2}$, or 

\noindent
$\bullet$ it consists of two irreducible component,  $B_{\rm reg,2}$ as above and also 
another component $B_{\rm sup,2}$ which is {\em superabundant},  being of dimension $6g-d-2 \nu -6 > \rho^{k_2}_d=8g-2d-11$, generically smooth and {\em ruled}, providing also an explicit description of the stable vector bundle corresponding to the general point of this extra component.}

\medskip

The present paper focuses instead on speciality $i=3$. We use some recent results in Brill-Noether theory of line bundles on a general $\nu$-gonal curve $C$ as in \cite{CJ,FFR,JR,LLV,Lar,P} in order to study and classify  irreducible components of Brill-Noether loci  $B_d^{k_3} \cap U^s_C(d)$ when $C$ is a general $\nu$-gonal curve of genus $g$, with $3 \leq \nu < \lfloor \frac{g+3}{2}\rfloor$. 

Our methodology focuses on obtaining explicit geometric descriptions of rank-two stable bundles $\mathcal{F}$ representing general points of the irreducible components of $B^{k_3}_d \cap U^s_C(d)$. These bundles are characterized by their presentations as extensions of line bundles:\begin{equation}\tag{$*$}0 \to N \to \mathcal{F} \to L \to 0,\end{equation}where $L$ and $N$ are chosen from rank-one Brill-Noether loci on the $\nu$-gonal curve $C$. This approach, rooted in the work of Segre \cite{Seg} and modernized by e.g. Beauville \cite{Beau},  Bertram--Feinberg \cite{Be,BeFe} and Mukai \cite{Muk}, allows us to translate rank-two theory into the study of special unisecants on the ruled surface $\mathbb{P}(\mathcal{F})$. 

We construct irreducible universal extension spaces $\mathcal{W}$ parametrizing triples $(N, L, \mathcal{F})$; the bundle $\mathcal{F}$ is defined in such a way that its speciality, which is obviously related to specialities of $L$ and $N$ and to the coboundary map $\partial: H^0(C, L) \to H^1(C, N)$, is preserved to $h^1(C, \mathcal{F}) = 3$. Once the general bundle $\Ff$ in $\mathcal W$ is proved to be also stable, we can define a rational modular map:$$\pi: \mathcal{W} \dasharrow U^s_C(d), \quad (N, L, \mathcal{F}) \mapsto [\mathcal{F}].$$Then one has to show that image of $\pi$ actually constitutes an irreducible component $\mathcal{B} \subseteq B^{k_3}_d \cap U^s_C(d)$. This framework enables us to categorize such components as {\em first}, {\em second}, or {\em modified type} based on the properties of the presentation $(*)$ at a general point $[\Ff]$ of such a component. The parametric nature of $\mathcal{W}$ provides explicit strategies to study also the birational geometry of $\mathcal{B}$ and to establish degree-based bounds for existence; for instance, {\em regular components} cannot exist for $d \geq \frac{10}{3}g - 5$ (i.e. $\rho_d^{k_3} < 0$). More precisely, the main result of this paper is the following:

\medskip

\begin{theorem*}\label{thm:i=3} Let $C$ be a general $\nu$-gonal curve of genus $g$ and let $2g-2 \leq d \leq 4g-4$ be any integer.  Then, for $B^{k_3}_d \cap U^s_C(d)$, one has the following cases:

\bigskip

\noindent
(A) For $g \geq 4$, $3 \leq \nu < \lfloor \frac{g+3}{2}\rfloor$ and $4g-6 \leq d \leq 4g-4$, one has $B_d^{k_3}  \cap U_C^s(d) = \emptyset$.

\bigskip

\noindent
(B) For $g \geq 15$, $3 \leq \nu < \frac{g-2}{4}$ and $4g-4-2\nu \leq d \leq 4g-7$, one has $B^{k_3}_d \cap U^s_C(d)=\emptyset$. 

\bigskip

\noindent 
(C) Concerning irreducible components of {\em first type} (in the sense of Definition \ref{def:comp12mod} below), we have the following situation: 

\medskip

\noindent
(C-i)  If $g \geq 4$, $3 \leq \nu < \lfloor \frac{g+3}{2}\rfloor$ and either $2g-2 \leq d \leq 2g-7+ 2\nu$ or $4g-4-4\nu \leq d \leq 4g-7$, there is no irreducible component of {\em first type} of $B_d^{k_3} \cap U_C^s(d)$.

\medskip

\noindent
(C-ii) If $g \geq 9$, $3 \leq \nu \leq \frac{g}{3}$ and $2g-6+2\nu \leq d \leq  4g-5-4\nu$, there exists an irreducible component of $B_d^{k_3} \cap U_C^s(d)$, denoted by $\mathcal B_{\rm 3,\,(1-a)}$, which is {\em uniruled, generically smooth} and of dimension $$\dim(B_{\rm 3,\,(1-a)} )= 6g-6-4\nu - d.$$In particular $\mathcal B_{\rm 3,\,(1-a)}$ is {\em superabundant}, i.e. $ \mathcal B_{\rm 3,\,(1-a)} = \mathcal B_{\rm sup,\,3,\,(1-a)}$, when $d \geq 2g-5+2\nu$ whereas it is {\em regular}, i.e. $ \mathcal B_{\rm 3,\,(1-a)} = \mathcal B_{\rm reg,\,3,\,(1-a)}$, when otherwise $d = 2g-6+2\nu$. 

Furthermore, its general point $[\mathcal F] \in B_{\rm 3,\,(1-a)}$ corresponds to a rank-$2$ stable bundle of degree $d$, speciality $h^1(\mathcal F) = 3$, which is of {\em first type} (as in case $(1-a)$ of Lemma \ref{lem:barjvero} below), fitting in an exact sequence of the form $$0 \to N \to \mathcal F \to K_C-2A \to 0,$$where $A \in {\rm Pic}^{\nu}(C)$ is the unique line bundle on $C$ such that $|A|= g^1_{\nu}$, whereas $N \in {\rm Pic}^{d - 2g + 2 + 2\nu}(C)$ is general. Moreover $\mathcal B_{\rm 3,\,(1-a)}$ is the only irreducible component whose general point arises from a bundle $\mathcal F$ of {\em first type} (in the sense of Definition \ref{def:comp12mod} below) on $C$.

\bigskip

\noindent
(D) Concerning irreducible components of {\em modified type} (in the sense of Definition \ref{def:comp12mod} below), we have the following situation: 

\medskip

\noindent
(D-i)  If either $g \geq 6$, $ 3 \leq \nu \leq \frac{g}{2}$, $ 3g -5 \leq d \leq 4g-5-2\nu$ and $\nu+1 \leq s \leq g-\nu+1 $ or $g \geq 8$, $ 3 \leq \nu \leq \frac{g+4}{4}$, $ 3g - 3-\nu \leq d \leq 3g-6$ and $ g-\nu+2 \leq s \leq g-1$, there exists an irreducible component of $B_d^{k_3} \cap U_C^s(d)$, denoted by $\mathcal B_{\rm sup,\,3,\,(2-b)-mod}$, which is {\em uniruled, generically smooth} and {\em superabundant}, being of dimension
$$\dim \left(\mathcal B_{\rm sup,\,3,\,(2-b)-mod}\right) = 8g-11-2\nu - 2d > \rho_d^{k_3},$$whose general point $[\mathcal F] \in \mathcal B_{\rm sup,\,3,\,(2-b)-mod}$ corresponds to a rank-$2$ stable bundle of degree $d$ and speciality $h^1(\mathcal F) = 3$, which is obtained by natural modification of bundles of {\em second type} (as in case $(2-b)$ of Lemma \ref{lem:barjvero} below). Such a bundle $\Ff$ fits in an exact sequence of the form  
\[
0\to \omega_C(-D_s) \to\mathcal F\to \omega_C \otimes A^{\vee} \to 0, 
\] where $A \in {\rm Pic}^{\nu}(C)$ is the unique line bundle on $C$ such that $|A|= g^1_{\nu}$ and where $D_s\in C^{(s)}$ is general.

\medskip

\noindent
(D-ii)  If either $g \geq 9$, $ 3 \leq \nu \leq \frac{g+3}{4}$, $2g-7+3\nu \leq d \leq 3g-4-\nu$  or otherwise $g \geq 10$, $ 3 \leq \nu \leq \frac{g+5}{5}$, \linebreak $ 2g-7+2\nu \leq d \leq 2g-8 + 3 \nu$, there exists an irreducible component of $B_d^{k_3} \cap U_C^s(d)$, denoted by $\mathcal B_{\rm 3,\,(2-b)-mod}$, which is {\em uniruled, generically smooth} and  of dimension
$$\dim \left(\mathcal B_{\rm 3,\,(2-b)-mod}\right) = 8g-11-2\nu - 2d \geq \rho_d^{k_3},$$whose general point $[\mathcal F] \in \mathcal B_{\rm 3,\,(2-b)-mod}$ corresponds to a rank-$2$ stable bundle of degree $d$ and speciality $h^1(\mathcal F) = 3$, which fits in an exact sequence of the form  
\[
0\to N \to\mathcal F\to \omega_C \otimes A^{\vee} \to 0\, 
\] where $A \in {\rm Pic}^{\nu}(C)$ is the unique line bundle on $C$ such that $|A|= g^1_{\nu}$ whereas $N \in Pic^{d-2g+2+\nu}(C)$ is  general of its degree, so non-effective and special. 

Moreover, the component $\mathcal B_{\rm 3,\,(2-b)-mod}$ is {\em superabundant}, i.e. $\mathcal B_{\rm 3,\,(2-b)-mod} = \mathcal B_{\rm sup, 3,\,(2-b)-mod} $, when \linebreak $2g-6+2\nu \leq d \leq 3g-4-\nu$,  whereas it is {\em regular}, i.e. $\mathcal B_{\rm 3,\,(2-b)-mod} = \mathcal B_{\rm reg, 3,\,(2-b)-mod} $, if $d=2g-7+2\nu$.

\bigskip

\noindent
(E)  Concerning irreducible components of {\em second type} (in the sense of Definition \ref{def:comp12mod} below), we have the following situation: 

\medskip

\noindent
(E-i) For any $g \geqslant 7$, $3 \leqslant \nu < \frac{g}{2}$ and $2g-2 \leqslant d \leqslant 3g-7$, the Brill-Noether locus $B_d^{k_3} \cap U_C^s(d)$ contains a {\em regular} component, denoted by $\mathcal B_{\rm reg,3}$, which is {\em uniruled}. For $[\Ff]\in  \mathcal B_{\rm reg,3}$ general, the corresponding bundle  $\Ff$ fits into an exact sequence  of the form 
$$0 \to N \to \Ff \to \omega_C  \to 0,$$where $N \in \Pic^{d-2g+2}(C)$ is general (special, non-effective), and $\Ff$ arises from a general point $u$ of an irreducible component $\mathcal W_2^*$ of the degeneracy locus 
$$\mathcal W_2:=\{u\in\ext^1(\omega_C,N)\ |\ {\rm corank} \left(H^0(\omega_C) \stackrel{\partial_u}{\longrightarrow} H^1(N)\right) \geq 2\}\subseteq \ext^1(\omega_C,N)$$ (cf. formula \eqref{W1} below, where $\partial_u$ denotes the natural coboundary map associated to the exact sequence defining $\Ff$),  such that the component $\mathcal W_2^*$ is of ({\em expected})  codimension $2d-10-4g$ (cf. formula \eqref{eq:clrt} below) and $u \in \mathcal W_2^*$ is such that ${\rm corank}(\partial_u) =2$.  Furthermore, the bundle also  fits in the exact sequence $$0 \to N(p) \to \Ff \to \omega_C(-p) \to 0,$$where $p \in C$ general and $N(p)  \in {\rm Pic}^{d-2g+3}(C)$ general of its degree.  

\medskip

\noindent
(E-ii) For any $g \geq 9$, $\frac{g+3}{4} \leq \nu < \frac{g}{2}$ and $3g-6 \leq d \leq \frac{10}{3}g-6$, the Brill-Noether locus $B^{k_3}_d  \cap U^s_C (d)$ admits a {\em regular}  component, denoted by $\mathcal B_{\rm reg,3}$,  whose general point $[\mathcal F] \in \mathcal B_{\rm reg,3}$ corresponds to a rank-$2$ stable bundle $\Ff$ of degree $d$ and speciality $h^1(\Ff)=3$, which fits into an exact sequence of the form 
$$0\to \omega_C \otimes G_t^{\vee}  \to \mathcal F\to  \omega_C \to 0,$$where $|G_t| = g^2_t$ is a base point free and birationally very ample linear system, corresponding to a general point $G_t \in \overline{W^{\vec{w}_{2,0}}} \subseteq W^2_t (C)$, where $\overline{W^{\vec{w}_{2,0}}}$ the irreducible component 
of the rank-$1$ Brill-Noether locus $ W^2_t (C)$ as in Lemma \ref{lem:w2tLar} below, where necessarily $\frac{2}{3}g+2 \leq t \leq g+2$, and where $\Ff$ arises as an extension from $ u \in {\rm Ext} ^1 (K_C  , K_C-g^2_{t})$ corresponding to a general point of an irreducible component $ \mathcal U_2$ of the degeneracy locus $\mathcal W_2$ similarly as above, such that $\mathcal U_2$ is {\em superabundant}, of dimension $\dim(\mathcal U_2) = 2$ and such that ${\rm corank}(\partial_u) =2$. Furthermore, the bundle $\Ff$  also fits in the exact sequence $$0 \to \omega_C\otimes G_t^{\vee} \otimes \mathcal O_C(p)  \to \Ff \to \omega_C(-p) \to 0.$$ 
\end{theorem*}

\begin{proof}  The proof of the above result is given in full details in Sect.\;\ref{S:BN3}. Precisely, Lemma \ref{specialityhigh} reminds that, for any $2g-2 \leq d \leq 4g-4$, there is no irreducible component of $B_d^{k_3} \cap U^s_C(d)$ whose general point $[\mathcal F]$ corresponds to a rank-$2$ stable vector bundle of speciality $j := h^1(\mathcal F) > 3$. Therefore, one can focus on stable bundles of degree $d$ and speciality exactly $i=h^1(\Ff) =3$. 

The proof of (A) is then given in Proposition \ref{lemmaB} whereas that of (B) in Proposition \ref{ChoiA}. For what concerns case (C), taking into account  Lemma \ref{lem:barjvero} and Definition \ref{def:comp12mod} below, the uniqueness among components of {\em first type} is a consequence of Propositions \ref{lem:youngooksup3}, \ref{prop:nocases}-$(i)$, \ref{prop:17mar12.52}, \ref{prop:fanc19-12} and Lemma \ref{lem:sup3case1-b}; all the rest is proved in Proposition \ref{prop:case1a}. Finally, proof of (D) is conducted in Proposition \ref{prop:case2b},  whereas that of (E) in Propositions \ref{i=3first-1} and \ref{thm:sonja}. 
\end{proof}

\begin{remark}\label{rem:graph} {\normalfont To sum-up, {\bf Main Theorem} states that Brill-Noether loci $B_d^{k_3} \cap U_C^s(d)$ on a general $\nu$-gonal curve $C$ have irreducible components as follows:

\vspace{0,2cm}

{\tiny $\bullet$ ${\underline{3\leq \nu\leq \frac{g}{6}}}$ 

{  \tiny
\definecolor{xdxdff}{rgb}{0.5,0.5,1}
\definecolor{qqqqff}{rgb}{0,0,1}
\definecolor{cqcqcq}{rgb}{0.75,0.75,0.75}

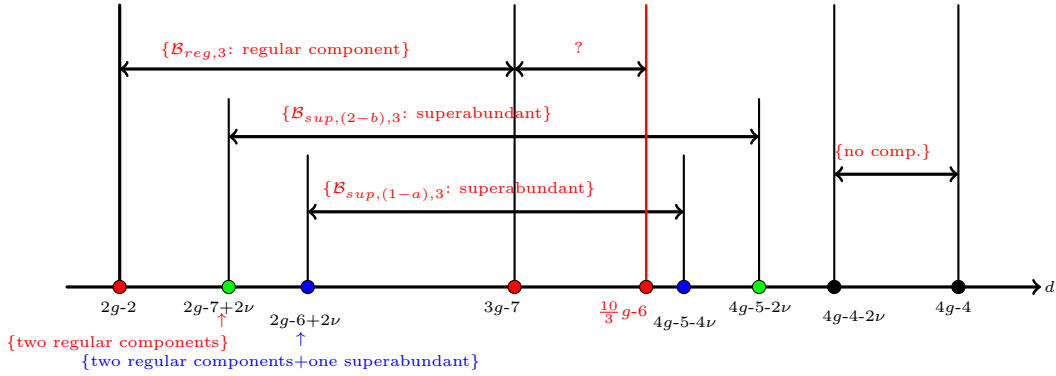
\begin{tikzpicture}[line cap=round,line join=round,x=0.5cm,y=0.5cm]

\draw[->,line width=1.2pt,color=black] (-0.4,0.) -- (25.5,0.);

\draw[color=black] (25.4,0) node[right] {$d$};

\clip(-2,-2.5) rectangle (25.5,8);

\begin{tiny}

\draw [line width=.8pt] (23.3,0.) -- (23.3,7.5);
\draw [fill=black] (23.3,0) circle (2.5pt); 
\draw (22.5,-0.2) node[anchor=north west] {{4$g$-4}};

\draw [line width=.8pt] (20,0.) -- (20,7.5);
\draw [fill=black] (20,0) circle (2.5pt); 
\draw (19.5,-0.45) node[anchor=north west] {{4$g$-4-2$\nu$}};
\draw [<->,line width=1.2pt,color=black]  (20,3.0) -- (23.3,3.0);
\draw (19.8,4.0) node[anchor=north west] {{\color{red}\{no comp.\}\color{black}}};

\draw [line width=.8pt] (18,0.) -- (18,5.0);
\draw [fill=green] (18,0) circle (2.5pt); 
\draw (17,-0.2) node[anchor=north west] {{4$g$-5-2$\nu$}};
\draw [<->,line width=1.2pt,color=black] (6,2.0) -- (16,2.0);
\draw (5.1,5) node[anchor=north west] {{\color{red}\{$\mathcal B_{sup,(2-b),3}$: superabundant\}\color{black}}};

\draw [line width=.8pt] (16,0.) -- (16,3.5);
\draw [fill=qqqqff] (16,0) circle (2.5pt); 
\draw (15,-0.6) node[anchor=north west] {{4$g$-5-4$\nu$}};
\draw [<->,line width=1.2pt,color=black] (3.9,4.0) -- (18,4.0);
\draw (6.2,3) node[anchor=north west] {{\color{red}\{$\mathcal B_{sup,(1-a),3}$: superabundant\} \color{black}}};

\draw [red, line width=.9pt] (15,0.) -- (15,7.5);
\draw [fill=red] (15,0) circle (2.5pt); 
\draw (13.5,-0.2) node[anchor=north west] {{\color{red}$\frac{10}{3}g$-6  \color{black}}};
\draw [<->,line width=1.2pt,color=black] (11.5,5.8) -- (15,5.8);
\draw (12.9,6.7) node[anchor=north west] {{\color{red} ? \color{black}}};

\draw [line width=.8pt] (11.5,0.) -- (11.5,7.5);
\draw [fill=red] (11.5,0) circle (2.5pt); 
\draw (10.5,-0.2) node[anchor=north west] {{3$g$-7}};

\draw [line width=.8pt] (6,0) -- (6,3.5);
\draw [fill=qqqqff] (6,0) circle (2.5pt); 
\draw (4.8,-0.55) node[anchor=north west] {{2$g$-6+2$\nu$}};

\draw [line width=.8pt] (3.9,0) -- (3.9,5.0);
\draw [fill=green] (3.9,0) circle (2.5pt); 
\draw (2.5,-0.2) node[anchor=north west] {{2$g$-7+2$\nu$}};

\draw [line width=1.2pt] (1,0) -- (1,7.5);
\draw [fill=red] (1,0) circle (2.5pt); 
\draw (0.3,-0.2) node[anchor=north west] {{2$g$-2}};

\draw [<->,line width=1.2pt,color=black] (1,5.8) -- (11.5,5.8);
\draw (1.9,6.7) node[anchor=north west] {{\color{red}\{$\mathcal B_{reg,3}$: regular component\}\color{black}}};

\draw (3.4,-0.5) node[anchor=north west] {{\color{red}$\uparrow$\color{black}}};
\draw (-2.2,-1.1) node[anchor=north west] {{\color{red}\{two regular components\}\color{black}}};

\draw (5.5,-1) node[anchor=north west] {{\color{blue}$\uparrow$\color{black}}};
\draw (-0.4,-1.6) node[anchor=north west] {{ \color{blue}\{two regular components+one superabundant\}\color{black}}};

\end{tiny}
\end{tikzpicture}
}}

\vspace{0,1cm}

{\normalfont 
\tiny $\bullet$ ${\underline{\frac{g}{6} < \nu\leq \frac{g+5}{5}}}$ 

{  \tiny
\definecolor{xdxdff}{rgb}{0.5,0.5,1}
\definecolor{qqqqff}{rgb}{0,0,1}
\definecolor{cqcqcq}{rgb}{0.75,0.75,0.75}

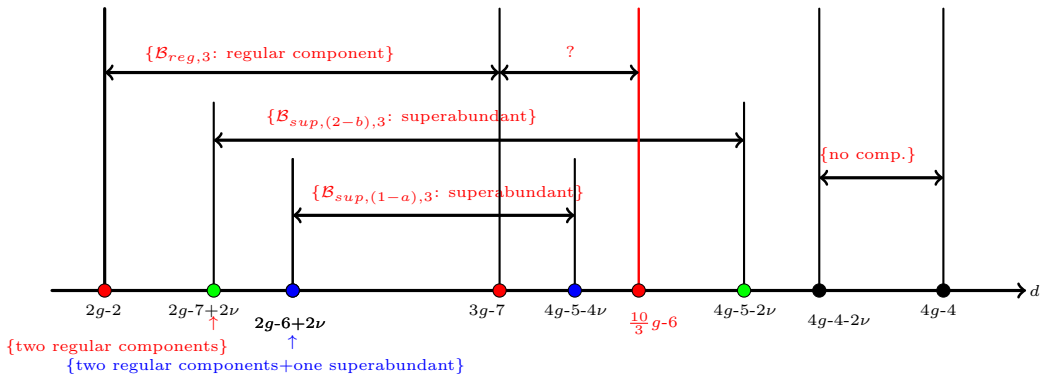
\begin{tikzpicture}[line cap=round,line join=round,x=0.5cm,y=0.5cm]

\draw[->,line width=1.2pt,color=black] (-0.4,0.) -- (25.5,0.);

\draw[color=black] (25.4,0) node[right] {$d$};

\clip(-2,-2.5) rectangle (25.5,8);

\begin{tiny}

\draw [line width=.8pt] (23.3,0.) -- (23.3,7.5);
\draw [fill=black] (23.3,0) circle (2.5pt); 
\draw (22.5,-0.2) node[anchor=north west] {{4$g$-4}};

\draw [line width=.8pt] (20,0.) -- (20,7.5);
\draw [fill=black] (20,0) circle (2.5pt); 
\draw (19.5,-0.45) node[anchor=north west] {{4$g$-4-2$\nu$}};
\draw [<->,line width=1.2pt,color=black]  (20,3.0) -- (23.3,3.0);
\draw (19.8,4.0) node[anchor=north west] {{\color{red}\{no comp.\}\color{black}}};

\draw [line width=.8pt] (18,0.) -- (18,5.0);
\draw [fill=green] (18,0) circle (2.5pt); 
\draw (17,-0.2) node[anchor=north west] {{4$g$-5-2$\nu$}};
\draw [<->,line width=1.2pt,color=black] (3.9,4.0) -- (18,4.0);
\draw (6.3,3) node[anchor=north west] {{\color{red}\{$\mathcal B_{sup,(1-a),3}$: superabundant\} \color{black}}};

\draw [line width=.8pt] (13.5,0.) -- (13.5,3.5);
\draw [fill=qqqqff] (13.5,0) circle (2.5pt); 
\draw (12.5,-0.2) node[anchor=north west] {{4$g$-5-4$\nu$}};
\draw [<->,line width=1.2pt,color=black] (6,2.0) -- (13.5,2.0);
\draw (5.1,5) node[anchor=north west] {{\color{red}\{$\mathcal B_{sup,(2-b),3}$: superabundant\}\color{black}}};

\draw [red, line width=.9pt] (15.2,0.) -- (15.2,7.5);
\draw [fill=red] (15.2,0) circle (2.5pt); 
\draw (14.7,-0.4) node[anchor=north west] {{\color{red}$\frac{10}{3}g$-6  \color{black}}};
\draw [<->,line width=1.2pt,color=black] (11.5,5.8) -- (15.2,5.8);
\draw (12.9,6.7) node[anchor=north west] { {  \color{red} ? \color{black}}};

\draw [line width=.8pt] (11.5,0.) -- (11.5,7.5);
\draw [fill=red] (11.5,0) circle (2.5pt); 
\draw (10.5,-0.2) node[anchor=north west] {{3$g$-7}};

\draw [line width=.8pt] (6,0) -- (6,3.5);
\draw [fill=qqqqff] (6,0) circle (2.5pt); 
\draw (4.8,-0.55) node[anchor=north west] {{2$g$-6+2$\nu$}};

\draw [line width=.8pt] (6,0) -- (6,3.5);
\draw [fill=qqqqff] (6,0) circle (2.5pt); 
\draw (4.8,-0.55) node[anchor=north west] {{2$g$-6+2$\nu$}};

\draw [line width=.8pt] (3.9,0) -- (3.9,5.0);
\draw [fill=green] (3.9,0) circle (2.5pt); 
\draw (2.5,-0.2) node[anchor=north west] {{2$g$-7+2$\nu$}};

\draw [line width=1.2pt] (1,0) -- (1,7.5);
\draw [fill=red] (1,0) circle (2.5pt); 
\draw (0.3,-0.2) node[anchor=north west] {{2$g$-2}};

\draw [<->,line width=1.2pt,color=black] (1,5.8) -- (11.5,5.8);
\draw (1.7,6.7) node[anchor=north west] { {  \color{red}\{$\mathcal B_{reg,3}$: regular component\}\color{black}}};

\draw (3.4,-0.5) node[anchor=north west] { {  \color{red}$\uparrow$\color{black}}};
\draw (-2,-1.1) node[anchor=north west] { {  \color{red}\{two regular components\}\color{black}}};

\draw (5.5,-1) node[anchor=north west] { {  \color{blue}$\uparrow$\color{black}}};
\draw (-0.4,-1.6) node[anchor=north west] { {  \color{blue}\{two regular components+one superabundant\}\color{black}}};

\end{tiny}
\end{tikzpicture}
}}


\newpage

{\normalfont 
\vspace{0,2cm}
\tiny $\bullet$ ${\underline{\frac{g+5}{5} <  \nu\leq \frac{g+3}{4}}}$ 
{ \tiny  
\definecolor{xdxdff}{rgb}{0.5,0.5,1}
\definecolor{qqqqff}{rgb}{0,0,1}
\definecolor{cqcqcq}{rgb}{0.75,0.75,0.75}

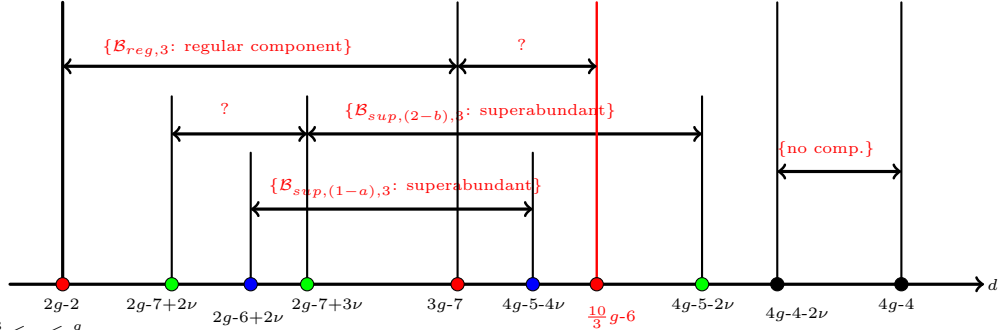
\begin{tikzpicture}[line cap=round,line join=round,x=0.5cm,y=0.5cm]

\draw[->,line width=1.2pt,color=black] (-0.4,0.) -- (25.5,0.);

\draw[color=black] (25.4,0) node[right] {$d$};

\clip(-2,-2.5) rectangle (25.5,8);

\begin{tiny}

\draw [line width=.8pt] (23.3,0.) -- (23.3,7.5);
\draw [fill=black] (23.3,0) circle (2.5pt); 
\draw (22.5,-0.2) node[anchor=north west] {{4$g$-4}};

\draw [line width=.8pt] (20,0.) -- (20,7.5);
\draw [fill=black] (20,0) circle (2.5pt); 
\draw (19.5,-0.45) node[anchor=north west] {{4$g$-4-2$\nu$}};
\draw [<->,line width=1.2pt,color=black]  (20,3.0) -- (23.3,3.0);
\draw (19.8,4.0) node[anchor=north west] {{\color{red}\{no comp.\}\color{black}}};

\draw [line width=.8pt] (18,0.) -- (18,5.0);
\draw [fill=green] (18,0) circle (2.5pt); 
\draw (17,-0.2) node[anchor=north west] {{4$g$-5-2$\nu$}};
\draw [<->,line width=1.2pt,color=black] (7.5,4.0) -- (18,4.0);
\draw (8.3,5) node[anchor=north west] {{\color{red}\{$\mathcal B_{sup,(2-b),3}$: superabundant\} \color{black}}};
\draw [<->,line width=1.2pt,color=black] (3.9,4.0) -- (7.5,4.0);
\draw (5,5) node[anchor=north west] {{\color{red} ? \color{black}}};

\draw [line width=.8pt] (13.5,0.) -- (13.5,3.5);
\draw [fill=qqqqff] (13.5,0) circle (2.5pt); 
\draw (12.5,-0.2) node[anchor=north west] {{4$g$-5-4$\nu$}};
\draw [<->,line width=1.2pt,color=black] (6,2.0) -- (13.5,2.0);
\draw (6.3,3) node[anchor=north west] {{\color{red}\{$\mathcal B_{sup,(1-a),3}$: superabundant\}\color{black}}};

\draw [red, line width=.9pt] (15.2,0.) -- (15.2,7.5);
\draw [fill=red] (15.2,0) circle (2.5pt); 
\draw (14.7,-0.4) node[anchor=north west] {{\color{red}$\frac{10}{3}g$-6  \color{black}}};
\draw [<->,line width=1.2pt,color=black] (11.5,5.8) -- (15.2,5.8);
\draw (12.9,6.7) node[anchor=north west] {{\color{red} ? \color{black}}};

\draw [line width=.8pt] (11.5,0.) -- (11.5,7.5);
\draw [fill=red] (11.5,0) circle (2.5pt); 
\draw (10.5,-0.2) node[anchor=north west] {{3$g$-7}};

\draw [line width=.8pt] (7.5,0) -- (7.5,5);
\draw [fill=green] (7.5,0) circle (2.5pt); 
\draw (6.9,-0.2) node[anchor=north west] {{2$g$-7+3$\nu$}};

\draw [line width=.8pt] (6,0) -- (6,3.5);
\draw [fill=qqqqff] (6,0) circle (2.5pt); 
\draw (4.8,-0.55) node[anchor=north west] {{2$g$-6+2$\nu$}};

\draw [line width=.8pt] (3.9,0) -- (3.9,5.0);
\draw [fill=green] (3.9,0) circle (2.5pt); 
\draw (2.5,-0.2) node[anchor=north west] {{2$g$-7+2$\nu$}};

\draw [line width=1.2pt] (1,0) -- (1,7.5);
\draw [fill=red] (1,0) circle (2.5pt); 
\draw (0.3,-0.2) node[anchor=north west] {{2$g$-2}};

\draw [<->,line width=1.2pt,color=black] (1,5.8) -- (11.5,5.8);
\draw (1.7,6.7) node[anchor=north west] { {  \color{red}\{$\mathcal B_{reg,3}$: regular component\}\color{black}}};

\end{tiny}
\end{tikzpicture}
}}




{\normalfont 
\vspace{-0,8cm}
\tiny $\bullet$ ${\underline{\frac{g+3}{4}\leq \nu\leq \frac{g}{3}}}$

{  \tiny
\definecolor{xdxdff}{rgb}{0.5,0.5,1}
\definecolor{qqqqff}{rgb}{0,0,1}
\definecolor{cqcqcq}{rgb}{0.75,0.75,0.75}

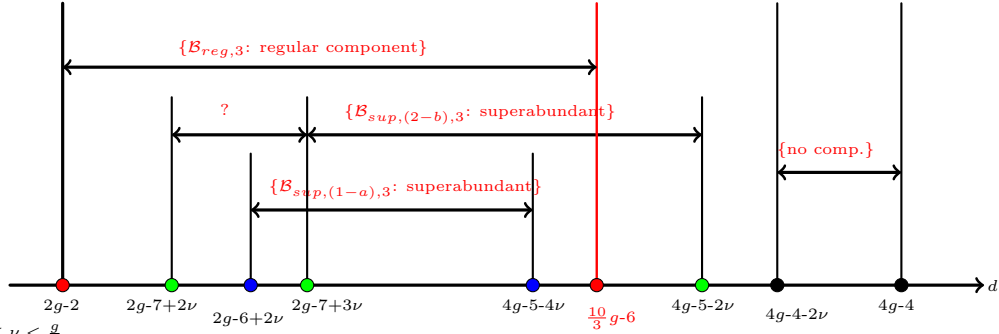
\begin{tikzpicture}[line cap=round,line join=round,x=0.5cm,y=0.5cm]

\draw[->,line width=1.2pt,color=black] (-0.4,0.) -- (25.5,0.);

\draw[color=black] (25.4,0) node[right] {$d$};

\clip(-2,-2.5) rectangle (25.5,8);

\begin{tiny}

\draw [line width=.8pt] (23.3,0.) -- (23.3,7.5);
\draw [fill=black] (23.3,0) circle (2.5pt); 
\draw (22.5,-0.2) node[anchor=north west] {{4$g$-4}};

\draw [line width=.8pt] (20,0.) -- (20,7.5);
\draw [fill=black] (20,0) circle (2.5pt); 
\draw (19.5,-0.45) node[anchor=north west] {{4$g$-4-2$\nu$}};
\draw [<->,line width=1.2pt,color=black]  (20,3.0) -- (23.3,3.0);
\draw (19.8,4.0) node[anchor=north west] {{\color{red}\{no comp.\}\color{black}}};

\draw [line width=.8pt] (18,0.) -- (18,5.0);
\draw [fill=green] (18,0) circle (2.5pt); 
\draw (17,-0.2) node[anchor=north west] {{4$g$-5-2$\nu$}};
\draw [<->,line width=1.2pt,color=black] (7.5,4.0) -- (18,4.0);
\draw (8.3,5) node[anchor=north west] {{\color{red}\{$\mathcal B_{sup,(2-b),3}$: superabundant\} \color{black}}};
\draw [<->,line width=1.2pt,color=black] (3.9,4.0) -- (7.5,4.0);
\draw (5,5) node[anchor=north west] {{\color{red} ? \color{black}}};

\draw [line width=.8pt] (13.5,0.) -- (13.5,3.5);
\draw [fill=qqqqff] (13.5,0) circle (2.5pt); 
\draw (12.5,-0.2) node[anchor=north west] {{4$g$-5-4$\nu$}};
\draw [<->,line width=1.2pt,color=black] (6,2.0) -- (13.5,2.0);
\draw (6.3,3) node[anchor=north west] {{\color{red}\{$\mathcal B_{sup,(1-a),3}$: superabundant\}\color{black}}};

\draw [red, line width=.9pt] (15.2,0.) -- (15.2,7.5);
\draw [fill=red] (15.2,0) circle (2.5pt); 
\draw (14.7,-0.4) node[anchor=north west] {{\color{red}$\frac{10}{3}g$-6  \color{black}}};


\draw [line width=.8pt] (7.5,0) -- (7.5,5);
\draw [fill=green] (7.5,0) circle (2.5pt); 
\draw (6.9,-0.2) node[anchor=north west] {{2$g$-7+3$\nu$}};

\draw [line width=.8pt] (6,0) -- (6,3.5);
\draw [fill=qqqqff] (6,0) circle (2.5pt); 
\draw (4.8,-0.55) node[anchor=north west] {{2$g$-6+2$\nu$}};

\draw [line width=.8pt] (3.9,0) -- (3.9,5.0);
\draw [fill=green] (3.9,0) circle (2.5pt); 
\draw (2.5,-0.2) node[anchor=north west] {{2$g$-7+2$\nu$}};

\draw [line width=1.2pt] (1,0) -- (1,7.5);
\draw [fill=red] (1,0) circle (2.5pt); 
\draw (0.3,-0.2) node[anchor=north west] {{2$g$-2}};

\draw [<->,line width=1.2pt,color=black] (1,5.8) -- (15.2,5.8);
\draw (3.7,6.7) node[anchor=north west] { {  \color{red}\{$\mathcal B_{reg,3}$: regular component\}\color{black}}};


\end{tiny}
\end{tikzpicture}
}}


{\normalfont 
\vspace{-0,8cm}
\tiny $\bullet$ ${\underline{\frac{g}{3}< \nu<  \frac{g}{2}}}$

{  \tiny
\definecolor{xdxdff}{rgb}{0.5,0.5,1}
\definecolor{qqqqff}{rgb}{0,0,1}
\definecolor{cqcqcq}{rgb}{0.75,0.75,0.75}

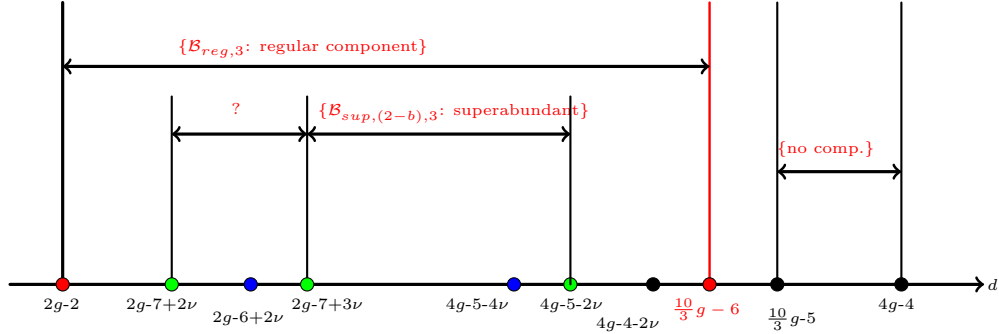
\begin{tikzpicture}[line cap=round,line join=round,x=0.5cm,y=0.5cm]

\draw[->,line width=1.2pt,color=black] (-0.4,0.) -- (25.5,0.);

\draw[color=black] (25.4,0) node[right] {$d$};

\clip(-2,-2.5) rectangle (25.5,8);

\begin{tiny}

\draw [line width=.8pt] (23.3,0.) -- (23.3,7.5);
\draw [fill=black] (23.3,0) circle (2.5pt); 
\draw (22.5,-0.2) node[anchor=north west] {{4$g$-4}};

\draw [line width=.8pt] (20,0.) -- (20,7.5);
\draw [fill=black] (20,0) circle (2.5pt); 
\draw (19.5,-0.45) node[anchor=north west] {{$\frac{10}{3}g$-5 }};
\draw [<->,line width=1.2pt,color=black]  (20,3.0) -- (23.3,3.0);
\draw (19.8,4.0) node[anchor=north west] {{\color{red}\{no comp.\}\color{black}}};

\draw [red, line width=.9pt] (18.2,0.) -- (18.2,7.5);
\draw [fill=red] (18.2,0) circle (2.5pt); 
\draw (17,-0.2) node[anchor=north west] {{\color{red}$\frac{10}{3}g-6$  \color{black}}};
\draw [<->,line width=1.2pt,color=black] (7.5,4.0) -- (14.5,4.0);
\draw (7.6,5) node[anchor=north west] {{\color{red}\{$\mathcal B_{sup,(2-b),3}$: superabundant\} \color{black}}};
\draw [<->,line width=1.2pt,color=black] (3.9,4.0) -- (7.5,4.0);
\draw (5.3,5) node[anchor=north west] {{\color{red} ? \color{black}}};

\draw [fill=qqqqff] (13,0) circle (2.5pt); 
\draw (11,-0.2) node[anchor=north west] {{4$g$-5-4$\nu$}};

\draw [fill=green] (14.5,0) circle (2.5pt); 
\draw (13.5,-0.2) node[anchor=north west] {{4$g$-5-2$\nu$}};
\draw [line width=.8pt] (14.5,0.) -- (14.5,5.0);
\draw [fill=black] (16.7,0) circle (2.5pt); 
\draw (15,-0.7) node[anchor=north west] {{4$g$-4-2$\nu$}};


\draw [line width=.8pt] (7.5,0) -- (7.5,5);
\draw [fill=green] (7.5,0) circle (2.5pt); 
\draw (6.9,-0.2) node[anchor=north west] {{2$g$-7+3$\nu$}};

\draw [fill=qqqqff] (6,0) circle (2.5pt); 
\draw (4.8,-0.55) node[anchor=north west] {{2$g$-6+2$\nu$}};

\draw [line width=.8pt] (3.9,0) -- (3.9,5.0);
\draw [fill=green] (3.9,0) circle (2.5pt); 
\draw (2.5,-0.2) node[anchor=north west] {{2$g$-7+2$\nu$}};

\draw [line width=1.2pt] (1,0) -- (1,7.5);
\draw [fill=red] (1,0) circle (2.5pt); 
\draw (0.3,-0.2) node[anchor=north west] {{2$g$-2}};

\draw [<->,line width=1.2pt,color=black] (1,5.8) -- (18.2,5.8);
\draw (3.7,6.7) node[anchor=north west] { {  \color{red}\{$\mathcal B_{reg,3}$: regular component\}\color{black}}};

\end{tiny}
\end{tikzpicture}
}}

}
\end{remark}

Our strategy proves furthermore that for some $d$ the general fiber of a natural {\em determinant map} is frequently non-empty, even in cases where numerical expectations might suggest otherwise. More precisely, as a direct consequence of {\bf Main Theorem}, we also obtain:

\medskip

\noindent
{\bf Main Corollary.} {\em Take assumptions and notation as in {\bf Main Theorem}  and set $M \in {\rm Pic}^d(C)$ general. Consider the natural surjective {\em determinant map}
$$\det: U_C(d) \to {\rm Pic}^d(C), \;\; [\Ff] \longrightarrow \det(\Ff).$$Set 
$$U_C(M):= {\det}^{-1}(M),\;\; U_C^s(M):= U_C(M) \cap U_C^s(d) \;\; {\rm and}  \;\; B_M^{k_3} := \left\{[\Ff] \in U_C(M) \;\; | \;\; h^0(\Ff) \geqslant k_3\right\}.$$One has the following situation. 

\medskip

\noindent
$\bullet$ With numerical assumptions as in (C-ii) of {\bf Main Theorem},  
the Brill-Noether locus $B_M^{k_3}\cap U_C^s(M)$ is not empty, even when its {\em expected dimension} $\rho_M^{k_3} = 9g-18-3d$ is negative,  which occurs for \linebreak $d > 3g-6$ and $\nu<\frac{g+1}{4}$.  More precisely, such a locus contains an irreducible component of {\em first type}, namely $\mathcal B_{\rm 3,(1-a),M} := U^s_C(M) \cap \mathcal B_{\rm 3,(1-a)}$, whose general point corresponds to a bundle $\Ff$ as in (C-ii), with $\det(\Ff) = K_C-2A + N = M \in {\rm Pic}^d(C)$ general. 

Moreover $\dim(\mathcal B_{\rm 3,(1-a),M}) = 5g-6-4\nu-d$, which is of {\em expected dimension} $\rho_M^{k_3} = 9g-18-3d$ if and only if $d = 2g-6 + 2\nu$, otherwise it is {\em superabundant}.

\medskip

\noindent
$\bullet$ With numerical assumptions as in (D-ii) of {\bf Main Theorem} together with $d \leq 3g-4-\nu$, the Brill-Noether locus $B_M^{k_3}\cap U_C^s(M)$ is not empty. More precisely, it contains an irreducible component of {\em modified type}, namely $\mathcal B_{\rm 3,(2-b)-mod,\,M} := U^s_C(M) \cap \mathcal B_{\rm 3,(2-b)-mod}$, 
whose general point corresponds to a bundle $\Ff$ as in (D-ii), with  $\det(\Ff) = K_C-A +N = M \in {\rm Pic}^d(C)$ general.  

Moreover  $\dim( \mathcal B_{\rm 3,(2-b)-mod,\,M}) = 7g-11-2\nu-2d$,  which is the {\em expected dimension} $\rho_M^{k_3} = 9g-18-3d \geq 0$ if and only if $d = 2g-7 + 2\nu$, otherwise $\mathcal B_{\rm 3,(2-b)-mod,\,M}$ is {\em superabundant}.} 
\medskip

\begin{proof} Consider the restriction of the determinant map    
$B^{k_3}_d \cap U^s_C(d) \stackrel{det_|}{\longrightarrow} {\rm Pic}^d(C)$ and its further restriction to the irreducible components constructed in {\bf Main Theorem}. 

Notice first that, for $M \in {\rm Pic}^d(C)$ general, from \cite{LNS} the {\em expected dimension} of $B_M^{k_3}\cap U_C^s(M)$ is $\rho_M^{k_3} = 3g-3 - 3 k_3 = 9g-18-3d$, which is negative if and only if $d > 3g-6$. 

For the component of {\em first type} as in {\bf Main Theorem}-(C-ii), from the generality of $N$ in \linebreak ${\rm Pic}^{d-2g+2+2\nu}(C)$ we deduce that $\det(\Ff) = N+K_C-2A$ is also general in ${\rm Pic}^{d}(C)$. Therefore, the map $\det_{|_{\mathcal B_{\rm 3,(1-a)}}}$ is dominant onto ${\rm Pic}^d(C)$. Hence, for $M \in {\rm Pic}^{d}(C)$ general, one has $\det_{|_{\mathcal B_{\rm 3,(1-a)}}}^{-1}(M) \neq \emptyset$. Moreover, since 
$M=N + K_C - 2A \cong M':= N' + K_C - 2 A'$ if and only if $N \cong N'$, one deduces the statement about the dimension  $\dim(\mathcal B_{\rm 3,(1-a),M}) = 5g-6-4\nu-d$. The fact that it is of the {\em expected dimension} $\rho_M^{k_3} = 9g-18-3d$ if and only if $d = 2g-6 + 2 \nu$, otherwise it is {\em superabundant}, follows from the fact that $\rho_M^{k_3} = \rho_d^{k_3} - g$, the dominance of $\det_{|_{\mathcal B_{\rm 3,(1-a)}}}$ and the fact that the component of {\em first type} as in {\bf Main Theorem}-(C-ii) is {\em superabundant} unless $d= 2g-6 + 2 \nu$, when it is {\em regular}.

Similar proof for the case (D-ii) as in {\bf Main Theorem}, dealing with components of {\em modified type}. On the other hand, differently from the previous case, for  $M \in {\rm Pic}^d(C)$ general the Brill-Noether locus $B_M^{k_3}\cap U_C^s(M)$ does not contain a component $ \mathcal B_{\rm 3,(2-b)-mod,\,M}$ from case (D) when $\rho_M^{k_3} = 9g-18-3d <0$. Indeed, for  $M \in {\rm Pic}^d(C)$ general, $B_M^{k_3}\cap U_C^s(M)$ contains a component of modified type only from cases (D-ii) in {\bf Main Theorem}; if such a component $ \mathcal B_{\rm 3,(2-b)-mod,\,M}$ existed for $\rho_M^{k_3} = 9g-18-3d <0$, namely when $d > 3g-6$ which is contrary to the upper bound of $d\leq 3g-4-\nu$.

Finally, components of {\em second type} as in {\bf Main Theorem}-(E) do not give rise to irreducible components of $B_M^{k_3}\cap U_C^s(M)$ for $M \in {\rm Pic}^{d}(C)$ general; indeed in case (E-i) of {\bf Main Theorem} we have $d \leq 3g-7$ where $\rho_M^{k_3} = 9g-18-3d <0$, whereas in case (E-ii) $\det(\Ff) = 2 K_C - G_t$ is not general in ${\rm Pic}^{d}(C)$. 
\end{proof}

\smallskip

\noindent
{\bf Acknowledgements} The first author was supported by the National Research Foundation of Korea (NRF) (RS-2024-00352592). The second author was supported by the MIUR Excellence Department Project MatMod@TOV (CUP-E83C23000330006, 2023–2027). The third author was supported by the NRF (2022R1A2C100597713). The authors express their gratitude to G. Pareschi for the AAAGH research funds (CUP E83C22001710005) that supported the first and third authors' visit to the University of Rome "Tor Vergata" in October 2023. They also thank D. S. Hwang, J.-M. Hwang, and Y. Lee for supporting the second author’s participation in the Workshop in Classical Algebraic Geometry (Daejeon, 2023). We are grateful to the Department of Mathematics at the University of Rome "Tor Vergata," the Complex Geometry Center (Daejeon), and KIAS for their hospitality. Special thanks go to A. Bajravani for pointing out reference \cite{Re}, and to F. Catanese and J.H. Keum for the invitation to the 46th Frontier Scientists Workshop (Ischia, 2025). Finally, we thank C. Ciliberto, J.M. Hwang, and E. Sernesi for valuable discussions and advice.

\section{Preliminaries}\label{S:pre} In this paper we work over $\mathbb C$. All schemes will be endowed with the Zariski topology. We will  interchangeably use the terms {\em rank-$r$ vector bundle} on a smooth, projective variety $X$ and {\em rank-$r$ locally free sheaf}; for this reason we may sometimes abuse notation by identifying divisor classes with  corresponding line bundles and interchangeably using additive and multiplicative notation. 
The dual bundle of a rank-$r$ vector bundle $\mathcal E$ will be denoted by $\mathcal E^{\vee}$; if $L$ is a line bundle, we therefore  interchangeably use $L^{\vee}$ or $-L$. 

Given a smooth, projective variety $X$, we will denote by $\sim$ the linear equivalence of Cartier divisors on $X$ and by $\sim_{alg}$ their algebraic equivalence. If $\mathcal M$ is a moduli space, parametrizing objects modulo a given equivalence relation, and if $Y$ is a representative of an equivalence class in $\mathcal M$,  we denote by $[Y] \in \mathcal M$ the point corresponding to $Y$. 

If $C$ denotes any smooth, irreducible, projective curve of genus $g$, as customary,  $W^r_d(C)$ will denote the  {\em Brill-Noether locus} as in the  Introduction,  \color{black} parametrizing  line bundles $L \in {\rm Pic}^d(C)$ such that $h^0(L) \geq r+1$. The integer
$$\rho(g,r,d) := g - (r+1) (g+r-d)$$
denotes the {\em Brill-Noether number} and the multiplication map
\begin{equation*}\label{eq:Petrilb}
\mu_L : H^0(C,L) \otimes H^0(\omega_C \otimes L^{\vee}) \to H^0(C, \omega_C)
\end{equation*} denotes the so called {\em Petri map} of the line bundle $L$. As for the rest, we will use standard terminology and notation as in \cite{ACGH,H}.


\subsection{Brill-Noether theory of line bundles on general $\nu$-gonal curves}\label{ss:BNlbgon}  Here we briefly review basic results in Brill-Noether theory of line bundles on a general $\nu$-gonal curve, which will be used in the sequel. From the Introduction, given a positive integer $g$, from now on $\nu$ will be an integer such that 
\begin{equation}\label{eq:nu}
3 \leq \nu < \lfloor \frac{g+3}{2}\rfloor, \; {\rm (in \; particular} \; g \geq 4). 
\end{equation} We will say that a smooth, irreducible, projective curve $C$ is a 
{\em general $\nu$-gonal curve of genus $g$} if $[C]$ is a general point of the {\em $\nu$-gonal stratum} $\mathcal M^1_{g,\nu} \subsetneq \mathcal M_g$, where $\mathcal M^1_{g,\nu}$ is the image via the natural modular map $\mathcal H_{g,\nu} \stackrel{\Phi}{\longrightarrow} \mathcal M_g $ where $\mathcal H_{g,\nu}$ the {\em Hurwitz scheme} as in the Introduction. We first remind the following:

\begin{proposition}[cf. \cite{Ball}]\label{Ballico} Let $C$ be a general $\nu$-gonal curve of genus $g$, let $r$ be a non-negative integer and let $A \in {\rm Pic}^{\nu}(C)$ be the unique line bundle on $C$ associated to the  unique base-point-free pencil on $C$ of degree $\nu$, i.e. $|A| = g^1_{\nu}$. If $g\geq r(\nu-1)$, then $\dim (|A^{\otimes r}|)=r$ whereas if $g< r(\nu-1)$ then $A^{\otimes r}$ is non-special hence $\dim (|A^{\otimes r}|) =r\,\nu-g>r$.
\end{proposition} Notice further that the unique $g^1_\nu$ on $C$ induces a morphism $f: C \to \mathbb{P}^1$. For any $L \in \text{Pic}^d(C)$, the pushforward $f_* (L)$ is a rank-$\nu$ bundle on $\mathbb{P}^1$ which, by Grothendieck's Theorem, decomposes as:$$f_*(L) \cong \bigoplus_{i=1}^{\nu} \mathcal{O}_{\mathbb{P}^1}(e_i), \quad \text{with } e_1 \leq \dots \leq e_{\nu} \; \text{ and } \; \sum_{i=1}^{\nu} e_i = d+1-g-\nu.$$The {\em splitting type vector} $\vec{e} := (e_1, \dots, e_{\nu})$ provides a refinement of the classical Brill-Noether loci. Indeed, one defines the {\em splitting locus} and its closure:$$W^{\vec{e}}(C) := \{ L \in \text{Pic}^d(C) \mid f_*(L) \cong \mathcal{O}_{\mathbb{P}^1}(\vec{e}) \}, \quad \overline{W^{\vec{e}}}: = \bigcup_{\vec{e}\,' \leq \vec{e}} W^{\vec{e}\,'}$$where $\leq$ denotes the standard partial ordering on partitions. The expected codimension of $W^{\vec{e}}(C)$ is the {\em magnitude} 
$$u(\vec{e}) := h^1(\mathbb{P}^1, \mathcal{E}nd(f_*(L)) = \sum_{i<j} \max\{0, e_j - e_i - 1\}$$cf. e.g. \cite{Lar}. It turns out that the classical Brill-Noether locus $W_d^r(C)$ decomposes into the union of splitting loci, namely 
\begin{equation}\label{eq:LarRef2}
W_d^r(C)\ = 
\bigcup_{\begin{array}{c} 
\vec{e} \; {\rm maximal \; s.t.} \; |\vec{e}| = d-g+1+\nu \\ 
h^0(\mathcal O_{\mathbb P^1}(\vec{e})) \geq r+1
\end{array}} \overline{W^{\vec{e}}}. 
\end{equation} 
\normalsize
Characterizing contributions of some splitting loci to the irreducible components of  $W_d^r(C)$ is therefore equivalent to 
determining those splitting type vectors that are maximal w.r.t. $\leq$ among those satisfying: 
\begin{equation}\label{eq:maximal}
\deg(\mathcal O_{\mathbb P^1}(\vec{e})) = \sum_{i=1}^{\nu} e_i = d+1-g-\nu \;\; {\rm and} \;\; h^0(\mathcal O_{\mathbb P^1}(\vec{e})) = \sum_{i=1}^{\nu} {\rm max}\,\{0,\,e_i+1\} \geq r+1.
\end{equation} These are the so-called  
{\em ``balanced plus balanced"} splitting type vectors in \cite[p.\,769]{Lar}, uniquely determined by the number of their non-negative components. In this set-up, one has

\begin{theorem}\label{thm:Lar2} Let $C$ be a general $\nu$-gonal curve of genus $g \geq 4$ and let $r$ and $d$ be non-negative integers. 

\noindent
$(1)$ If $r \leq d-g$, then $W_d^r(C) = {\rm Pic}^d(C)$;

\medskip

\noindent
$(2)$ If $r > d-g$, set $d' := d+1-g-\nu$. Then:
\begin{itemize}
\item[(2-i)]  a splitting type vector $\vec{e}$, which is maximal among those satisfying \eqref{eq:maximal}, is such that \linebreak $h^0(\mathcal O_{\mathbb P^1}(\vec{e}))= \sum_{i=1}^{\nu} {\rm max}\,\{0,\,e_i+1\} = r+1$ (cf. \cite[Lemma\,2.7]{CJ}); 

\item[(2-ii)] maximal splitting type vectors as in (2-i) are vectors $\vec{w}_{r,\ell}$ such that 
$$\mathcal O_{\mathbb P^1}(\vec{w}_{r,\ell}) = B(\nu-r-1+\ell, \; d'-\ell) \oplus B(r+1-\ell, \;\ell),$$where as above $B(- ,\; -)$ denotes the unique balanced bundle on $\Pp^1$ of given rank and degree, respectively, for 
${\rm max} \; \{0,\; r+2-\nu\} \leq \ell \leq r$ an integer such that either $\ell=0$ or $\ell \leq g+2r+1-d - \nu$. Moreover,
$u (\vec{w}_{r,\ell}) 
= g - \left(\rho(g, r - \ell, d) - \ell\,\nu\right)\ {\mbox{ (cf. \cite[Lemma\,2.2]{Lar})}}.$

\item[(2-iii)] for any integer $\ell$ as in (2-ii), the corresponding splitting type vector $\vec{w}_{r,\ell}$ is the unique 
degree--$d'$, rank--$\nu$ vector with $r+1-\ell$ non-negative coordinates, which is maximal among those with $r+1$ linearly independent global sections (cf. \cite[Cor.\;1.3]{Lar});

\item[(2-iv)] Every component of $W_d^r(C)$ is generically smooth, of {\em (expected) dimension} $$\rho' (g, \vec{w}_{r,\ell}) := 
g - u (\vec{w}_{r,\ell}) = \rho(g, r - \ell, d) - \ell\,\nu,$$  for some integer ${\rm max} \; \{0,\; r+2-\nu\} \leq \ell \leq r$, where $\rho(g, r - \ell, d)$ the classical {\em Brill-Noether number}, such that either $\ell=0$ or  $\ell \leq g + 2r+1 - d - \nu$. Such a component exists for each integer $\ell$ as above satisfying furthermore $\rho' (g, \vec{w}_{r,\ell}) \geq 0$ (cf. \cite[Cor.\;1.3]{Lar}). 
\end{itemize}
\end{theorem}

\subsection{Ruled surfaces and families of special unisecants}\label{ss:unisec} Let $\mathcal{F}$ be a rank-two vector bundle of degree $d$ on $C$. We denote by $\pi: F := \mathbb{P}(\mathcal{F}) \to C$ the associated ruled surface (or {\em surface scroll}). Let $H$ be a divisor in the class of the tautological line bundle $\mathcal{O}_F(1)$, such that $H^2 = \deg(\det \mathcal{F}) = d$. The speciality of the bundle (and of the scroll $F$) is defined as $i(\mathcal{F}) := h^1(C, \mathcal{F})$. 

Divisors $\widetilde{\Gamma}$ on $F$ such that $\mathcal{O}_F(\widetilde{\Gamma}) \cong \mathcal{O}_F(1) \otimes \pi^*(N^{\vee})$ for some $N \in \text{Pic}(C)$ are termed {\em unisecants} of $F$. Setting $f_D := \pi^*(D)$, for any $D \in \text{Div}(C)$, a unisecant satisfies $\widetilde{\Gamma} \sim H - f_D$, with degree $\deg(\widetilde{\Gamma}) = d - \deg(D)$. Irreducible unisecants correspond to {\em sections} of $F$ and are isomorphic to $C$. For a positive integer $\delta$, we denote by $\text{Div}_F^{1,\delta}$ the {\em Hilbert scheme of unisecants of degree} $\delta$ on $F$. Any unisecant $\widetilde{\Gamma}$ of degree $\delta \leq d$ admits a decomposition $\widetilde{\Gamma} = \Gamma + f_D$, where $\Gamma$ is a section and $D \geq 0$ an effective divisor on $C$. This gives rise to 
\begin{equation}\label{eq:Fund2}
0 \to N (-D) \to \Ff \to L \oplus \Oc_{D} \to 0;
\end{equation} in particular if $D =0$, i.e. 
$\widetilde{\Gamma} = \Gamma$ is a section of $F$, then $\Ff$ fits in 
\begin{equation}\label{eq:Fund}
0 \to N \to \Ff \to L \to 0
\end{equation} and the {\em normal bundle} of $\Gamma$ in $F$ is: 
\begin{equation}\label{eq:Ciro410}
\N_{\Gamma/F} \simeq L \otimes N^{\vee}, \;\; {\rm so} \;\; \Gamma^2 = \deg(L) - \deg(N) = 2\delta - d  
\end{equation} (cf. e.g. \cite[Rem.\,2.1]{CF}). Accordingly, ${\rm Div}_F^{1,\delta}$ is endowed with the structure of Quot scheme; namely (cf. \cite[\S\,4.4]{Ser}) one has 
\begin{equation}\label{eq:isom1}
{\rm Div}_F^{1,\delta}  \stackrel{\simeq}{\longrightarrow} {\rm Quot}^C_{\Ff,\delta+t-g+1}, \;\;\;\; \widetilde{\Gamma} \to \left\{\Ff \to\!\!\to L \oplus \Oc_D \right\}.
\end{equation}As such, one has  
$$H^0(\N_{\widetilde{\Gamma}/F}) \simeq T_{[\widetilde{\Gamma}]} ({\rm Div}_F^{1,\delta}) \simeq {\rm Hom} (N (-D), L \oplus \Oc_D) \;\; {\rm and} \;\; H^1(\N_{\widetilde{\Gamma}/F})  \simeq {\rm Ext}^1 (N (-D), L \oplus \Oc_D).$$Finally, if $\widetilde{\Gamma} \sim H - f_D$ on $F$, for some effective divisor $D$ on $C$, then one has
\begin{equation}\label{eq:isom2}
|\Oc_F(\widetilde{\Gamma})| \simeq \Pp(H^0(\Ff (-D))).  
\end{equation}
\begin{definition}(cf. \cite[Def.\;2.7]{CF})\label{def:ass0} A unisecant $\widetilde{\Gamma} \in {\rm Div}^{1,\delta}_F$ is said to be:  

\noindent
(a) {\em linearly isolated (li)}  if $\dim(|\Oc_F(\widetilde{\Gamma} )|) =0$; 

\noindent
(b) {\em algebraically isolated (ai)}  if  $\dim({\rm Div}^{1,\delta}_F) =0$.  

\end{definition}For a unisecant $\widetilde{\Gamma}$ as in \eqref{eq:Fund2}, by \eqref{eq:isom2}, one has 
$\widetilde{\Gamma} \in |\Oc_F(1) \otimes \rho^*(N^{\vee}(D))|$ so, applying $\rho_*$ and Leray's isomorphism, one gets
\begin{equation}\label{eq:iLa}
i(\widetilde{\Gamma}) = h^1(L \oplus \Oc_D) = h^1(L) = i(\Gamma), 
\end{equation}where $\Gamma$ denotes the unique section contained in  $\widetilde{\Gamma}$. 

In general, speciality is not constant either in linear systems or in  algebraic families of unisecants (cf. e.g. \cite[Examples\,2.8,\,2.9]{CF}); nevertheless, since ${\rm Div}^{1,\delta}_F$ has a Quot-scheme structure there is a universal quotient $\mathcal Q_{1,\delta} \to {\rm Div}^{1,\delta}_F$ so, taking  $\mathbb P (\mathcal Q_{1,\delta}):= {\rm Proj} (\mathcal Q_{1,\delta}) \stackrel{p}{\to} {\rm Div}^{1,\delta}_F$, one can consider 
\begin{equation}\label{eq:aga}
\mathcal S_F^{1,\delta} := \{\widetilde{\Gamma} \in {\rm Div}^{1,\delta}_F \;\; | \;\; R^1p_*(\Oc_{\Pp(\mathcal Q_{1,\delta})}(1))_{\widetilde{\Gamma}} \neq 0\}, 
\end{equation} i.e.  $\mathcal S_F^{1,\delta}$ is the support of $R^1p_*(\Oc_{\Pp(\mathcal Q_{1,\delta})}(1))$, which parametrizes degree-$\delta$, special unisecants of $F$. In this set-up, we recall the following results which will be sometimes used in the sequel.

\begin{proposition}\label{prop:CFliasi} (cf. \cite[Lemma\,2.11, Prop.\;2.12]{CF}) (1) Let $\Gamma \subset F$ be a section corresponding to the exact sequence \eqref{eq:Fund}. A section $\Gamma'$, corresponding to a quotient $\Ff \to\!\!\!\!\! \to L'$, is such that  
$\Gamma \sim \Gamma'$ on $F$ if and only if $L \simeq L'$. In particular $i(\Gamma) = i(\Gamma')$.

\smallskip

\noindent
(2) Assume $\Ff$ to be non-splitting and let $\Gamma \in \mathfrak{F} \subseteq \mathcal{S}^{1,\delta}_F$ be a section of $F$, where  $\mathfrak{F}$ is an irreducible, projective scheme of dimension $k$, parametrizing a flat family of unisecants. Assume that:  

\noindent
(2-i) either $k \geq 1$, if $\mathfrak{F}$ is a linear system; 

\noindent
(2-ii) otherwise, if $\mathfrak{F}$ is not a linear family, assume either $k \geq 2$ or  $k=1$ and $\mathfrak{F}$ with base points.

\smallskip

\noindent
Then, $\mathfrak{F}$ contains reducible unisecants $ \widetilde{\Gamma} \subset F$ such that $i(\widetilde{\Gamma}) \geq i(\Gamma)$.
\end{proposition}

\subsection{Moduli spaces of rank-two (semi)stable bundles on curves and Brill-Nother loci}\label{ss:modspaces} Recall that a rank-two vector bundle $\mathcal{F}$ on a smooth curve $C$ is {\em stable} (resp. {\em semistable}) if for every sub-line bundle $N \subset \mathcal{F}$, its {\em slope} satisfies $\mu(N) < \mu(\mathcal{F})$ (resp. $\mu(N) \leq \mu(\mathcal{F})$), where $\mu(\mathcal{E}) := \frac{\deg(\mathcal{E})}{\text{rk}(\mathcal{E})}$. 

The moduli space $U_C(d)$, parametrizing (S-equivalence classes of) semistable bundles of rank two and degree $d$, is an irreducible variety of dimension $4g-3$. Within this space, $U_C^s(d)$ denotes the open, dense subset of isomorphism classes of stable bundles whereas the boundary $U_C(d) \setminus U_C^s(d)$, if not empty, contains strictly semistable bundles  (cf. e.g. \cite{Ram,Ses}).

\begin{proposition}\label{prop:sstabh1} (cf. e.g. \cite[Prop.\,3.1]{CF}) Let $C$ be a smooth, irreducible, projective curve of  genus $g \geq 1$ and let $d$ be an integer.

\noindent
$(i)$ If $d \geq 4g-3$, then for any $[\Ff] \in U_C(d)$, one has $i(\Ff) = h^1(\Ff) = 0$.

\noindent 
$(ii)$ If $g \geq 2$ and $d \geq 2g-2$, for $[\Ff] \in U_C(d)$ general,  one has $i(\Ff) =h^1(\Ff) = 0$.  
\end{proposition} From Proposition \ref{prop:sstabh1}, Serre duality and invariance of (semi)stability under operations like tensoring with a line bundle or passing to the dual bundle, for $g \geq 2$ it makes sense to consider the proper sub-loci of $U_C(d)$ parametrizing classes  $[\Ff] \in U_C(d)$ such that  $i(\Ff) > 0$ in the following range for $d$:    
\begin{equation}\label{eq:congd}
2g-2 \leq d \leq 4g-4. 
\end{equation}

\begin{definition}\label{def:BNloci} Given $C$ any smooth, irreducible, projective curve of genus $g \geq 2$ and considered  $d$ and $ i $ two non-negative integers, we set $k_i := d - 2 g + 2 + i$ and define 
\begin{equation}\label{eq:BdKi}
B^{k_i}_d := \left\{ [\Ff] \in U_C(d) \, | \,  h^0(\Ff) \geq k_i \right\} = 
\left\{ [\Ff] \in U_C(d) \, | \, h^1(\Ff) = i(\Ff) \geq i \right\},
\end{equation} the $k_i^{th}$--{\em Brill-Noether locus}, 
simply {\em Brill-Noether locus}, when no confusion arises.
\end{definition}

\begin{remark}\label{rem:BNloci}  {\normalfont Brill-Noether loci $B^{k_i}_d$ have a natural structure of (locally) determinantal subschemes of $U_C(d)$: 

\medskip

\noindent
$(a)$ When $d$ is odd, $U_C(d)=U^s_C(d)$, then $U_C(d)$  is a fine moduli space and the existence of a universal bundle on $C \times U_C(d)$  allows one to construct $B^{k_i}_d$ as the degeneracy locus of a morphism between suitable  vector bundles on $U_C(d)$ (see, e.g. \cite{GT} for details). Accordingly, the  {\em expected dimension} of $B^{k_i}_d$ is 
${\rm max} \{ -1, \ \rho_d^{k_i}\}$, where 
\begin{equation}\label{eq:bn}
\rho_d^{k_i}:= 4g - 3 - i k_i
\end{equation} is  the rank-$2$ \emph{Brill-Noether number} as in the Introduction.  If  $\emptyset \neq B^{k_i}_d \subsetneq U_C(d)$, 
then $B^{k_i+1}_d \subseteq {\rm Sing}(B^{k_i}_d)$. 

Since any $[\Ff] \in B^{k_i}_d$ corresponds to a stable bundle, it is a smooth point of $U_C(d)$ and the Zariski tangent space $T_{[\Ff]}(U_C(d))$ at $[\Ff]$ to $U_C(d)$ identifies to $H^0( \omega_C \otimes \Ff \otimes \Ff^{\vee})^{\vee}$. If $[\Ff] \in B^{k_i}_d \setminus B^{k_i+1}_d$, 
 the tangent space to $B^{k_i}_d$ at $[\Ff]$ is otherwise the  {\em annihilator} of the image  of the {\em Petri map} of $\Ff$ (see, e.g. \cite{Teixidor2}) 
\begin{equation}\label{eq:petrimap}
\mu_{\Ff} : H^0(C, \Ff) \otimes H^0(C, \omega_C \otimes \Ff^{\vee})
\longrightarrow H^0(C, \omega_C \otimes \Ff \otimes \Ff^{\vee}).
\end{equation} 
If $[\Ff] \in B^{k_i}_d \setminus B^{k_i+1}_d$, then 
$\rho_d^{k_i} = h^1(C, \Ff \otimes \Ff^{\vee}) - h^0(C, \Ff) h^1(C, \Ff)$ moreover $B^{k_i}_d$ is non--singular and of the expected dimension at $[\Ff]$ if and only if the Petri map $\mu_{\Ff}$ is injective.

\medskip 

\noindent
$(b)$ When otherwise $d$ is even, $U_C(d)$ is not a fine moduli space (because $U^{ss}_C(d)\neq \emptyset$). There is a suitable open, non-empty subscheme $ \mathcal Q^{ss} \subset \mathcal Q$ of a certain Quot scheme $\mathcal Q$ defining $U_C(d)$ via a GIT-quotient map $\pi$ (cf. e.g. \cite{Tha} for details) and one can define $B^{k_i}_C(d)$ as the image via $\pi$ of the degeneracy locus of a morphism between suitable vector bundles on $\mathcal Q^{ss}$. It may happen for a component   $\mathcal B $ of a Brill--Noether locus to be  totally contained in $U_C^{ss}(d)$; 
in this case the lower bound  $\rho_d^{k_i}$ for the expected dimension of $\mathcal B$ is no longer valid  (cf. e.g. 
\cite[Remark 7.4]{BGN}).  Nonetheless, the lower bound $ \rho_d^{k_i}$ is still valid if $\mathcal B\cap U^s_C(d)\neq \emptyset$. } 
\end{remark} Recall, finally, the following terminology.

\begin{definition}\label{def:regsup} Assume $B_d^{k_i} \neq \emptyset$. A component $ \mathcal B \subseteq B_d^{k_i}$ 
such that $\mathcal B \cap U_C^s(d) \neq \emptyset$ will be called {\em regular},  if it is generically smooth and of 
(expected) dimension $\dim(\mathcal B) = \rho_d^{k_i}$, 
{\em superabundant}, otherwise. 
\end{definition}

\subsection{Line bundle extensions and (semi)stability via Segre invariant}\label{ss:extsegre} For a rank-two vector bundle $\mathcal{F}$ of degree $d$ on a curve $C$, the {\em Segre invariant} $s(\mathcal{F})$ is:$$s(\mathcal{F}) := \min_{N \subset \mathcal{F}} \{ \deg(\mathcal{F}) - 2 \deg(N) \},$$where $N$ ranges over all sub-line bundles of $\Ff$. This invariant is stable under twisting by line bundles and by dualization. A sub-line bundle $N$ computing $s(\Ff)$ is a {\em maximal degree} sub-line bundle of $\Ff$, consequently its corresponding quotient $L = \mathcal{F}/N$ is a {\em minimal degree} quotient line bundle of $\Ff$. In terms of the surface scroll $F := \mathbb{P}(\mathcal{F})$ as in Sect. \ref{ss:unisec}, a section $\Gamma \in \text{Div}^{1,\delta}_F$ of minimal degree $\delta$ satisfies $\Gamma^2 = s(\mathcal{F}) = 2\delta - d$; hence, geometrically stability of $\mathcal{F}$ is entirely determined by the Segre invariant, indeed $\mathcal{F}$ is semistable iff $s(\mathcal{F}) \geq 0$, whereas it is stable iff $s(\mathcal{F}) > 0$.

Let now $\delta \leq d$ be a positive integer; consider $L\in \pic^\delta(C)$ and $N\in\pic^{d-\delta}(C)$.  Any $u\in\ext^1(L,N)$ gives rise to a degree-$d$, rank-$2$ vector bundle $\mathcal F_u$, fitting in: 
\[
(u):\;\; 0 \to N \to \mathcal F_u \to L \to 0.
\] If $\Gamma$ is the section on the associated surface scroll $F_u = \mathbb P(\Ff_u)$ corresponding to the surjection $\mathcal F_u \to\!\!\to L$ then, since $\Gamma^2 = \deg(L-N)$, a necessary condition  for $\mathcal F_u$ to be semistable is therefore 
\begin{equation}\label{eq:neccond}
\Gamma^2= 2\delta-d \ge s(\mathcal F_u)\ge 0.
\end{equation} In such a case, Riemann-Roch theorem gives 
\begin{equation}\label{eq:m}
\dim (\ext^1(L,N))=
\begin{cases}
\ 2\delta-d+g-1\ &\text{ if } L\ncong N \\
\ g\ &\text{ if } L\simeq  N.
\end{cases}
\end{equation}

\begin{proposition} \label{thm:barj} (cf. \cite[Prop.\;1]{Barj}) Let $C$ be a smooth, irreducible, projective curve of genus $g \geq 3$. Let $\Ff$ be a rank-$2$ vector bundle on $C$ of degree $d\geq 2g-2$, with speciality $i = h^1(\Ff) \geq 2$. Then $\Ff = \Ff_u$ for $u \in {\rm Ext^1}(\omega_C(-D), \, N)$, namely $\Ff$ arises as an extension of the form
\begin{equation}\label{eq:barj-1}
0 \to N \to \Ff= \Ff_u \to \omega_C(-D) \to 0,
\end{equation}where $D$ is an effective divisor on $C$ such that either $h^1(\omega_C(-D))=1$ or $h^1(\omega_C(-D)) = i$. 
\end{proposition}
\begin{proof} The first part of the statement is the ``Serre-dual" version of \cite[Lemma\,1]{Barj} or even of \cite[Cor.\,1.2]{Teixidor2}, cited in the latter 
as a direct consequence on an unpublished result of B. Feinberg; this result has been completely proved in \cite[Prop.\;1]{Barj}. The last part of the statement is a direct consequence of (semi)stability and \eqref{eq:m}. 
\end{proof}

\begin{definition}\label{def:fstype}
A rank-$2$ vector bundle $\Ff$ with speciality $i = h^1(\Ff) \geq 2$, fitting in \eqref{eq:barj-1} and which satisfies $h^1(\omega_C(-D)) = i$ is said to be {\em a bundle of first type}; otherwise if $h^1(\omega_C(-D)) =1$ then $\Ff$ is said to be {\em a bundle of second type} (cf. \cite[Definition\;1]{Barj}). 
\end{definition} 

We will construct irreducible components  $\mathcal B$ of $B_d^{k_3} \cap U^s_C(d)$, with the use of the above subdivision in bundles of {\em first} or {\em second type} for the general point $[\Ff]$ of any such a component $\mathcal B$, or even by performing some natural {\em modifications}  of these two types of presentations, when such a variation let the construction of a prospected irreducible component $\mathcal B$ to be easier (cf. e.g. Proposition \ref{prop:case2b} below). To this aim, in order to characterize the stability of rank-two bundles constructed as extensions, we may use geometric criteria established by Lange and Narasimhan in \cite{LN} as reminded below. 

Consider a bundle $\mathcal{F}_u$ defined by an extension:
\begin{equation}\label{eq:ext-rep}
0 \to N \to \mathcal{F}_u \to L \to 0,
\end{equation} where $u \in H^1(C, N \otimes L^\vee) \setminus \{0\}$. Let $\mathbb{P} := \mathbb{P}(H^1(C, N \otimes L^\vee))$. Assuming the linear system $|K_C \otimes L \otimes N^\vee|$ induces a morphism $\varphi: C \to \mathbb{P}$, we let $X = \varphi(C)$ denote the image curve. The stability of $\mathcal{F}_u$ will be governed by the position of the extension class $[u]$ relative to the {\em secant varieties} of $X$. 

Precisely, recall that the {\em $h$-th secant variety} of $X$, $\text{Sec}_h(X) \subseteq \mathbb{P}$, is defined to be the closure of the union of linear spans $\langle \varphi(D) \rangle$, for all degree-$h$ effective divisors $D \in C^{(h)}$. In this set-up, we have: 

\begin{proposition} (\cite[Proposition 1.1]{LN})\label{LN} Assume $2\delta-d\ge 2$.  Then $\varphi=\varphi_{|K_C+L-N|}:C \to \mathbb{P} $ is a morphism. Moreover, for any integer
$\sigma \equiv 2\delta-d\  \text{ (mod\ 2) }$ such that $4+ d-2\delta\le \sigma \le 2\delta-d$, one has $s (\mathcal F_u)\ge \sigma$ if and only if $[u]\notin \Sec_{\frac{1}{2}(2\delta-d+\sigma-2)}(X)$, where $X=\varphi(C)\subset \mathbb P$.
\end{proposition} 

\noindent
Consequently, $\mathcal{F}_u$ is stable if and only if $[u]$ lies outside specific secant varieties that would violate the condition $s(\mathcal{F}_u) > 0$.

\section{Auxiliary initial results for the proof of {\bf Main Theorem}}\label{S:BN23} 
In this section we introduce preliminaries which are needed for the proof of {\bf Main Theorem}.  From now on, unless otherwise stated, we will therefore consider $C$ to be a {\em general $\nu$-gonal curve} of genus $g \geq 4$ and, as in \eqref{eq:nu} and \eqref{eq:congd}, we will set $$3 \leq \nu < \lfloor \frac{g+3}{2}\rfloor\;\;\; {\rm and} \;\;\; 2g-2 \leq d \leq 4g-4.$$
To prove {\bf Main Theorem} recall the following:

\begin{lemma}\label{specialityhigh} (cf.\,e.g.\,\cite[Lemme\,2.6 and pp.101--102]{L}) Let $C$ be any smooth, irreducible, projective curve of genus $g$. For any integer $2g-2 \leq d \leq 4g-4$ and $i >0$ there is no irreducible component of $B_d^{k_i} \cap U^s_C(d)$ whose general member $[\mathcal F]$ corresponds to  rank-$2$ vector bundle $\Ff$ of speciality $j := h^1(\mathcal F) > i$. 
\end{lemma}
\begin{proof} The proof is identical to that in \cite[pp.101-102]{L} for $B^0_d$, $1 \leq d \leq 2g-2$, which uses {\em elementary transformations} of vector bundles. 
\end{proof}

For a general point $[\mathcal{F}]$ of an irreducible component $\mathcal{B} \subseteq B_d^{k_3} \cap U^s_C(d)$, the bundle must be stable with degree $d$ and speciality $i(\mathcal{F}) = h^1(C, \Ff) =3$. By Proposition \ref{thm:barj}, such bundles arise as non-split extensions:
\begin{equation}\label{eq:ext-rep-final}
(u)\:\: 0 \to N \to \mathcal{F}_u \to L \to 0,
\end{equation}where $L \cong \omega_C(-D)$ is an effective, special line bundle of degree $\delta > \frac{d}{2}$. Following Definition \ref{def:fstype}, we classify $\mathcal{F}$ based on the speciality of $L$: {\em first type} if $h^1(L) = 3$, {\em second type} if $h^1(L) = 1$. 

These categories allow us to identify potential obstructions to the existence of components (cf. Propositions \ref{lem:youngooksup3}, \ref{prop:nocases}, \ref{prop:17mar12.52}, \ref{prop:fanc19-12}). Otherwise, in the absence of obstructions, we employ a parametric approach to construct these components and analyze their infinitesimal structure—specifically, the kernel of the associated Petri map $\mu_{\mathcal{F}}$. While the primary classification focuses on specialities $1$ and $3$, we occasionally utilize line bundles $L$ with $h^1(L)=2$. This {\em modified type} approach (cf. Proposition \ref{prop:case2b}) simplifies the construction of certain loci and the study of their local geometry. 

Our analysis integrates the study of degeneracy loci in $\mathrm{Ext}^1(L, N)$ with the Brill–Noether theory of rank-one loci $W^r_t(C)$ for $r=1, 2$ on general $\nu$-gonal curves. Precisely, for any exact sequence $(u)$ as in \eqref{eq:ext-rep-final},  if one sets $H^0(L) \stackrel{\partial_u}{\longrightarrow} H^1(N)$ the associated {\em coboundary map} then, for any integer $t>0$, one can define the {\em degeneracy locus}: 
\begin{equation}\label{W1}
\mathcal W_t:=\{u\in\ext^1(L,N)\ |\ {\rm corank} (\partial_u)\ge t\}\subseteq \ext^1(L,N), 
\end{equation} which has a natural structure of (locally) determinantal scheme and, as such, it has an {\em expected codimension}, which is 
\begin{equation}\label{eq:clrt}
c(t):= {\rm max} \{0,\; t(h^0(L) -h^1(N) +t)\}
\end{equation} (cf. e.g. \cite[\S\,5.2]{CF} for details). In this set--up, one has:

\begin{theorem}(\cite[Thm.\,5.8, Cor.\,5.9]{CF})\label{CF5.8} Let $C$ be a smooth, irreducible, projective curve of genus $g\ge 3$. 
Let $0 \leq \delta \leq d $ be integers and let $L \in \pic^{\delta}(C)$ be special and effective and $N \in \pic^{d-\delta}(C)$. 
Set $l := h^0(L)$, $r:= h^1(N)$, $m:=\dim(\ext^1(L,N))$ and assume $ r\ge 1,\,l\ge \max\{1,r-1\},\, m \ge l+1$. Then:

\smallskip

\noindent
$(i)$ $\mathcal W_1$ as in \eqref{W1} is irreducible of the {\em expected dimension} 
$m-c(1) = m- (l-r+1)$; 

\smallskip
\noindent
$(ii)$ if $l \geq r$, then $\mathcal W_1 \subsetneq \ext^1(L,N)$ and, for general $u \in  \ext^1(L,N)$,  the associated coboundary map $\partial_u$ is surjective whereas, for general $w \in  \mathcal W_1$, one has  ${\rm corank} (\partial_w)=1$. 
\end{theorem}

To study instead rank-one Brill-Noether loci of interests on $C$ a general $\nu$-gonal curve, we will use Theorem \ref{thm:Lar2} to prove the next results.

\begin{lemma}\label{lem:Lar} Let $g \geq 4$ and $3 \leq \nu < \lfloor \frac{g+3}{2}\rfloor$ be integers. Let $C$ be a general $\nu$-gonal curve of genus $g$ and let $A \in {\rm Pic}^{\nu}(C)$ be the unique line bundle on $C$ such that $|A| = g^1_{\nu}$. 
Let $t \geq \nu$ be an integer and, for a given integer $0 \leq \ell \leq 1$, let 
$\vec{w}_{1,\ell}$ be the corresponding {\em maximal splitting type vector} and 
$\mathcal O_{\Pp^1}(\vec{w}_{1,\ell})$ be the associated {\em balanced-plus-balanced bundle} on $\Pp^1$ as in Theorem \ref{thm:Lar2}. Then one has: 
\begin{equation*}
\mathcal O_{\Pp^1}(\vec{w}_{1,0}) = B(\nu-2, \,t+1-g-\nu) \oplus B(2,\,0)\;\; {\rm and }  \;\;
 \mathcal O_{\Pp^1}(\vec{w}_{1,1}) = B(\nu-1, \,t+1-g-\nu) \oplus B(1,\,1),
\end{equation*} 
\normalsize
where 
$B(-,-)$ denotes the most balanced bundle on $\Pp^1$ of given $({\rm rank},\;{\rm degree})$, therefore $B(2,0) = \mathcal O_{\Pp^1}^{\oplus 2}$, $B(1,1) = \mathcal O_{\Pp^1}(1)$ whereas the other summands are bundles of negative degrees. Moreover, 
\[W^1_t(C) = 
\begin{cases}
\overline{W^{\vec{w}_{1,1}}} & \mbox{ if $\nu\leq t< \frac{g+2}{2}$}\\
\overline{W^{\vec{w}_{1,1}}} \ \cup \ \overline{W^{\vec{w}_{1,0}}} &  \mbox{ if $\frac{g+2}{2}\leq t\leq g+2-\nu$ } \\
\overline{W^{\vec{w}_{1,0}}} &  \mbox{ if $g-\nu+2< t < g+1$}\\
{\rm Pic}^t(C) & \mbox{if $t \geq g+1$}
\end{cases}
\]and, when $W^1_t(C) \subsetneq {\rm Pic}^t(C) $, any of its irreducible components is generically smooth of dimension
$$\dim(\overline{W^{\vec{w}_{1,1}}} ) = t-\nu \;\; {\rm and} \;\; \dim(\overline{W^{\vec{w}_{1,0}}} ) = \rho(g,1,t) = 2t-g-2.$$Furthermore:

\medskip

\noindent
$(i)$  a general element in $\overline{W^{\vec{w}_{1,1}}}$ is given by $\mathcal L_{1,1}:= A \otimes \mathcal O_C(B_{t-\nu})$, where $|A| = g^1_{\nu}$ and $B_{t-\nu} = p_1+ \cdots + p_{t-\nu}$ the effective divisor of base points of $|\mathcal L_{1,1}|$; 

\medskip

\noindent
$(ii)$ a general element in $\overline{W^{\vec{w}_{1,0}}}$ is given by $\mathcal L_{1,0}$, where $|\mathcal L_{1,0}|$ is a base-point-free, complete $g^1_t$.
\end{lemma}

\begin{proof}
If $t \geq g+1$, then $W^1_t(C) = \text{Pic}^t(C)$ by Riemann-Roch. For $t < g+1$, we identify the maximal splitting type vectors $\vec{w}_{1,\ell}$ using Theorem \ref{thm:Lar2}. The admissible range for $\ell$ is $0 \leq \ell \leq 1$. Specifically, if $t \leq g+2-\nu$, both $\ell=0$ and $\ell=1$ occur whereas, if $g+2-\nu < t < g+1$, only $\ell=0$ occurs. Setting $t' := t+1-g-\nu$, we analyze the resulting components:

\smallskip

\noindent
(i) For $(r,\ell)=(1,1)$, we have $\mathcal{O}_{\mathbb{P}^1}(\vec{w}_{1,1}) = B(\nu-1, t'-1) \oplus \mathcal{O}_{\mathbb{P}^1}(1)$. The component exists if $\rho'(g, \vec{w}_{1,1}) = t-\nu \geq 0$. A dimension count shows that its general element is of the form $A \otimes \mathcal{O}_C(p_1 + \dots + p_{t-\nu})$, where $|A|=g^1_\nu$.

\smallskip

\noindent
(ii) For $(r,\ell)=(1,0)$, we have $\mathcal{O}_{\mathbb{P}^1}(\vec{w}_{1,0}) = B(\nu-2, t') \oplus \mathcal{O}_{\mathbb{P}^1}^{\oplus 2}$. The component has dimension $\rho'(g, \vec{w}_{1,0}) = \rho(g,1,t) = 2t-g-2$. To show that the general element $\mathcal{L}_{1,0} \in \overline{W^{\vec{w}_{1,0}}}$ is base-point-free, suppose otherwise. Then $\mathcal{L}_{1,0}(-p) \in W^1_{t-1}(C)$ would share the same splitting type $B(\nu-2, t'-1) \oplus \mathcal{O}_{\mathbb{P}^1}^{\oplus 2}$. This implies $\dim(\overline{W^{\vec{w}_{1,0}}}) \leq \dim(\overline{W^{\vec{w}_{1,0}}_t-1}) + 1$, which contradicts the fact that $\rho(g,1,t) = \rho(g,1,t-1) + 2$. Thus, $\mathcal{L}_{1,0}$ is base-point-free.
\end{proof}

\begin{lemma}\label{lem:w2tLar} Let $g \geq 4$ and $3 \leq \nu < \lfloor \frac{g+3}{2}\rfloor$ be integers. Let $C$ be a general $\nu$-gonal curve of genus $g$ and let $A \in {\rm Pic}^{\nu}(C)$ be the unique line bundle on $C$ such that $|A| = g^1_{\nu}$. 
Let $t \geq \nu$ be an integer and, for a given integer $0 \leq \ell \leq 2$, let 
$\vec{w}_{2,\ell}$ denote the corresponding {\em maximal splitting type vector} and 
$\mathcal O_{\Pp^1}(\vec{w}_{2,\ell})$ be the associated {\em balanced-plus-balanced bundle} on $\Pp^1$ as in Theorem \ref{thm:Lar2}. Then one has: 
\[
\begin{cases}
\mathcal O_{\Pp^1}(\vec{w}_{2,0}) = &  B(\nu-3, \,t+1-g-\nu) \oplus B(3,\,0),\\
\mathcal O_{\Pp^1}(\vec{w}_{2,1}) = &  B(\nu-2, \,t-g-\nu) \oplus B(2,\,1), \\
\mathcal O_{\Pp^1}(\vec{w}_{2,2}) = & B(\nu-1, \,t-1-g-\nu) \oplus B(1,\,2)
\end{cases}
\]where $B(-,-)$ denotes the most balanced bundle on $\Pp^1$ of given $({\rm rank},\;{\rm degree})$, therefore $B(3,0) = \mathcal O_{\Pp^1}^{\oplus 3}$, $B(2,1) = \mathcal \mathcal O_{\mathbb P^1}\oplus \mathcal O_{\Pp^1}(1)$ and $B(1,2) = \mathcal O_{\Pp^1}(2)$, whereas the other summands are bundles of negative degrees. Moreover: 

\smallskip

{\footnotesize

\noindent
$(i)$ for $\nu = 3$, one has respectively: 

\noindent
for $ g = 4$  
\[W^2_t(C) = 
\begin{cases}
\emptyset & \mbox{if $3 \leq t < \frac{g+5}{2}=\frac{9}{2}$} \\
\overline{W^{\vec{w}_{2,1}}} \ \cup \ \overline{W^{\vec{w}_{2,0}}} &  \mbox{if $\frac{2g+6}{3}=\frac{14}{3}\leq t=5< 6 = 2\nu$ } \\
{\rm Pic}^t(C) & \mbox{if $t \geq 6= g+2$}
\end{cases}
\]

\noindent
for $ g = 5$  
\[W^2_t(C) = 
\begin{cases}
\emptyset & \mbox{if $3 \leq t < \frac{g+5}{2}=\frac{10}{2}$} \\
\overline{W^{\vec{w}_{2,1}}}  &  \mbox{if $t=5=\frac{g+5}{2}<\frac{2g+6}{3}=\frac{16}{3}$ } \\
\overline{W^{\vec{w}_{2,2}}} \ \cup \ \overline{W^{\vec{w}_{2,1}}} \cup \overline{W^{\vec{w}_{2,0}}} &  \mbox{if $t=6 = 2\nu \ < 7= g+2$}\\
{\rm Pic}^t(C) & \mbox{if $t \geq 7 = g+2$}
\end{cases}
\]

\noindent
for $g=6$  
\[W^2_t(C) = 
\begin{cases}
\emptyset & \mbox{if $3 \leq t < \frac{g+5}{2}=\frac{11}{2}$} \\
\overline{W^{\vec{w}_{2,2}}} \cup \overline{W^{\vec{w}_{2,1}}} \ \cup \ \overline{W^{\vec{w}_{2,0}}} &  \mbox{if $2\nu= \frac{2g+6}{3}=6\leq t< 8=g+2$ } \\
{\rm Pic}^t(C) & \mbox{if $t \geq 8=g+2$}
\end{cases}
\]

\noindent
for $g = 7$, 
\[W^2_t(C) = 
\begin{cases}
\emptyset & \mbox{if $3 \leq t < 6 = 2\nu$} \\
\overline{W^{\vec{w}_{2,2}}} \cup \overline{W^{\vec{w}_{2,1}}} &  \mbox{if $t = 6 = 2\nu = \frac{g+5}{2} < \frac{2g+6}{3}=\frac{20}{3}$ } \\
\overline{W^{\vec{w}_{2,2}}} \cup \overline{W^{\vec{w}_{2,1}}} \ \cup \ \overline{W^{\vec{w}_{2,0}}} &  \mbox{if $\frac{2g+6}{3} < 7 \leq t< 9=g+2$ } \\
{\rm Pic}^t(C) & \mbox{if $t \geq 9= g+2$}
\end{cases}
\]

\noindent
whereas, for $g \geq 8$, 
\[W^2_t(C) = 
\begin{cases}
\emptyset & \mbox{if $3 \leq t < 6 = 2\nu$} \\
\overline{W^{\vec{w}_{2,2}}} & \mbox{if $6 = 2\nu \leq t< \frac{g+5}{2}$}\\
\overline{W^{\vec{w}_{2,2}}} \cup \overline{W^{\vec{w}_{2,1}}} &  \mbox{if $\frac{g+5}{2} \leq t< \frac{2g+6}{3}$ } \\
\overline{W^{\vec{w}_{2,2}}} \cup \overline{W^{\vec{w}_{2,1}}} \ \cup \ \overline{W^{\vec{w}_{2,0}}} &  \mbox{if $\frac{2g+6}{3} \leq t< g+2$ } \\
{\rm Pic}^t(C) & \mbox{if $t \geq g+2$}
\end{cases}
\]

\noindent
$(ii) $ for $4 \leq \nu \leq \frac{g+2}{3}$, one has: 

\[W^2_t(C) = 
\begin{cases}
\emptyset & \mbox{if $\nu \leq t < 2\nu$} \\
\overline{W^{\vec{w}_{2,2}}} & \mbox{if $2\nu\leq t< \frac{g+2+\nu}{2}$}\\
\overline{W^{\vec{w}_{2,2}}} \ \cup \ \overline{W^{\vec{w}_{2,1}}} &  \mbox{if $\frac{g+2+\nu}{2}\leq t< \frac{2g+6}{3}$ } \\
\overline{W^{\vec{w}_{2,2}}} \ \cup \ \overline{W^{\vec{w}_{2,1}}} \cup \overline{W^{\vec{w}_{2,0}}} &  \mbox{if $\frac{2g+6}{3} \leq t \leq g+3-\nu$}\\
\overline{W^{\vec{w}_{2,1}}} \ \cup \ \overline{W^{\vec{w}_{2,0}}} &  \mbox{if $t = g+4-\nu$ } \\
\overline{W^{\vec{w}_{2,0}}} & \mbox{if $g+5-\nu\leq t< g+2$}\\
{\rm Pic}^t(C) & \mbox{if $t \geq g+2$}
\end{cases}
\]

\noindent
$(iii) $ for $4 \leq \nu = \frac{g+3}{3}$, one has:

\[W^2_t(C) = 
\begin{cases}
\emptyset & \mbox{if $\nu \leq t < 2\nu=g+3-\nu$} \\
\overline{W^{\vec{w}_{2,2}}} \ \cup \ \overline{W^{\vec{w}_{2,1}}} \cup \overline{W^{\vec{w}_{2,0}}} &  \mbox{if $t=2\nu=g+3-\nu$}\\
\overline{W^{\vec{w}_{2,1}}} \ \cup \ \overline{W^{\vec{w}_{2,0}}} &  \mbox{if $t = g+4-\nu$ } \\
\overline{W^{\vec{w}_{2,0}}} & \mbox{if $g+5-\nu\leq t< g+2$}\\
{\rm Pic}^t(C) & \mbox{if $t \geq g+2$}
\end{cases}
\]

\noindent
$(iv)$ for $4 \leq \nu = \frac{g+4}{3}$, one has:

\[W^2_t(C) = 
\begin{cases}
\emptyset & \mbox{if $\nu \leq t < g+3-\nu$} \\
\overline{W^{\vec{w}_{2,1}}} & \mbox{if $ t=g+3-\nu= \frac{g+2+\nu}{2}$}\\
\overline{W^{\vec{w}_{2,1}}} \cup \overline{W^{\vec{w}_{2,0}}} &  \mbox{if $t= 2\nu=g+4-\nu$}\\
\overline{W^{\vec{w}_{2,0}}} & \mbox{if $g+5-\nu\leq t< g+2$}\\
{\rm Pic}^t(C) & \mbox{if $t \geq g+2$}
\end{cases}
\]

\noindent
$(v)$ for $4 \leq \frac{g+4}{3} < \nu \leq \frac{g+6}{3}$, one has: 

\[W^2_t(C) = 
\begin{cases}
\emptyset & \mbox{if $\nu \leq t \leq g+3-\nu$} \\
\overline{W^{\vec{w}_{2,1}}} \cup \overline{W^{\vec{w}_{2,0}}} &  \mbox{if $t= g+4-\nu$}\\
\overline{W^{\vec{w}_{2,0}}} & \mbox{if $g+5-\nu\leq t< g+2$}\\
{\rm Pic}^t(C) & \mbox{if $t \geq g+2$}
\end{cases}
\]

\noindent
$(vi)$ for $4 \leq \frac{g+6}{3} < \nu \leq \frac{g+9}{3}$, one has: 

\[W^2_t(C) = 
\begin{cases}
\emptyset & \mbox{if $\nu \leq t \leq g+4-\nu$} \\
\overline{W^{\vec{w}_{2,0}}} & \mbox{if $g+5-\nu\leq t< g+2$}\\
{\rm Pic}^t(C) & \mbox{if $t \geq g+2$}
\end{cases}
\]

\noindent
$(vii)$ for $4 \leq \frac{g+9}{3} < \nu < \lfloor \frac{g+3}{2}\rfloor$, one has: 

\[W^2_t(C) = 
\begin{cases}
\emptyset & \mbox{if $\nu \leq t < \frac{2g+6}{3}$} \\
\overline{W^{\vec{w}_{2,0}}} & \mbox{if $\frac{2g+6}{3} \leq t< g+2$}\\
{\rm Pic}^t(C) & \mbox{if $t \geq g+2$}
\end{cases}
\]} Furthermore, when $\emptyset \neq W^2_t(C) \subsetneq {\rm Pic}^t(C) $, any of its irreducible components as above is generically smooth of the following dimensions: 
$$\dim(\overline{W^{\vec{w}_{2,2}}} ) = t-2\nu,\;\; \dim(\overline{W^{\vec{w}_{2,1}}} ) = 2t-g-\nu-2\;{\rm and}\; \dim(\overline{W^{\vec{w}_{2,0}}} )  = 3t-2g-6,$$and 

\smallskip

\noindent
(a)  a general element in $\overline{W^{\vec{w}_{2,2}}}$ is given by $\mathcal L_{2,2}:= A^{\otimes 2} \otimes \mathcal O_C(B_{t-2\nu})$, where $|A| = g^1_{\nu}$ and $B_{t-2\nu} = p_1+ \cdots + p_{t-2\nu}$ an effective divisor of base points of $|\mathcal L_{2,2}|$; 

\smallskip

\noindent
(b)  a general element in $\overline{W^{\vec{w}_{2,1}}}$ is given by $\mathcal L_{2,1}:= A \otimes \mathcal O_C(M_{t-\nu})$, where $|A| = g^1_{\nu}$ and  $M_{t-\nu} = p_1+ \cdots + p_{t-\nu}\in C^{(t-\nu)}$ an effective divisor of movable points such that $|\mathcal L_{2,1}|$ is a base-point-free  
and  {\em birationally very ample} linear system $g^2_t$, that is,  the morphism associated to the $g^2_t$ is a birational map onto its image. 

\smallskip

\noindent
(c) a general element in $\overline{W^{\vec{w}_{2,0}}}$ is given by $\mathcal L_{2,0}$, where $|\mathcal L_{2,0}|$ is a base-point-free and  {\em birationally very ample} linear system $g^2_t$, that is,  its associated  morphism is a birational map onto its image.

\end{lemma}

\begin{proof} The proof of the first part of the statement is similar to that of Lemma \ref{lem:Lar}; it simply deals with taking into consideration conditions given by Theorem \ref{thm:Lar2} so the details can be left to the interested reader. 

For what concerns the description of the general element $\mathcal{L}_{2,\ell}$ in any irreducible component $\overline{W^{\vec{w}_{2,\ell}}}$, $0 \leq \ell \leq 2$, the reasoning is as follows. 

\smallskip

\noindent 
$(a)$ Case $\ell=2$: since $g \geq 2(\nu -1)$, Proposition \ref{Ballico} implies $h^0(A^{\otimes 2}) = 3$, thus $A^{\otimes 2} \in W^2_{2\nu}(C)$. A general element is of the form $A^{\otimes 2}(p_1 + \cdots + p_{t-2\nu})$. The result follows from the fact that $\dim(C^{(t-2\nu)}) = t-2\nu = \dim(\overline{W^{\vec{w}_{2,2}}})$, and the splitting type of its pushforward is maximal.

\smallskip

\noindent 
$(b)$ Case $\ell=1$: let $g^2_t = |\mathcal{L}_{2,1}|$. If $g^2_t$ had a base point $p$, then $g^2_t(-p)$ would belong to a component $\overline{W^{\vec{w}_{2,1}}}_* \subset W^2_{t-1}(C)$ with splitting type $B(\nu-2, t'-2) \oplus B(2,1)$. This would imply \linebreak $\dim(\overline{W^{\vec{w}_{2,1}}}) \leq \dim(\overline{W^{\vec{w}_{2,1}}}_*) + 1$, contradicting the actual dimension obtained by Lemma \ref{lem:Lar}. Thus, general $g^2_t$ is base-point-free. To prove its birational very-ampleness, let $\varphi: C \to \mathbb{P}^2$ be the induced map of degree $m$. If $m > 1$, the dimension of the corresponding sublocus in $\mathcal{M}_g$ is strictly less than the dimension of the $\nu$-gonal stratum, unless the geometric genus $\gamma$ of the image is $0$. However, $\gamma=0$ implies $g^2_t$ is composed with a $g^1_m$. For $t < g+2$, a general $\nu$-gonal curve only admits pencils composed with the unique $g^1_\nu$ (cf. \cite{AC}), forcing $g^2_t = 2g^1_\nu$, a contradiction to the generality of $\mathcal{L}_{2,1}$. Thus, $m=1$.

\smallskip

\noindent 
$(c)$ Case $\ell=0$: for $g^2_t = |\mathcal{L}_{2,0}|$, the base-point-free property follows from a dimension count: $\dim(\overline{W^{\vec{w}_{2,0}}}) = 3t-2g-6$, while $\dim(\overline{W^{\vec{w}_{2,0}}}_*) = 3t-2g-9$ which precludes existence of base points. Birational very ampleness follows by the same argument used in $(b)$.
\end{proof}

\section{Proof of {\bf Main Theorem}}\label{S:BN3} In this section, we construct the irreducible components of $B^{k_3}_d \cap U^s_C(d)$ for $2g-2 \leq d \leq 4g-4$ to establish the {\bf Main Theorem}. From \eqref{eq:bn}, the expected dimension of any such component is:$$\rho_d^{k_3} = 10g - 18 - 3d.$$By Lemma \ref{specialityhigh}, a general point $[\mathcal{F}]$ in any component $\mathcal{B}$ satisfies $h^1(\mathcal{F})=3$. By Proposition \ref{thm:barj}, we analyze stability-driven extensions \eqref{eq:barj-1} via effective quotient line bundles $L = \omega_C(-D)$ of degree $\delta > \frac{d}{2}$. 

We will identify existence obstructions (cf. Props. \ref{lem:youngooksup3}, \ref{prop:nocases}, \ref{prop:17mar12.52}, \ref{prop:fanc19-12}) by classifying $\mathcal{F}$ according to the speciality of $L$: {\em first type bundles} $\Ff$ for  $h^1(L) = 3$ and {\em second type bundles} when  $h^1(L) = 1$. In the absence of obstructions, we employ a parametric approach to construct these components and analyze their infinitesimal structure (specifically the kernel of the Petri map $\mu_{\mathcal{F}}$). To facilitate this construction and the study of local geometry, we occasionally use quotient line-bundles $L$ with $h^1(L)=2$, extending the primary classification of Definition \ref{def:fstype} (cf. Proposition \ref{prop:case2b}).

\subsection{No components for $4g-6 \leq d \leq 4g-4$} In the given range for $d$ one can very easily prove the following: 

\begin{proposition}\label{lemmaB} Let $g \geq 4$, $3 \leq \nu < \lfloor \frac{g+3}{2} \rfloor$ and $4g-6 \leq d \leq 4g-4$ be integers and let $C$ be a general $\nu$-gonal curve of genus $g$. Then $B_d^{k_3}  \cap U_C^s(d) = \emptyset$. 
\end{proposition} 
\begin{proof} By Clifford's Theorem (cf.\,e.g.,\,\cite[Theorem\,2.1]{BGN} or \cite[Theorem\,3]{Re}), any rank-$2$ (semi)stable bundle $\Ff$ of degree $0 \leq d \leq 4g-4$ is such that $h^0(\Ff) \leq \frac{d}{2}+2$. Thus, when $4g-5 \leq d \leq 4g-4$, if we assume by contradiction $B_d^{k_3}  \cap U_C^s(d) \neq \emptyset$ then the general member $[\Ff]$ of any irreducible component $\mathcal B \subseteq B_d^{k_3}  \cap U_C^s(d)$ has speciality exactly $i=3$, as from Lemma \ref{specialityhigh}. Therefore, $h^0(\Ff) = k_3 = d - 2 g +5$; nevertheless, by Clifford's Theorem above we should have $k_3 = d-2g+5 \leq \frac{d}{2} +2$ which would give $d \leq 4g-6$, against the assumptions. 

When otherwise $d =4g-6$, if once again by contradiction we had $B_d^{k_3}  \cap U_C^s(d) \neq \emptyset$, then for $[\Ff]$ general in any component we can apply  \cite[Propositions\,2,3,4]{Re}; in any case, we would obtain $h^0(\Ff) \leq \frac{d}{2} +1 = 2g-2$. However, as above, from Lemma \ref{specialityhigh} we would also get 
$h^0(\Ff) = k_3 = 2g-1$, once again a contradiction. 
\end{proof}

\subsection{Preliminary reductions for $2g-2 \leq d \leq 4g-7$}  Reducing to $2g-2 \leq d \leq 4g-7$, from Remark \ref{rem:BNloci}, any possible irreducible component $\mathcal B \subseteq B_d^{k_3} \cap U^s_C(d)$ must have $\dim(\mathcal B) \geq \rho_d^{k_3}= 10g-3d-18$ and, by Lemma \ref{specialityhigh}, a general point $[\Ff] \in \mathcal B$ must have $h^1(\Ff) =3$. Furthermore, by Proposition \ref{thm:barj}, $\Ff$ can be either of {\em first} or of {\em second type} in the sense of Definition \ref{def:fstype}. 

Putting all things together, we find below some necessary conditions for a component of $B^{k_3}_{d} \cap U^s_C(d)$ to exist as in the following result, where Lemmas \ref{lem:Lar} and \ref{lem:w2tLar} play a key role.

\begin{lemma}\label{lem:barjvero} Let $C$ be a general $\nu$-gonal curve of genus $g$ and let $\mathcal \Ff$ be a rank-$2$ stable bundle,  of speciality $i=h^1(\mathcal F) = 3$, degree 
$2g-2 \leq d \leq 4g-7$ and fitting in an exact sequence \eqref{eq:barj-1} where $\delta := \deg(\omega_C(-D))$ (i.e. $D \in C^{(2g-2-\delta)}$), with $\delta \leq 2g-2$, and $\deg(N) = d-\delta$.  

\smallskip

\noindent
(1) Assume $\Ff$ to be a bundle of {\em first type} (in the sense of Definition \ref{def:fstype}). Then one of the following occurs: 

\begin{itemize}
\item[(1-a)] $\delta \leq 2g-2 - 2\nu $, or
\item[(1-b)] $\delta  \leq \frac{3g-6-\nu}{2} $, or
\item[(1-c)] $\delta \leq \frac{4g-12}{3}  $.
\end{itemize}  
 
 \smallskip

\noindent
(2) If otherwise $\Ff$ is assumed to be of {\em second type}, then one of the following occurs:
\begin{itemize}
\item[(2-a)] $ 2g-2 \leq d \leq 3g-3$, or 
\item[(2-b)] $2g-2\leq d \leq 4g-4-\nu$, or
\item[(2-c)] $2g-2 \leq d \leq \frac{7g-6}{2}$.  
\end{itemize}
\end{lemma}

\begin{proof} Consider the isomorphism $B_d^{k_3} \cap U_C^s(d) \cong B^3_{d'} \cap U^s_C(d')$ induced by the {\em Serre duality map} \linebreak $\mathcal{F} \mapsto \mathcal{E} := \omega_C \otimes \mathcal{F}^\vee$, where $d' = 4g-4-d$. By Serre duality, $h^0(\mathcal{E}) = 3$, and $\mathcal{E}$ fits into the extension:
\begin{equation}\label{eq:barjvera}
0 \to \mathcal{O}_C(D) \to \mathcal{E} \to \mathcal{L} \to 0,
\end{equation}where $\mathcal{L} := \omega_C \otimes N^\vee$. Stability of $\mathcal{E}$ and the fact that $D \geq 0$ imply $0 \leq \deg(D) < d'/2$, $\delta' := \deg(\mathcal{L}) > \frac{d'}{2}$. These bounds translate to $g \leq \delta := \deg(\omega_C(-D)) \leq 2g-2$. We distinguish two cases:

\smallskip

\noindent 
(1) First Type ($h^0(\mathcal{O}_C(D)) = 3$): since $\deg(D) \leq g-2$, $\mathcal{O}_C(D)$ must belong to an irreducible component of $W^2_{\deg(D)}(C)$ as classified in Lemma \ref{lem:w2tLar}. The non-emptiness criteria for these components yield the existence conditions:
\begin{itemize}
\item[(1-a)] $\mathcal{O}_C(D) \in \overline{W^{\vec{w}_{2,2}}}$ implies $\deg(D) \geq 2\nu$; this occurs iff $2g-2-\delta \geq 2\nu$.
\item[(1-b)] $\mathcal{O}_C(D) \in \overline{W^{\vec{w}_{2,1}}}$ implies $2\deg(D) - g - \nu - 2 \geq 0$. This occurs iff $3g - 6 - 2\delta - \nu \geq 0$.
\item[(1-c)] $\mathcal{O}_C(D) \in \overline{W^{\vec{w}_{2,0}}}$ forces $3\deg(D) - 2g - 6 \geq 0$ which holds iff $4g - 12 - 3\delta \geq 0$.
\end{itemize}

\smallskip

\noindent 
(2) Second Type ($h^0(\mathcal{O}_C(D)) = 1$): for $h^0(\mathcal{E}) \geq 3$ to hold, we must have $h^0(\mathcal{L}) \geq 2$. Applying Lemma \ref{lem:Lar} to $\mathcal{L} \in \text{Pic}^{\delta'}(C)$:

\begin{itemize}
\item[(2-a)] If $\mathcal{L}$ is general, then $\delta' \geq g+1$, which implies $d \leq 3g-5$.
\item[(2-b)] If $\mathcal{L} \in \overline{W^{\vec{w}_{1,1}}}$, non-emptiness requires $\delta' \geq \nu$, thus $\nu \leq 4g-4-d$.
\item[(2-c)] If $\mathcal{L} \in \overline{W^{\vec{w}_{1,0}}}$, we require $\delta' \geq \frac{g+2}{2}$, implying $4g-4-d \geq \frac{g+2}{2}$.
\end{itemize} These cases provide the necessary conditions for the existence of components of $B_d^{k_3} \cap U^s_C(d)$.
\end{proof}

According to cases in Lemma \ref{lem:barjvero}, we set the following:

\begin{definition}\label{def:comp12mod} Let $\mathcal B \subseteq B^{k_3}_{d} \cap U^s_C(d)$ be an irreducible component. 

\medskip

\noindent
(1) Assume its general point $[\Ff] \in \mathcal B$ is a bundle of {\em first type} in the sense of Definition \ref{def:fstype}. Then, we will set $\mathcal B = \mathcal B_{\rm sup, 3, (1-x)}$ (respectively, $ \mathcal B =  \mathcal B_{\rm reg, 3, (1-x)}$) according to the fact that the component is {\em superabundant} (respectively, {\em regular}) and its general point $[\Ff]$ corresponds to a stable, rank-$2$, degree $d$ vector bundle of speciality $h^1(\Ff) =3$ arising as in case (1-x), for $x = a,b,c$ as in the proof of Lemma \ref{lem:barjvero}, whose cases arise from Lemma \ref{lem:w2tLar}; 

\medskip

\noindent
(2) Similarly, assume its general point $[\Ff] \in \mathcal B$ is a bundle of {\em second type}. Then, we will set $\mathcal B =  \mathcal B_{\rm sup, 3, (2-x)}$ (respectively, $ \mathcal B = \mathcal B_{\rm reg, 3, (2-x)}$) according to the fact that the component is {\em superabundant} (respectively, {\em regular}) and its general point $[\Ff]$ corresponds to a stable, rank-$2$, degree $d$ vector bundle of speciality $h^1(\Ff) =3$ arising as in case (2-x), for $x = a,b,c$ as in the proof of Lemma \ref{lem:barjvero}, whose cases arise from Lemma \ref{lem:Lar}; 
\medskip

\noindent
(3) If otherwise $\mathcal B$ has been constructed via {\em parametric approach} using line-bundle extensions obtained by suitable natural {\em modifications} of case (1-x) (respectively, (2-x)), for $x= a,b,c$ as above, then we will set 
 $\mathcal B =  \mathcal B_{\rm sup, 3, (j-x)-mod}$ (respectively, $ \mathcal B =  \mathcal B_{\rm reg, 3, (j-x)-mod}$)  according to the fact that the component is {\em superabundant} (respectively, {\em regular}) and its general point $[\Ff]$ corresponds to a stable, rank-$2$, degree $d$ vector bundle of speciality $h^1(\Ff) =3$ arising via 
 modification of case (j-x), for $j = 1, 2$ and $x = a,b,c$ as in the proof of Lemma \ref{lem:barjvero}. 
\end{definition} The following auxiliary result deals with rank-$2$ vector bundles of {\em second type}.

\begin{lemma}\label{reg_seq_1_3} Let $C$ be a general $\nu$-gonal curve of genus $g$ and let $\mathcal F$ be a rank-$2$, degree $d$ vector bundle of speciality $i = h^1(\Ff) = 3$, which is assumed to be of {\em second type} in the sense of Definition \ref{def:fstype}. Let $D_s \in C^{(s)}$ be an effective divisor of degree $s >0$ 
and $N \in {\rm Pic}^{d-2g+2+s}(C)$ such that $\Ff$ fits in  an exact sequence of the form \eqref{eq:barj-1}.

\smallskip

\noindent
(i)  if $\Ff$ is of type $(2-b)$  as in Lemma \ref{lem:barjvero}, assume that $N = K_C - (A + B_b)$, where $A$ is the unique line bundle on $C$ such that $|A| = g^1_{\nu}$, $B_b$ is an effective divisor of degree $b \geq 0$ such that $|A+ B_b| = g^1_{\nu+b}$, i.e. $B_b$ is the base locus for such a $g^1_{\nu+b}$, which is assumed to be such that $B_b \cap D_s = \emptyset$;  

\smallskip

\noindent
(ii) if $\Ff$ is of type $(2-c)$  as in Lemma \ref{lem:barjvero}, assume that  $N = K_C - g^1_t - B_b$, where $g^1_t$ is a base point free linear series of degree $t$ and dimension $1$, $B_b$ is an effective divisor of degree $b \geq 0$ such that $|g^1_t+ B_b| = g^1_{t+b}$, i.e. $B_b$ is the base locus for such a $g^1_{t+b}$, which is assumed to be such that $B_b \cap D_s = \emptyset$. 

\smallskip

Then $\Ff$ fits also in the exact sequence $0 \to N(-D_s) \to \Ff \to \omega_C \to 0$. 

\end{lemma}

\begin{proof}  Since $\Ff$ is assumed to be of speciality $3$, of {\em second type} and fitting in an exact sequence of the form \eqref{eq:barj-1}, one has $h^1(\omega_C(-D_s)) =1$, i.e.  $h^0(\omega_C(-D_s)) = g-s$. Moreover $h^1(N) \geq 2$ and 
$\Ff = \Ff_u$, for some $u \in \mathcal W_2 \subseteq  \ext^1(K_C-D_s, N)$, i.e. one has $$(u):\;\; 0 \to N \to \Ff=\Ff_u \to \omega_C(-D_s) \to 0.$$We may compose the line-bundle injection $N (-D_s) \hookrightarrow N$ to deduce that $N(-D_s)$ is also a sub-line bundle of $\mathcal F_u$; to prove the statement, we must show that this sub-line bundle is saturated in $\Ff$, i.e. that the zero-loci of the global sections of $\omega_C \otimes \mathcal{F}_u^\vee$ have empty intersection. This requires $h^0(\omega_C \otimes \mathcal{F}_u^\vee(-q)) = h^1(\mathcal{F}_u(q)) \leq 2$ for every $q \in C$. Tensoring the dualized extension by $\omega_C \otimes N^\vee \otimes \mathcal{O}_C(D_s)$ yields:$$ 0 \to \mathcal{O}_C(D_s) \xrightarrow{\sigma} \omega_C \otimes \mathcal{F}_u^\vee \to \omega_C \otimes N^\vee \to 0. $$The section $\tilde{\sigma}$ induced by $\sigma$ is non-vanishing for $q \notin D_s$. For $q \in D_s$, we tensor the original extension by $\mathcal{O}_C(q)$ to obtain:
\begin{equation}\label{eq:6gencogl2_rev}
0 \to N(q) \to \mathcal{F}_u(q) \to \omega_C - (D_s - q) \to 0.
\end{equation}In both cases (i) and (ii), the assumption $D_s \cap B_b = \emptyset$ ensures that $h^1(N(q)) = 1$ for any $q \in D_s$. Since $h^1(\omega_C - (D_s - q)) = 1$, the long exact sequence in cohomology associated to \eqref{eq:6gencogl2_rev} implies $h^1(\mathcal{F}_u(q)) \leq 1 + 1 = 2$. Thus, the intersection of zero-loci is empty, confirming that $N(-D_s)$ is saturated.
\end{proof}

\subsection{Some obstructions to fill-up components} As we will see below, Lemma \ref{lem:barjvero} provides, for some values of $d$, explicit obstructions for any irreducible component of $B^{k_3}_d \cap U^s_C(d)$ to exist. Precisely we have the following situation.

\medskip

\noindent
$\boxed{4g-4-2\nu \leq d \leq 4g-7: \; \mbox{no component}}$ In this range for $d$, one has: 

\begin{proposition}\label{ChoiA} Let $g \geq 15$, $3 \leq \nu < \frac{g-2}{4}$ and $4g-4-2\nu \leq d \leq 4g-7$ be integers and let $C$ be a general $\nu$-gonal curve of genus $g$. Then $B^{k_3}_d \cap U^s_C(d)=\emptyset$.
\end{proposition} 

\begin{proof} By the chosen bounds on $\nu$, any potential irreducible component $\mathcal{B} \subseteq B_d^{k_3} \cap U_C^s(d)$ must superabundant, as the expected dimension $\rho_d^{k_3} = 10g-18-3d$ is strictly negative. We proceed by contradiction, assuming $B_d^{k_3} \cap U_C^s(d) \neq \emptyset$.

In the range $4g-3-\nu \leq d \leq 4g-7$, cases $(2-a),\, (2-b)$, and $(2-c)$ of Lemma \ref{lem:barjvero} are numerically excluded. Consequently, $\mathcal{F}$ must be of first type. However, analysis of cases $(1-a)$, $(1-b)$, and $(1-c)$ reveals that the non-emptiness conditions for the required components of $W^2_{2g-2-\delta}(C)$ conflict with the stability requirement $d-\delta \leq 2g-4$ and the degree bounds on $d$ and $\nu$. Thus, the locus is empty.

Arguments are similar for the range $4g-4-2\nu \leq d \leq 4g-4-\nu$; first-type bundles and cases $(2-a)$ and $(2-c)$ are excluded by the bounds on $\nu$ and $d$. We then focus on case $(2-b)$, where $\mathcal{F}$ is of second type and $K_C-N \in \overline{W^{\vec{w}_{1,1}}}$. Thus, $\mathcal{F}$ fits into:
\begin{equation}\label{eq:newstr1}
0 \to K_C-(A+B_b) \to \mathcal{F} \to K_C-D_s \to 0,
\end{equation} with $|A|=g^1_\nu$. By Lemma \ref{reg_seq_1_3}-(i), $\mathcal{F}$ also admits the extension $0 \to K_C-(A+B_b+D_s) \to \mathcal{F} \to K_C \to 0$. This implies the existence of a linear pencil of sections $|\mathcal{I}_{q/F}(\Gamma)|$ on the surface scroll $F = \mathbb{P}(\mathcal{F})$. By Proposition \ref{prop:CFliasi}, this pencil contains a reducible unisecant, identifying a sub-line bundle $K_C - (A+B_b+D_s - E_f)$ with $f = \deg(E_f) \geq \nu$. However, stability of $\mathcal{F}$ in this degree range requires $d < 4g-4-2f$, which implies $f < \nu$. This contradiction excludes case $(2-b)$, completing the proof.
\end{proof}

\medskip

\noindent
$\boxed{2g-2 \leq d \leq 4g-5- 2\nu:\;\mbox{further obstructions and case exclusions}}$ From Propositions \ref{lemmaB} and \ref{ChoiA},  we can therefore limit to the range: 
\begin{equation}\label{eq:29dicnew}
2g-2 \leq d \leq 4g-5- 2\nu.
\end{equation}

\begin{remark}\label{rem:29dicnew} {\normalfont For $3 \leq \nu \leq \frac{g}{3}$, the relevant range of $d$ as defined in \eqref{eq:29dicnew} is subdivided into two distinct intervals based on the sign of the Brill-Noether number $\rho_d^{k_3} = 10g - 18 - 3d$. 

\smallskip

\noindent
{\em Superabundant Interval}: when $\rho_d^{k_3} < 0$ any existing irreducible component $\mathcal{B} \subseteq B_d^{k_3} \cap U^s_C(d)$ necessarily is {\em superabundant}, i.e. of dimension higher than the expected one. This occurs when $\frac{10}{3} g - 5 \leq d \leq 4g - 5 - 2\nu$, since $3 \leq \nu \leq \frac{g}{3}$;

\smallskip

\noindent
{\em Regular/Mixed Interval}: when otherwise $\rho_d^{k_3} \geq 0$, one may theoretically expect both {\em regular} components (where the dimension equals $\rho_d^{k_3}$) and {\em superabundant} ones; this occurs when $2g - 2 \leq d \leq \frac{10}{3} g - 6$. 

Therefore, the main subdivision that will be used below in our analysis is:  
\begin{equation}\label{eq:29dicnewbis}
\frac{10}{3} g -5\leq d \leq 4g-5- 2\nu, \; {\rm with} \; 3 \leq \nu \leq \frac{g}{3}, \;\;\; {\rm and} \;\; 2g-2 \leq d \leq \frac{10}{3} g - 6.
\end{equation} This serves as a primary framework for our analysis; the subsequent results identify further obstructions, effectively narrowing the candidate cases initially presented in Lemma \ref{lem:barjvero}.
}
\end{remark}

\medskip

\noindent
$\boxed{\mbox{Obstructions: Case $(1-c)$ can never give rise to a component}}$ Concerning bundles of type $(1-c)$ as in Lemma \ref{lem:barjvero}, one has the following result.

\begin{proposition}\label{lem:youngooksup3} Let $C$ be a general $\nu$-gonal curve of genus $g$. Let $\mathcal B \subset B^{k_3}_d  \cap U^s_C (d)$ be an irreducible locus, whose general element $[\mathcal F] \in \mathcal B$ is assumed to be a stable, rank-$2$, degree $d$ vector bundle $\Ff$ of speciality $h^1(\Ff) = 3$ and fitting in an exact sequence of the form
$$(u):\;\; 0\to N \to \mathcal F\to  K_C - E \to 0,$$where $E \in \overline{W^{\vec{w}_{2,0}}} \subseteq W^2_t(C)$ for some integer $t$, whereas $N\in\pic^{d-2g+2+t}(C)$. Then $\mathcal B$ cannot be an irreducible component of  $B^{k_3}_d  \cap U^s_C (d)$.

In particular, case $(1-c)$ as in Lemma \ref{lem:barjvero} can never give rise to an irreducible component of $B^{k_3}_d \cap U^s_C(d)$. 
\end{proposition}

\begin{proof}Since $h^1(K_C-E) = h^0(E) = 3 = h^1(\mathcal{F})$, the extension $(u)$ implies that either $N$ is non-special or the coboundary map $\partial_u: H^0(K_C-E) \to H^1(N)$ is surjective. From Lemma \ref{lem:w2tLar}, $E$ must belong to a component such as $\overline{W^{\vec{w}_{2,0}}}$, hence $t \geq \frac{2g}{3}+2$ and $\dim(\{E\}) \leq 3t-2g-6$. Stability of $\mathcal{F}$ requires $d < 4g-4-2t$. This bound implies $\deg(2K_C-N-E) > 2g-2$;  thus, by Riemann-Roch, $\dim(\text{Ext}^1(K_C-E, N)) = 5g-5-d-2t$. Summing the dimensions of the choices for $N$, $E$, and the projectivized extension space $\mathbb P( \text{Ext}^1(K_C-E, N)$ we obtain:$$ \dim(\mathcal{B}) \leq g + (3t-2g-6) + (5g-6-d-2t) = 4g-12-d+t. $$For $\mathcal{B}$ to be a component of $B^{k_3}_d \cap U^s_C(d)$, it must satisfy $\dim(\mathcal{B}) \geq \rho_d^{k_3} = 10g-3d-18$. This requirement leads to the inequality:$$ (4g-12-d+t) - (10g-3d-18) = 2d + 6 - 6g + t \geq 0 \implies d \geq 3g-3-\frac{t}{2}. $$However, combining this with the stability bound $d < 4g-4-2t$ yields:$$ 3g-3-\frac{t}{2} \leq d < 4g-4-2t \implies g-1 < \frac{3t}{2}. $$Substituting the non-emptiness condition $t \geq \frac{2g}{3}+2$ gives $(4g-4-2t) - (3g-3-\frac{t}{2}) = g-1-\frac{3t}{2} \leq -4 < 0$. This contradiction ensures that case $(1\text{-}c)$ of Lemma \ref{lem:barjvero} cannot yield a component.
\end{proof}

\medskip

\noindent
$\boxed{\mbox{Obstructions: further case exclusions for $d$ near the upper-bound in \eqref{eq:29dicnewbis}}}$ The following result shows that there are some extra case exclusions for a component to exist.

\begin{proposition}\label{prop:nocases} Let $C$ be a general $\nu$-gonal curve of genus $g$. One has the following: 

\noindent
$(i)$ For any $g \geq 9$, $ 3 \leq \nu \leq \frac{g}{3}$ and 
$\frac{10}{3} g -5\leq d \leq 4g-5-2\nu$, there is no component of $B_d^{k_3} \cap U_C^s(d)$ arising from extensions as in cases $(1-b)$ and $(2-a)$ in Lemma \ref{lem:barjvero}.

\medskip

\noindent
$(ii)$ For any $g \geq 17$, $ 3 \leq \nu \leq \frac{g-5}{4}$ and 
$\frac{7g-5}{2} \leq d \leq 4g-5-2\nu$, there is no component of $B_d^{k_3} \cap U_C^s(d)$ arising from extensions as in case $(2-c)$ in Lemma \ref{lem:barjvero}.

\medskip

\noindent
$(iii)$ For any integers $ 0 \leq s \leq \frac{g}{2}$, 
$3g-5 \leq d \leq \frac{7}{2} g-5 - s \leq \frac{7}{2} g-5$, there is no stable bundle $\Ff$ arising from an extension as in case $(2-c)$ in Lemma \ref{lem:barjvero} of the form 
\begin{equation}\label{eq:sonjagen25}
0 \to N := K_C-(g^1_t + B_b) \to \Ff \to K_C - D_s \to 0,
\end{equation} where $D_s \in C^{(s)}$ is general whereas $t+b= 4g-4-d-s$, $g^1_t$ is a base point free linear system and $B_b \geq 0$ (possibly $B_b= 0$) is an effective divisor of base points, i.e. $g^1_{t+b} = |g^1_t + B_b| = g^1_t + B_b$ (in other words  
$K_C-N = g^1_t + B_b \in  \overline{W^{\vec{w}_{1,0}}} \subseteq W^1_{4g-4-d-s}(C)$, as in Lemma \ref{lem:barjvero}). In particular, there is no component of $B_d^{k_3} \cap U_C^s(d)$ whose general point $[\Ff]$ is a bundle as in the previous description.
\end{proposition}

\begin{proof} $(i)$ and $(ii)$: Cases $(1-b)$, $(2-a)$, and $(2-c)$ of Lemma \ref{lem:barjvero} are numerically excluded in the range $d \geq \frac{10}{3}g-5$. For instance, in $(1-b)$, stability requires $d < 2\delta \leq 3g-6-\nu$, which is strictly less than the given lower bound for $d$.

\smallskip

\noindent
$(iii)$ To analyze the remaining cases, we employ the following:
\begin{claim}\label{Lemma1Sgen} Let $C$ be a smooth, irreducible, projective curve of genus $g \geq 4$. Let  $\Ff$ be a rank-$2$ vector bundle on $C$ of degree $d = \deg(\Ff) \geq 3g-5$ and of speciality $h^1(\Ff) = 3$. Assume that $\Ff$ fits in 
$(u):\;\;\; 0 \to N \to \Ff \to K_C- D_s \to 0$ where: 
\begin{itemize}
\item[$(i)$] $h^0(K_C-N) = 2$,
\item[$(ii)$] $D_s$ is an effective divisor such that $h^0(D_s) =1$,
\item[$(iii)$] $\deg(N) \leq \deg(K_C-D_s)$.
\end{itemize}Then either ${\rm Bs}(|K_C-N|) \neq \emptyset$ or $h^0(N-D_s) \geq h^0(N) - s + 1$. 
\end{claim}
\begin{proof}[Proof of Claim \ref{Lemma1Sgen}] If $\text{Bs}(|K_C-N|) = \emptyset$ and $h^0(N-D_s) = h^0(N)-s$, the base-point-free pencil trick implies $\dim(\text{im}(\mu)) = 3g-3-\deg(N)-s$. However, for $h^1(\mathcal{F})=3$, the coboundary map must vanish, implying $\dim(\text{im}(\mu)) < \dim(\text{Ext}^1(K_C-D_s, N))$. Riemann-Roch shows these two dimensions are equal, yielding a contradiction.
\end{proof} 
 
 For bundles $\mathcal{F}$ as in \eqref{eq:sonjagen25}, we verify that the assumptions of Claim \ref{Lemma1Sgen} are satisfied. Specifically, since $N = K_C - (g^1_t + B_b)$, we have $h^0(K_C - N) = 2$. For a general $D_s$ of degree $s \leq g/2$, one has $h^0(D_s) = 1$. Condition $\deg(N) \leq \deg(K_C - N)$ is satisfied for the given range of $d$. Generality of $D_s$ implies $h^0(N - D_s) = h^0(N) - s$. As a consequence of Claim \ref{Lemma1Sgen}, the linear system $|K_C - N|$ must have a non-empty base locus, meaning $B_b$ is an effective divisor of degree $b > 0$. This effectively proves part $(iii)$ of Proposition \ref{prop:nocases} for the case where $B_b = 0$ (i.e., when $K_C - N$ is general in $\overline{W^{\vec{w}_{1,0}}}$). Moving to the case $B_b \neq 0$, the disjointness of $B_b$ and $D_s$ allows us to apply Lemma \ref{reg_seq_1_3}-(ii), showing that $\mathcal{F}$ fits into the following exact sequence:
 \begin{equation}\label{eq:16gen}
 0 \to K_C - (g^1_t + B_b + D_s) \to \mathcal{F} \to K_C \to 0.
\end{equation} This representation is fundamental for applying Claim \ref{lem-1;S3Reg-1} below, which allows one to find bounds on the Segre invariant $s(\mathcal{F})$.

\begin{claim} \label{lem-1;S3Reg-1} Let $C$ be any smooth, irreducible, projective curve of genus $g \geq 4$ and let $g^1_t$ be a base point free complete  pencil on $C$. Assume that a rank-$2$ vector bundle $\mathcal F$ on $C$, with $h^1 (\Ff )=3$, fits into an exact sequence of the form: 
\begin{equation}\label{seq_u-1}
 0 \to K_C -(g^1_t +B_b) \to \mathcal F \to \omega_C \to 0,
\end{equation}
where $B_b$ is an effective divisor of degree $b$ with $h^0 (g^1_t +B_b)=2.$  Then $s (\Ff) \leq -t+b$.
\end{claim}

\begin{proof} [Proof of Claim \ref{lem-1;S3Reg-1}] Let $F=\mathbb{P}(\mathcal{F})$ and let $\widetilde{\Gamma}$ be the canonical section corresponding to the surjection $\mathcal{F} \twoheadrightarrow \omega_C$ in \eqref{seq_u-1}. Given $h^1(\mathcal{F})=3$, we have $h^0(\omega_C \otimes \mathcal{F}^\vee) = 3$, hence $\dim |\mathcal{O}_F(\widetilde{\Gamma})| = 2$. For a general point $q \in F$ projecting to $p \in C$, the pencil $|\mathcal{I}_{q/F}(\widetilde{\Gamma})|$ contains a reducible unisecant $\Gamma = \Gamma_{D_p} + f_{D_p}$, where $\Gamma_{D_p}$ is a section and $D_p \geq p$. This section induces:
\begin{equation}\label{seq_Dp}
0 \to K_C -(g^1_t + B_b) + D_p \to \mathcal{F} \to K_C - D_p \to 0.
\end{equation} The existence of this section implies $h^0(\omega_C \otimes \mathcal{F}^\vee(-D_p)) \geq 1$. Tensoring the dual of \eqref{seq_u-1} by $\omega_C(-D_p)$, we obtain $h^0(g^1_t + B_b - D_p) \geq 1$. Since $h^0(g^1_t + B_b) = 2$ and $p$ is general, we have $h^0(g^1_t + B_b - D_p) = 1$. Dualizing \eqref{seq_Dp} and tensoring by $\omega_C$ yields $h^0(D_p) = 2$ and $h^0(g^1_t + B_b - D_p) = 1$, which implies $|D_p| = g^1_t + B'$ for some $0 \leq B' \leq B_b$. Thus, \eqref{seq_Dp} rewrites as:$$ 0 \to K_C - (B_b - B') \to \mathcal{F} \to K_C - g^1_t - B' \to 0, $$from which we compute $s(\mathcal{F}) \leq \deg(K_C - g^1_t - B') - \deg(K_C - B_b + B') = -t + b$.
\end{proof}

\begin{claim} \label{Lemma3Seonjagen} Let $C$ be a general $\nu$-gonal curve, ad let $\Ff$ be a rank-$2$ vector bundle of degree $d = \deg(\Ff) \geq 3g-5$ and of speciality $h^1(\Ff) = 3$ which is assumed to fit in an exact sequence of the form 
\begin{equation}\label{(*)}
0 \to K_C- (g^1_t+B_b) \to \Ff \to K_C - D_s \to 0,
\end{equation}such that

\smallskip

\noindent
(1) $g^1_t$ is a base-point-free complete pencil of degree $t$; 
\smallskip

\noindent
(2) $B_b \in C^{(b)}$ and $D_s \in C^{(s)}$ such that $h^0(g^1_t+B_b) = h^0(g^1_t + B_b + D_s) = 2$; 
\smallskip

\noindent
(3) $2t \geq t + b + s = 4g-4-d$.

\smallskip

Then $\Ff$ is not stable. 
\end{claim}

\begin{proof} [Proof of Claim \ref{Lemma3Seonjagen}] Let $B' := B_b \cap D_s$, where $0 \leq \deg(B') = b'$. Dualizing \eqref{(*)} and tensoring by $\omega_C$ yields $H^0(\omega_C \otimes \mathcal{F}^\vee) \simeq H^0(g^1_t + B_b) \oplus H^0(D_s)$ since $h^0(D_s)=1$. This implies that every section in $H^0(\omega_C \otimes \mathcal{F}^\vee)$ vanishes on $B'$, so $h^1(\mathcal{F}) = h^1(\mathcal{F}(B')) = 3$. Tensoring \eqref{(*)} by $\mathcal{O}_C(B')$ gives:
\begin{equation}\label{(*1)_red}
0 \to K_C- (g^1_t+B_b-B') \to \mathcal{F}(B') \to K_C - (D_s - B') \to 0.
\end{equation} By definition of $B'$, the divisors $(D_s - B')$ and $(B_b - B')$ are disjoint. Since $g^1_t$ is base-point-free, we apply Lemma \ref{reg_seq_1_3}-(ii) to show $\mathcal{F}(B')$ also fits into:
\begin{equation}\label{(*3)_red}
0 \to K_C - (g^1_t + B_b + D_s - 2B') \to \mathcal{F}(B') \to K_C \to 0.
\end{equation} By assumption (2), $h^0(g^1_t + B_b + D_s) = 2$. Thus, $\mathcal{F}(B')$ satisfies assumptions in Claim \ref{lem-1;S3Reg-1}, yielding:$$ s(\mathcal{F}(B')) \leq -t + (b-b') + (s-b') = -t + b + s - 2b' \leq -t + b + s. $$Finally, assumption (3) ($2t \geq t+b+s$) implies $-t + b + s \leq 0$. Consequently, $s(\mathcal{F}(B')) \leq 0$, proving that $\mathcal{F}(B')$ (and thus $\mathcal{F}$) is not stable.
\end{proof}

We may complete the proof of $(iii)$. Indeed, by assumptions, $\Ff$ fits in \eqref{eq:sonjagen25} necessarily with $B_b>0$ (as the case $B_b = 0$ has been already ruled-out after the proof of Claim \ref{Lemma1Sgen}). Notice that the assumptions $D_s \in C^{(s)}$ general, $K_C-N = g^1_t + B_b \in  \overline{W^{\vec{w}_{1,0}}} \subseteq W^1_{4g-4-d-s}(C)$ as in Lemma \ref{lem:barjvero} and  $ 0 \leq s \leq \frac{g}{2}$, 
$3g-5 \leq d \leq \frac{7}{2} g-5 - s \leq \frac{7}{2} g-5$, $t + b + s = 4g-4-d$ imply that assumptions as in Claim \ref{Lemma3Seonjagen} hold true. Indeed, (1) is obviously satisfied by $g^1_t$; as for (3), i.e. $2t \geq t+b+s = 4g-4-d$, notice that from $d \geq 3g-5$ we have $4g-4-d \leq g+1$ whereas from $g^1_t \in  \overline{W^{\vec{w}_{1,0}}} \subseteq W^1_{t}(C)$ and from Lemma \ref{lem:barjvero} we have that $2t-g-2 \geq 0$, i.e. $2t \geq g+2$. 

Finally we have to check if assumption (2) is satisfied, i.e. we have to prove that  $h^0(g^1_t + B_b + D_s)= 2$. By Riemann-Roch  
$h^0(g^1_t + B_b + D_s) = (t+b+s) - g + 1 + h^0(K_C - (g^1_t + B_b) - D_s)$. Since 
$ h^0(K_C - (g^1_t + B_b)) = h^1(g^1_t + B_b) = g+1-t-b$, the generality of $D_s$ gives 
$h^0(K_C - (g^1_t + B_b) - D_s) =  {\rm max} \{0, g+1-t-b-s\}$ and, from the fact that 
\eqref{(*)} gives $d = 4g-4-t-b-s \geq 3g-5$, we get $g + 1 - t - b -s \geq 0$. Therefore, from above $h^0(g^1_t + B_b + D_s) = (t+b+s) - g + 1 + (g+1-t-b-s) = 2$. Thus, from Claim \ref{Lemma3Seonjagen}, $\Ff$ is not stable and the proof of $(iii)$ is complete. 
\end{proof}

\bigskip

\noindent
$\boxed{\mbox{Obstructions: further case exclusions for $d$ near the lower-bound in \eqref{eq:29dicnewbis}}}$ The following constraints, valid for values of $d$ close to the lower bound in \eqref{eq:29dicnewbis}, lead to additional case exclusions. In particular recall, from Proposition \ref{lem:youngooksup3}, that case $(1-c)$ in Lemma \ref{lem:barjvero} can never yield an irreducible component; for this reason, it will be never mentioned in the statement of the next result.

\begin{proposition}\label{prop:17mar12.52} Let $g \geq 7$, $3 \leq \nu < \frac{g}{2}$ and $2g-2 \le d\le \frac{10}{3}g-6$ be integers and let $C$ be a general $\nu$-gonal curve of genus $g$. Assume that $ B^{k_3}_d \cap U^s_C(3,d)\neq \emptyset$ and let $\mathcal B$ be any of its irreducible component, whose general point is denoted by $[\Ff]$. Then one of the following cases must occur.

\noindent
(1) If $\Ff$ is of {\em first type}, with $\mathcal O_C(D)$ as in \eqref{eq:barj-1}, then:

\begin{itemize}
\item[(1-a)] either $\mathcal O_C(D) \in \overline{W^{\vec{w}_{2,2}}} \subseteq 
W^2_{2g-2-\delta}(C)$, where necessarily either $2g-2 \le d\le \frac{10}{3}g-6$ and $3 \leq \nu \leq \frac{g}{6}$, or $2g-2 \leq d \leq 4g-5-4\nu$ and $\frac{g}{6} < \nu < \frac{g}{2}$. In particular, case $(1-a)$ cannot occur for $\frac{g}{6} < \nu < \frac{g}{2}$ and $4g-4-4\nu \leq d \leq  \frac{10}{3}g-6$;   

\item[(1-b)] or $\mathcal O_C(D) \in \overline{W^{\vec{w}_{2,1}}} \subseteq 
W^2_{2g-2-\delta}(C)$, where necessarily $2g-2 \leq d \leq 3g-7-\nu < \frac{10}{3}g-6$ for any $3 \leq \nu < \frac{g}{2}$. In particular, for $3g-6-\nu \leq d \leq  \frac{10}{3}g-6$ case $(1-b)$ can never occur. 
\end{itemize}

\noindent
(2) If otherwise $\Ff$ is of {\em second type}, then:

\begin{itemize}
\item[(2-a)] either $N$ is general of its degree and either $g=7,8$, for any $2g-2 \leq d \leq \frac{10}{3}g-6$ and $3 \leq \nu < \frac{g}{2}$ or, when $g \geq 9$, necessarily $2g-2 \leq d \leq 3g-3 < \frac{10}{3}g-6$, for any $3 \leq \nu < \frac{g}{2}$. In particular, for $g \geq 9$ and $3g-2 \leq d \leq \frac{10}{3}g-6$, case $(2-a)$ cannot occur;

\item[(2-b)] or $N$ is such that $K_C-N \in \overline{W^{\vec{w}_{1,1}}} \subseteq 
W^1_{2g-2+ \delta -d}(C)$ and either $2g-2 \leq d \leq \frac{10}{3}g-6$, for any $g \geq 12$ and $3 \leq \nu < \frac{g}{2}$, or $2g-2 \leq d \leq \frac{10}{3}g-6$, for $7 \leq g \leq 11$ and  $3 \leq \nu \leq \frac{2}{3}g-2$. In particular, for $7 \leq g  \leq 11$,  $\frac{2}{3}g-1 \leq \nu < \frac{g}{2}$ and  $2g-2 \leq d \leq \frac{10}{3}g-6$, case $(2-b)$ cannot occur; 

\item[(2-c)] or finally $N$ is such that $K_C-N \in \overline{W^{\vec{w}_{1,0}}} \subseteq 
W^1_{2g-2+ \delta -d}(C)$, for any $2g-2 \leq d \leq \frac{10}{3}g-6$ and $3 \leq \nu < \frac{g}{2}$. 
\end{itemize}
\end{proposition}
\begin{proof} A general points $[\mathcal{F}] \in \mathcal{B}$ corresponds to a bundle of either first or second type, satisfying stability condition $d < 2\delta$. We evaluate the numerical consistency of the cases from Lemma \ref{lem:barjvero} within the range $2g-2 \leq d \leq \frac{10}{3}g-6$.

\smallskip

\noindent
(1) Concerning first type bundles, we have:

\begin{itemize}
\item[(1-a)] If $\mathcal{O}_C(D) \in \overline{W^{\vec{w}_{2,2}}}$, then $\delta \leq 2g-2-2\nu$, implying $d \leq 4g-5-4\nu$. This range covers the entire interval if $\nu \leq g/6$. For $g/6 < \nu < g/2$, the case is restricted to $2g-2 \leq d \leq 4g-5-4\nu$.
\item[(1-b)] If $\mathcal{O}_C(D) \in \overline{W^{\vec{w}_{2,1}}}$, then $\delta \leq \frac{3g-6-\nu}{2}$, implying $d \leq 3g-7-\nu$. This sub-interval always lies within $[2g-2, \frac{10}{3}g-6]$.
\item[(1-c)] This case is excluded by Proposition \ref{lem:youngooksup3}.
\end{itemize}

\smallskip

\noindent
(2) As for second type bundles, we have: 
\begin{itemize}
\item[(2-a)] General quotients $\mathcal{L}$ require $d \leq 3g-3$, which is strictly less than the upper bound $\frac{10}{3}g-6$ for $g \geq 9$.
\item[(2-b)] If $K_C-N \in \overline{W^{\vec{w}_{1,1}}}$, then $d \leq 4g-4-\nu$. For $g \geq 12$, this always holds. For $7 \leq g \leq 11$, it requires $\nu \leq \frac{2}{3}g-2$.\item[(2-c)] If $K_C-N \in \overline{W^{\vec{w}_{1,0}}}$, non-emptiness requires $d \leq \frac{7g-6}{2}$. Since $\frac{7g-6}{2} < \frac{10}{3}g-6$ for the given genus, this case is confined to the lower part of the degree range.
\end{itemize}
\end{proof}

\subsection{Components from bundles of first type} In the light of Propositions \ref{prop:nocases}–$(i)$ and \ref{prop:17mar12.52}–$(1)$, there are no irreducible components of $B^{k_3}_d \cap U_C^s(d)$ whose general point corresponds to a bundle of \emph{first type} as in case $(1\text{-}b)$ of Lemma \ref{lem:barjvero}, for $\frac{10}{3}g - 5 \leq d \leq 4g - 5 - 2\nu$ and, respectively, for $3g - 6 - \nu \leq d \leq \frac{10}{3}g - 6$. Similarly, for case $(1\text{-}c)$, the same conclusion holds true because of Proposition \ref{lem:youngooksup3}. Therefore, for $2g-2 \leq d \leq 4g-7$ and among bundles of \emph{first type}, case $(1\text{-}a)$ as in Lemma \ref{lem:barjvero} remains the only admissible candidate for yielding an irreducible component of $B^{k_3}_d \cap U_C^s(d)$ filled by bundles of \emph{first type}. For this reason, in this section we focus on Case $(1\text{-}a)$.

\bigskip 

\noindent
$\boxed{2g-2 \leq d \leq 4g-5-2\nu: \; \mbox{Components from Case (1-a)}}$ In the range $2g-2 \leq d \leq 4g-7$, we have the following: 

\begin{proposition}\label{prop:case1a} Let $C$ be a general $\nu$-gonal curve $C$ of genus $g$.

\medskip

\noindent
$(i)$ If $g \geq 4$, $3 \leq \nu < \lfloor \frac{g+3}{2}\rfloor$ and $4g-4-4\nu \leq d \leq 4g-7$, there is no irreducible component of $B_d^{k_3} \cap U_C^s(d)$ whose general point $[\mathcal F]$ corresponds to a rank-$2$ stable bundle of degree $d$ and speciality $h^1(\mathcal F) = 3$, which is of {\em first type} as in case $(1-a)$ of Lemma \ref{lem:barjvero}.

\medskip 

\noindent 
$(ii)$   If $g \geq 9$, $3 \leq \nu \leq \frac{g}{3}$ and $2g-6+2\nu \leq d \leq  4g-5-4\nu$, there exists an irreducible component of $B_d^{k_3} \cap U_C^s(d)$, denoted by $\mathcal B_{\rm 3,\,(1-a)}$, which is {\em uniruled, generically smooth} and of dimension $$\dim(B_{\rm 3,\,(1-a)} )= 6g-6-4\nu - d.$$In particular 
it is {\em superabundant}, i.e. $ \mathcal B_{\rm 3,\,(1-a)} = \mathcal B_{\rm sup,\,3,\,(1-a)}$, when $d \geq 2g-5+2\nu$ whereas it is {\em regular}, i.e. $ \mathcal B_{\rm 3,\,(1-a)} = \mathcal B_{\rm reg,\,3,\,(1-a)}$, when otherwise $d = 2g-6+2\nu$. 

Furthermore, its general point $[\mathcal F] \in B_{\rm 3,\,(1-a)}$ corresponds to a rank-$2$ stable bundle of degree $d$, speciality $h^1(\mathcal F) = 3$, which is of {\em first type} as in case $(1-a)$ of Lemma \ref{lem:barjvero} and fitting in an exact sequence of the form $$0 \to N \to \mathcal F \to K_C-2A \to 0,$$where $A \in {\rm Pic}^{\nu}(C)$ is the unique line bundle on $C$ such that $|A|= g^1_{\nu}$, whereas $N \in {\rm Pic}^{d - 2g + 2 + 2\nu}(C)$ is general and $\mathcal B_{\rm 3,\,(1-a)}$ is the only irreducible component whose general point arises from a bundle $\mathcal F$ of {\em first type} on $C$. 

\medskip

\noindent
$(iii)$ If $g \geq 9$, $3 \leq \nu \leq \frac{g}{3}$ and $2g-2 \leq d \leq 2g-7+ 2 \nu$, there is no component of $B_d^{k_3} \cap U_C^s(d)$ whose general point $[\mathcal F]$ corresponds to a rank-$2$ stable bundle of degree $d$ and speciality $h^1(\mathcal F) = 3$, which is of {\em first type} as in case $(1-a)$ of Lemma \ref{lem:barjvero}. 
\end{proposition}

\begin{proof} (i). For a component $\mathcal{B} \subseteq B_d^{k_3} \cap U_C^s(d)$ to contain stable bundles, the stability condition requires $d < 2\delta$, where $\delta = \deg(L)$ for any quotient line bundle $L$. For a bundle of first type (case 1-a), the quotient $L = \omega_C(-D)$ satisfies $h^1(L) = 3$. According to Lemma \ref{lem:w2tLar}, this needs $\delta \leq 2g - 2 - 2\nu$. Combining these yields the sharp existence bound:$$d < 2\delta \leq 4g - 4 - 4\nu.$$(ii). We construct the component by considering the extension:$$(u): 0 \to N \to \mathcal{F}_u \to K_C - 2A \to 0$$where $|A| = g^1_{\nu}$ is the unique $\nu$-gonal pencil. We set the quotient $L = K_C - 2A$, so $\delta = 2g - 2 - 2\nu$. By Proposition \ref{Ballico}, $h^0(2A) = 3$ (since $g \geq 2\nu - 2$), which via Serre duality implies $h^1(L) = 3$. This confirms that $\mathcal{F}_u$ has the required speciality $h^1(\mathcal{F}) \geq h^1(L) = 3$. The condition $h^0(L) = g + 2 - 2\nu > 0$ ensures the existence of the quotient line bundle.

\begin{claim}\label{lem:i=2.2.3_1-a} Let $C$ be a general $\nu$-gonal curve with $d$ and $\nu$ as in the above assumptions. Let $\mathcal F_u$ be a rank-$2$ vector bundle arising as a general extension $u\in {\rm Ext}^1(K_C-2A, N)$, where $N$ is a general line bundle in ${\rm Pic}^{d-2g+2 + 2\nu}(C)$. Then  $\mathcal F_u$ is stable and such that $h^1(\mathcal F_u) = h^1(K_C-2A) = 3$. 
\end{claim} 

\begin{proof}[Proof of Claim \ref{lem:i=2.2.3_1-a}] Set the quotient line bundle $L := K_C - 2A$ with $\deg(L) = \delta = 2g-2-2\nu$. For a general line bundle $N$ of degree $d-\delta$, numerical assumption $d < 4g-4-4\nu$ ensures that $\delta > \frac{d}{2}$ and $N \ncong L$. The dimension of the extension space is:$$m := \dim(\text{Ext}^1(L, N)) = 5g-5-4\nu-d > 0.$$We may verify $h^1(\mathcal{F}_u) = 3$ by analyzing the coboundary map $\partial_u: H^0(L) \to H^1(N)$. Let $\ell = h^0(L) = g+2-2\nu$ and $r = h^1(N)$. Whether $N$ is special or non-special, the condition $\ell \geq r$ holds for $d \geq 2g-2$. By Theorem \ref{CF5.8}, for a general $u \in \text{Ext}^1(L, N)$, $\partial_u$ is surjective, implying $h^1(\mathcal{F}_u) = h^1(L) = 3$.

To prove that the general extension $\mathcal{F}_u$ is stable ($s(\mathcal{F}_u) > 0$), we distinguish two cases based on the difference $2\delta - d$:

\smallskip 

\noindent
$\bullet$ If $2 \delta -d \geq 2$, i.e. $d \leqslant 4g-6-4\nu$, we may apply  Proposition \ref{LN}, with $X = \varphi(C)$ the image in $\Pp:=\Pp(\ext^1(K_C- 2A\color{black}, N))$ via the morphism $\varphi= \varphi_{|2K_C-2A-N|}$. From above one has $m = \dim(\ext^1(K_C- 2A\color{black},N)) = 5g-5-4\nu-d$ so, for any integer $\sigma \leq g$, one has $$\dim\left(\Sec_{\frac{1}{2}((4g-4-4\nu -d) + \sigma -2)}(X)\right) < m-1 = \dim(\mathbb{P}),$$so $\Ff_u$ is stable in this case.

\smallskip 

\noindent
$\bullet$ when otherwise $2\delta - d = 1$, hence $d = 4g-5-4\nu$ (odd), if $\mathcal{F}_u$ were unstable, it would admit a sub-line bundle $\mathcal{M}$ with $\deg(\mathcal{M}) \geq \frac{d+1}{2}$. Since $\deg(\mathcal{M}) > \deg(N)$, the only possible map $\mathcal{M} \to L$ is an isomorphism. This would imply the extension splits, which contradicts the choice of a general $u \neq 0$ in the $g$-dimensional extension space. Thus, also in this case $\mathcal{F}_u$ is stable.
\end{proof}

From previous arguments, one can consider a vector bundle 
$\mathcal E_{d,\nu}$ on a suitable open, dense subset $S \subseteq 
\pic^{d-2g+2+2\nu+b}(C)$, described by $N $ varying in $S$, whose rank is  
$ m= \dim (\ext^1(K_C-2A,\, N))= 5g-5- d - 4 \nu$. Taking the associated projective bundle $\mathbb P(\mathcal E_{d,\nu,b})\to S$ (consisting of $\mathbb P\left(\ext^1(K_C-2A,\;N)\right)$'s as $N$ varies in $S$) one has $\dim (\mathbb P(\mathcal E_{d,\nu})) = g+ (5g-6- d - 4 \nu)= 6g-6-d - 4 \nu$. The latter quantity equals $\rho_d^{k_3} = 10g-18-3d$ only when $d= 2g-6+2\nu$, whereas it is bigger than $\rho_d^{k_3}$ for $d > 2g-6+2\nu$. When  otherwise $d<2g-6+2\nu$, the dimension of $\mathbb P(\mathcal E_{d,\nu})$ is instead less than $\rho_d^{k_3} $ (which in particular will also prove $(iii)$). From Claim \ref{lem:i=2.2.3_1-a}, one moreover has a natural (rational) modular map
 \begin{eqnarray*}
 &\mathbb P(\mathcal E_{d,\nu})\stackrel{\pi_{d,\nu}}{\dashrightarrow} &U_C(d) \\
 &(N, [u])\to &\mathcal F_u;
 \end{eqnarray*} such that  $ \im (\pi_{d,\nu})\subseteq B^{k_3}_d \cap  U^s_C(d)$.

\begin{claim}\label{cl:piddeltanugenfin_3_1-a} In the above assumptions, the map $\pi_{d,\nu}$ is {\em birational} onto its image. In particular, $\im(\pi_{d,\nu})$ is {\em uniruled} and one has also that
$\dim(\im(\pi_{d,\nu})) = \dim(\mathbb P(\mathcal E_{d,\nu})) = 6g-6-d - 4 \nu$. 
\end{claim}

\begin{proof} [Proof of Claim \ref{cl:piddeltanugenfin_3_1-a}] The image $\text{im}(\pi_{d,\nu})$ is uniruled as it is dominated by $\mathbb{P}(\mathcal{E}_{d,\nu})$, which is ruled over $S$. To determine its dimension, we show that the map $\pi_{d,\nu}$ is birational onto its image.

Let $\mathcal{F}_u$ be a general bundle in $\text{im}(\pi_{d,\nu})$. Suppose there exists a distinct pre-image $(N', [u'])$ such that $\mathcal{F}_{u'} \simeq \mathcal{F}_u$. Since the quotient $L = K_C - 2A$ is fixed, determinantal considerations force $N' \cong N$. If $[u] \neq [u']$, one gets the presence of an isomorphism $\varphi: \mathcal{F}_{u'} \to \mathcal{F}_u$ ov vector bundles which would provide two linearly independent sections $s_1, s_2 \in H^0(\mathcal{F}_u \otimes N^\vee)$. We will show instead that $h^0(\mathcal{F}_u \otimes N^\vee) = 1$.

By Serre duality, $h^0(\mathcal{F}_u \otimes N^\vee) = h^1(\mathcal{F}_u \otimes A^{\otimes 2})$. Tensoring the extension $(u)$ by $A^{\otimes 2}$ yields:$$0 \to N \otimes A^{\otimes 2} \to \mathcal{F}_u \otimes A^{\otimes 2} \to K_C \to 0.$$$\mathcal{F}_u \otimes A^{\otimes 2}$ corresponds to a general element in $\text{Ext}^1(K_C, N \otimes A^{\otimes 2})$. We distinguish two cases:

\smallskip

\noindent
$\bullet$ if $h^1(N \otimes A^{\otimes 2}) = 0$, then $h^1(\mathcal{F}_u \otimes A^{\otimes 2}) = h^1(K_C) = 1$;

\smallskip

\noindent
$\bullet$ if otherwise $h^1(N \otimes A^{\otimes 2}) \geq 1$, the generality of $N$ and the numerical bounds on $d$ ensure the hypotheses of Theorem \ref{CF5.8} are met. Thus, the coboundary map $H^0(K_C) \to H^1(N \otimes A^{\otimes 2})$ is surjective, again yielding $h^1(\mathcal{F}_u \otimes A^{\otimes 2}) = h^1(K_C) = 1$.

In both cases, $h^0(\mathcal{F}_u \otimes N^\vee) = 1$, implying $s_1$ and $s_2$ are linearly dependent hence $[u] = [u']$. Thus, $\pi_{d,\nu}$ is birational and $\dim(\text{im}(\pi_{d,\nu})) = 6g - 6 - d - 4\nu$.
\end{proof}

Let us denote by $\mathcal B_{\rm 3,(1-a)}$ the closure of $\im (\pi_{d,\nu})$ in $B_d^{k_3}\cap U^s_C(d)$; from Claim \ref{cl:piddeltanugenfin_3_1-a}, one has   
$$\dim (\mathcal B_{\rm 3,(1-a)}) = 6g - 6 - d - 4\nu,$$since $ \im (\pi_{d,\nu})$ is dense in it. We want to find under which conditions $\mathcal B_{\rm 3,(1-a)}$ gives rise to an irreducible component of $B_d^{k_3}\cap U^s_C(d)$ and, if so, determine when it is also generically smooth. To do so, by Claim \ref{lem:i=2.2.3_1-a}, we have $h^1(\mathcal F)=h^1(K_C-2A) = h^0(2A)= 3$; consider the following diagram:
{\footnotesize

\begin{equation*}
\begin{array}{ccccccccccccccccccccccc}
&&0&&0&&0&&\\[1ex]
&&\downarrow &&\downarrow&&\downarrow&&\\[1ex]
0&\lra& N + 2A - K_C & \rightarrow & \mathcal F\otimes A^{\otimes 2} \otimes \omega^{\vee}_C & \rightarrow & \mathcal O_C &\rightarrow & 0 \\[1ex]
&&\downarrow && \downarrow && \downarrow& \\[1ex]
0&\lra & N \otimes  \mathcal F^{\vee}& \rightarrow & \mathcal F\otimes \mathcal F^{\vee} & \rightarrow &
\mathcal  (K_C-2A) \otimes \mathcal F^{\vee}&\rightarrow & 0 \\[1ex]
&&\downarrow &&\downarrow &&\downarrow  &&\\[1ex]
0&\lra & \mathcal O_{C}& \lra & \mathcal F\otimes N^{\vee}&\lra& (K_C-2A)\otimes N^{\vee}&0& \\[1ex]
&&\downarrow &&\downarrow&&\downarrow&&\\[1ex]
&&0&&0&&0&&
\end{array}
\end{equation*}}which arises from the exact sequence $0 \to N \to \Ff \to K_C-2A\to 0$ tensored by its dual sequence. If we tensor the column in the middle of the above diagram by  $\omega_C$, we get $H^0(\mathcal F\otimes A^{\otimes 2})\hookrightarrow  H^0(K_C\otimes \mathcal F\otimes \mathcal F^{\vee})$. From the isomorphism $H^0(K_C\otimes\mathcal F^{\vee}) \simeq H^0(2A)$, the Petri map of $\Ff$ reads
$$\mu_{\Ff} : H^0(\mathcal F)\otimes H^0(2A)\to  H^0(K_C\otimes \mathcal F\otimes \mathcal F^{\vee}).$$Moreover, from the same column tensored by $\omega_C$, we get also  the injection  
\begin{equation}\label{eq:injh}
H^0(\Ff \otimes 2A) \stackrel{ \mathfrak{h}}{\hookrightarrow}  H^0(K_C\otimes \mathcal F\otimes \mathcal F^{\vee})
\end{equation}from which we deduce that the kernel of $\mu_{\Ff}$ is the same as the kernel of the following multiplication map 
$$H^0(\mathcal F)\otimes H^0(2A)\stackrel{\mu_{\mathcal F,\;2A}}{\longrightarrow}  H^0(\mathcal F\otimes 2A).$$We need to investigate the connection between the Petri map $\mu_{\mathcal F}$ and 
the multiplication map $\mu_{\mathcal F,\;2A}$. 

Let us therefore consider 
$[\mathcal F] \in \mathcal B_{\rm 3, (1-a)}$ general, so that $\mathcal F = \mathcal F_u$ fits in an exact sequence of the form  
\begin{equation}\label{eq;Seq_2}
(u) \ \ \  0\to N  \to \mathcal F_ u\to  K_C -2 A \to 0,
\end{equation}
   in which $ N\in {\rm Pic}^{d-2g-2-2 \nu}(C)$ is general. Dualizing \eqref{eq;Seq_2}, one gets 
  \begin{equation}\label{eq;seq_dual} 
   0\to 2 A \to K_C \otimes \mathcal F_u^{\vee} \to K_C-N \to 0,
 \end{equation} 
   whence one has the natural injection
   $H^0 (2 A) \stackrel{\boldsymbol\iota}\hookrightarrow H^0 (K_C \otimes \mathcal F_u^{\vee})$ which is indeed an isomorphism as 
   $h^0 (K_C \otimes \mathcal F_u^{\vee})=h^1 ( \mathcal F_u) =h^0 (2 A)=3$. Taking into account the injection \eqref{eq:injh}, in sum we obtain the following commutative diagram where, to ease notation, we have set $\mu':= \mu_{\mathcal F, \; 2A}$:
    {\footnotesize
   \begin{equation*}\label{eq;comm_cup map}
\begin{array}{ccccccccccccccccccccccc}
 H^0(\mathcal F_u)\otimes H^0(K_C\otimes\mathcal F_u^{\vee}) &\stackrel{\mu _{\mathcal F_u}}\longrightarrow  & H^0(K_C\otimes \mathcal F_u\otimes \mathcal F_u^{\vee}) & &  \\[1ex]
 \uparrow {\rm Id} \otimes {\boldsymbol\iota}  && \uparrow  \mathfrak{h}& \\[1ex]
H^0(\mathcal F_u)\otimes  H^0 (2 A) &\stackrel{\mu^\prime}\longrightarrow &H^0 (\mathcal F_u \otimes 2 A)&
\end{array}
\end{equation*}}We notice that  ${\rm Id} \otimes {\boldsymbol\iota}$ is an isomorphism so, by the injectivity of $\mathfrak{h}$, it follows that 
\begin{equation} \label{eq;ker_equal}
\dim (\ker ({\mu _{\mathcal F_u}})) = \dim ( \ker (\mu^\prime)).
\end{equation}We focus on the computation of $ \dim ( \ker (\mu^\prime)).$ To do this, we establish the following more general result.

\begin{claim} \label{cl:surj} Let $C$ be a general $\nu$-gonal curve and let $\mathcal F_u$ be a rank-$2$, degree $d$ vector bundle arising as a general extension as in \eqref{eq;Seq_2}, where $N \in {\rm Pic}^{d-2g+2+2\nu}(C)$ is general. If $d\geq 2g-6+2\nu$, then for any integer $1\leq k\leq 2$ the multiplication map 
$$\mu_k :  H^0(\mathcal F_u\otimes A^{\otimes(k-1)})\otimes H^0(A)\to H^0(\mathcal F_u\otimes A^{\otimes k})$$ is surjective.
\end{claim} 

\begin{proof}[Proof of Claim \ref{cl:surj}]To show the surjectivity of the multiplication maps $\mu_k$, we view $\mathcal{F}_u \otimes A^{\otimes s}$ as a general element of $\text{Ext}^1(K_C - (2-s)A, N + sA)$ for $-1 \leq s \leq 2$.First, we verify the hypotheses of Theorem \ref{CF5.8}. Let $\ell = h^0(K_C - (2-s)A)$ and $r = h^1(N + sA)$. Since $3\nu \leq g$, Proposition \ref{Ballico} implies $\ell = g - (2-s)\nu + 2 - s$. Given $d \geq 2g + 2\nu - 6$ and the generality of $N$, a dimension count confirms $\ell \geq r$. Furthermore, the dimension of the extension space $m = 5g - 5 - 4\nu - d$ satisfies $m \geq \ell + 1$ within the prescribed range for $d$. Thus, by Theorem \ref{CF5.8}, the coboundary maps $\partial_u: H^0(K_C - (2-s)A) \to H^1(N + sA)$ are surjective for all $- 1 \leq s \leq 2$. This surjectivity implies $h^1(\mathcal{F}_u \otimes A^{\otimes s}) = h^1(K_C - (2-s)A)$, which equals $3-s$ by Proposition \ref{Ballico}. Riemann-Roch then provides:\begin{equation} \label{eq:h0_val}h^0(\mathcal{F}_u \otimes A^{\otimes s}) = d + 2s\nu - 2g - s + 5.\end{equation}For $1 \leq k \leq 2$, consider the sequence derived from the base-point-free-pencil trick on $A$:$$0 \to H^0(\mathcal{F}_u \otimes A^{\otimes(k-2)}) \to H^0(\mathcal{F}_u \otimes A^{\otimes(k-1)}) \otimes H^0(A) \xrightarrow{\mu_k} H^0(\mathcal{F}_u \otimes A^{\otimes k}).$$From this, $\dim(\text{Im }\mu_k) = 2h^0(\mathcal{F}_u \otimes A^{\otimes(k-1)}) - h^0(\mathcal{F}_u \otimes A^{\otimes(k-2)})$. Substituting the values from \eqref{eq:h0_val}, we find that $\dim(\text{Im }\mu_k)$ exactly matches $h^0(\mathcal{F}_u \otimes A^{\otimes k})$, thereby proving that $\mu_k$ is surjective.
\end{proof}

To show that the multiplication map $\mu': H^0(\mathcal{F}_u) \otimes H^0(2A) \to H^0(\mathcal{F}_u \otimes 2A)$ is surjective, we use Claim \ref{cl:surj} through the following iterative argument. Since $\mu_1$ and $\mu_2$ are surjective, the composed map $\tilde{\mu} := \mu_2 \circ (\mu_1 \otimes \text{Id}_{H^0(A)})$ is also surjective:
$$ \tilde{\mu} : H^0(\mathcal{F}_u) \otimes H^0(A)^{\otimes 2} \twoheadrightarrow H^0(\mathcal{F}_u \otimes 2A).$$We use the fact that the multiplication map of the pencil $\mu_{2A} : H^0(A)^{\otimes 2} \to H^0(2A)$ is surjective (as $\dim|2A|=2$). This leads to the following diagram:$$\begin{array}{ccc}
H^0(\mathcal{F}_u) \otimes H^0(A)^{\otimes 2} & \xrightarrow{\text{Id} \otimes \mu_{2A}} & H^0(\mathcal{F}_u) \otimes H^0(2A) \\
\downarrow \tilde{\mu} & & \downarrow \mu' \\
H^0(\mathcal{F}_u \otimes 2A) & = & H^0(\mathcal{F}_u \otimes 2A)
\end{array}$$Since both $\tilde{\mu}$ and $(\text{Id} \otimes \mu_{2A})$ are surjections, a diagram-chase confirms that $\mu'$ must also be surjective.

\begin{claim}\label{lem:i=2.1.3_1-a} Let $C$ be a general $\nu$-gonal curve and let \; $2g-6+2\nu \leq d \leq 4g-5-4\nu$ be an integer. The locus  $\mathcal B_{\rm 3,(1-a)}$ constructed above is a reduced, irreducible  component of $B^{k_3}_d \cap U^s_C(d)$ which is {\em uniruled, generically smooth} and of dimension $\dim (\mathcal B_{\rm 3,(1-a)}) = 6g - 6 - d - 4\nu$. 

Furthermore, if $d= 2g  -6 +2\nu$, $\mathcal B_{\rm 3,(1-a)} = \mathcal B_{\rm reg,3,(1-a)}$, i.e. it is a {\em regular} component of \linebreak $B^{k_3}_d  \cap U^s_C (2,d)$;  if otherwise $2g  -5+2\nu   \leq d\leq 4g-5-4\nu$, $\mathcal B_{\rm 3,(1-a)} = \mathcal B_{\rm sup,3,(1-a)}$ is a {\em superabundant} component of $B^{k_{3}}_d  \cap U^s_C (2,d)$. 
\end{claim} 

\begin{proof}[Proof of Claim \ref{lem:i=2.1.3_1-a}] The locus $\mathcal{B}_{3,(1\text{-}a)}$ is uniruled, as it is the closure of $\text{im}(\pi_{d,\nu})$, which is dominated by the projective bundle $\mathbb{P}(\mathcal{E}_{d,\nu})$. By Claim \ref{cl:piddeltanugenfin_3_1-a}, its dimension is:$$\dim(\mathcal{B}_{3,(1\text{-}a)}) = 6g - 6 - d - 4\nu.$$To prove generic smoothness, we compute the dimension of the tangent space at a general point $[\mathcal{F}_u]$. Since $\mu'$ is surjective, $\dim(\ker \mu_{\mathcal{F}_u}) = h^0(\mathcal{F}_u)h^0(2A) - h^0(\mathcal{F}_u \otimes 2A)$. Recalling that $h^1(\mathcal{F}_u) = h^0(2A) = 3$, the Petri map dimension formula gives:
\begin{align*}
\dim T_{[\mathcal{F}_u]}(\mathcal{B}_{3,(1\text{-}a)}) &= 4g - 3 - h^0(\mathcal{F}_u)h^1(\mathcal{F}u) + \dim(\ker \mu{\mathcal{F}_u}) \&= 4g - 3 - h^0(\mathcal{F}_u \otimes 2A).
\end{align*} Replacing $h^0(\mathcal{F}_u \otimes 2A) = d + 4\nu - 2g + 3$ from \eqref{eq:h0_val}, we obtain exactly $6g - 6 - d - 4\nu$, confirming that the component is generically smooth. 

Finally, the component is superabundant because $\dim(\mathcal{B}_{3,(1\text{-}a)}) \geq \rho_d^{k_3} = 10g - 18 - 3d$ whenever $d \geq 2g - 6 + 2\nu$, with equality (regularity) holding iff $d = 2g - 6 + 2\nu$.
\end{proof}

To conclude the proof of $(ii)$, we establish the uniqueness of the irreducible component $\mathcal{B}_{3,(1-a)}$ by excluding alternative bundle types from Lemma \ref{lem:barjvero}. 
Following Proposition \ref{lem:youngooksup3}, bundles of type $(1-c)$ do not fill-up an irreducible component for any degree $d$. Similalry bundles of type $(1-b)$ cannot fill-up 
an irreducible component in the range $3g-6-\nu \leq d \leq 4g-5-4\nu$ from Propositions \ref{prop:nocases}-$(i)$ and \ref{prop:17mar12.52}-$(1)$; such bundles are later shown (cf. Lemma \ref{lem:sup3case1-b} and Proposition \ref{prop:fanc19-12} below) to be also incapable of forming a component in the lower range $2g-2 \leq d \leq 3g-7-\nu$. 

Thus, $\mathcal{B}_{3,(1-a)}$ remains the only candidate. We now show that within bundles of type $(1-a)$, the specific construction using the $\nu$-gonal pencil is the unique way to obtain a component of $B_d^{k_3} \cap U_C^s(d)$. This will be proved in the following:

\begin{claim}\label{lemsup} Let $g \geq 9$, $3 \leq \nu \leq \frac{g}{3}$ and $2g-6+2\nu \leq d \leq 4g-5 - 4 \nu$ be integers and let $C$ be a general $\nu$-gonal curve. Assume that $\mathcal B$ is any irreducible  component of $B_d^{k_3} \cap U^s_C(d)$ whose general point $[\mathcal F] \in  \mathcal B$ corresponds to a rank-$2$ stable bundle of degree $d$ arising as in case $(1-a)$ of Lemma \ref{lem:barjvero}.  Then $\mathcal B = \mathcal B_{3,(1-a)}$ as in Claim \ref{lem:i=2.1.3_1-a}. 
\end{claim}

\begin{proof} [Proof of Claim \ref{lemsup}] Since, by assumption, $[\mathcal F] \in  \mathcal B$ general arises as in case $(1-a)$ of Lemma \ref{lem:barjvero}, we have
$K_C - L \cong A +B_b$, where $B_b \in C^{(b)}$ is the base locus of $|A+B_b|$ of degree $b \geq 0$. Moreover $\mathcal F$ corresponds to a suitable $v \in {\rm Ext}^1(K_C-A-B_b, N_b)$, for some line bundle $N_b \in {\rm Pic}^{d-2g+2+\nu+b}(C)$. Moreover, always by assumption, $L = K_C-A-B_b$ is such that $h^1(L) = h^1(K_C-A-B_b)= h^1(\mathcal F) = 3$; therefore,  by  taking cohomology in 
$$0 \to N_b \to \mathcal F \to K_C-A-B_b \to 0,$$regardless the speciality of $N_b$, the corresponding 
coboundary map $H^0(K_C-A-B_b) \stackrel{\partial_v}{\longrightarrow} H^1(N_b)$ has to be surjective. From semicontinuity on ${\rm Ext}^1(K_C-A-B_b, N_b)$ and the fact that semistability is an open condition in irreducible flat families, for a general $u \in   {\rm Ext}^1(K_C-A-B_b, N_b)$ the coboundary map $\partial_u$ is also surjective and $\mathcal F_u$ is semistable, of speciality $3$. 

Since $\mathcal B$ is an irreducible component of $B_d^{k_3} \cap U^s_C(d)$ and since 
$u $ general specializes to $v \in {\rm Ext}^1(K_C-A-B_b, N_b)$, then 
$[\mathcal F] \in \mathcal B$ general has to come from a general $u \in   {\rm Ext}^1(K_C-A-B_b, N_b)$, for some $B_b \in C^{(b)}$ and from some $N_b \in {\rm Pic}^{d-2g+2+\nu+b}(C)$.  On the other hand one observes that a general  extension as 
$$(*):\;\;\; 0 \to N_b \to \mathcal F_u \to K_C-A-B_b \to 0$$is a {\em flat specialization} of a general extension of the form 
$$(**):\;\;\; 0 \to N \to \mathcal F \to K_C-A \to 0,$$where $N \cong N_b -B_b$. To observe this, it suffices to note that extensions of type $(**)$ are parametrized by the vector space $\mathrm{Ext}^1(K_C - A, N) \cong H^1(N + A - K_C)$, while extensions $(*)$ are parametrized by $\mathrm{Ext}^1(K_C - A - B_b, N + B_b) \cong H^1(N + 2B_b + A - K_C)$. The existence of such a flat specialization is ensured by the surjectivity $$H^1(N+A -K_C) \twoheadrightarrow  H^1(N + 2 B_b + A - K_C),$$which follows from the exact sequence 
$0 \to \mathcal O_C \to \mathcal O_C(2 B_b) \to \mathcal O_{2B_b} \to 0$ tensored by $N+A -K_C$. For an interpretation of this surjectivity via \emph{elementary transformations} of vector bundles, see \cite[pp.\;101-102]{L}. 

From Claim \ref{lem:i=2.1.3_1-a}, the construction of $ \mathcal B_{3,(1-a)}$ and semicontinuity, it follows that general extensions $(**)$ give rise to points in 
$ \mathcal B_{3,(1-a)}$. By flat specialization of a general extension $(**)$ to a general extension $(*)$, one can conclude that $\mathcal B \subseteq \mathcal B_{3,(1-a)}$, i.e. that $\mathcal B = \mathcal B_{3,(1-a)}$. 
\end{proof}

The above arguments conclude the proof of $(ii)$.

\medskip

\noindent
$(iii).$ From computations as in $(ii)$, recall that when  $d<2g-6+2\nu$ the dimension of the projective bundle $\mathbb P(\mathcal E_{d,\nu})$ is smaller than $\rho_d^{k_3} $, namely condition $d\geq 2g-6+2\nu$ is necessary for the existence of an irreducible component. Therefore also $(iii)$ is proved. 
 \end{proof}

\subsection{Components from modifications of bundles of second type} As for vector bundles of \emph{second type}, Propositions \ref{prop:nocases} and \ref{prop:17mar12.52} 
rule out the existence of irreducible components of $B^{k_3}_d \cap U_C^s(d)$ whose general point corresponds to a bundle as in case $(2\text{-}a)$ of Lemma \ref{lem:barjvero}, 
for $3g - 2 \leq d \leq 4g - 5 - 2\nu$. Similarly, Proposition \ref{prop:nocases} implies the same conclusion for case $(2\text{-}c)$ when $\frac{7g - 5}{2} \leq d \leq 4g - 5 - 2\nu$.

Therefore, within $\frac{7g - 5}{2} \leq d \leq 4g - 5 - 2\nu$, case $(2\text{-}b)$ in Lemma \ref{lem:barjvero} remains the only potential candidate giving rise to possible irreducible components of $B_d^{k_3} \cap U_C^s(d)$ filled-up by bundles of \emph{second type}. 
However, in Proposition \ref{prop:case2b}–$(i)$ below, we show that if $3g - 5 \leq d \leq 4g - 5 - 2\nu$, and a vector bundle $\mathcal F$ of type $(2\text{-}b)$ is to 
represent the general point of such a component, the kernel line bundle $N$ in its presentation — satisfying $\omega_C \otimes N^{\vee} \in \overline{W^{\vec{w}{1,1}}} \subseteq W^1_{\nu + b}(C)$ as required in case $(2\text{-}b)$ — must be such that ${\rm Bs}(|\omega_C \otimes N^{\vee}|) \neq \emptyset$ (recall description of the elements in the component   $\overline{W^{\vec{w}_{1,1}}}$ as in Lemma \ref{lem:Lar}-$(i)$). Nevertheless, we will also show that any such bundle admits an {\em alternative presentation} without base locus, namely via the quotient line bundle $\omega_C \otimes A^{\vee}$, $|A| = g^1_\nu$ (see the proof of Claim \ref{cl:case2b-1}). In this setting the line bundle $\omega_C(-D)$, where $D$ general effective, which is expected to serve as the quotient line bundle in a presentation of a bundle $\mathcal F$ of type $(2\text{-}b)$, can be replaced in the alternative presentation by the quotient $\omega_C \otimes A^{\vee}$ and appears instead as the \emph{kernel} line bundle of the alternative presentation.

Although this alternative presentation does not fall under the classification of either \emph{first} or \emph{second type} bundles as in Definition \ref{def:fstype}, it considerably simplifies both the explicit construction of desired components and the analysis of their Zariski tangent spaces (see e.g. Proposition \ref{prop:case2b}). Moreover, as an additional benefit, this approach extends to lower values of $d$, specifically for $2g - 7 + 2\nu \leq d \leq 3g - 6$ (cf. parts $(i_2)$ and $(ii)$ of Proposition \ref{prop:case2b}), where proving that such bundles admit a presentation as in case $(2\text{-}b)$ of Lemma \ref{lem:barjvero} becomes more difficult (note, for instance, that the proof of Claim \ref{cl:case2b-1} crucially relies on the assumption $d \geq 3g - 5$).

\medskip

\noindent
$\boxed{2g-7+2\nu \leq d \leq 4g-5-2\nu:\;\mbox{Components from modifications of Case (2-b)}}$ Proposition \ref{prop:case2b} below will provide \emph{parametric constructions} of components of $B^{k_3}d \cap U^s_C(d)$ arising via suitable \emph{modifications} of presentations of bundles of type $(2\text{-}b)$ in Lemma \ref{lem:barjvero}. Namely, constructed components will be of the form $\mathcal B_{\mathrm{sup}, 3, (2\text{-}b)\text{-mod}}$ or $\mathcal B_{\mathrm{reg}, 3, (2\text{-}b)\text{-mod}}$ in the sense of Definition \ref{def:comp12mod}.

\begin{proposition}\label{prop:case2b}  Let $C$ be a general $\nu$-gonal curve of genus $g$. One has the following situation:

\medskip

\noindent
$(i)$  If either $(i_1)$ $g \geq 6$, $ 3 \leq \nu \leq \frac{g}{2}$, $ 3g -5 \leq d \leq 4g-5-2\nu$ and $\nu+1 \leq s \leq g-\nu+1 $ or $(i_2)$ $g \geq 8$, $ 3 \leq \nu \leq \frac{g+4}{4}$, $ 3g - 3-\nu \leq d \leq 3g-6$ and $ g-\nu+2 \leq s \leq g-1 $, there exists an irreducible component of $B_d^{k_3} \cap U_C^s(d)$, denoted by $\mathcal B_{\rm sup,\,3,\,(2-b)-mod}$, which is {\em uniruled, generically smooth} and {\em superabundant}, being of dimension
$$\dim \left(\mathcal B_{\rm sup,\,3,\,(2-b)-mod}\right) = 8g-11-2\nu - 2d > \rho_d^{k_3},$$whose general point $[\mathcal F] \in \mathcal B_{\rm sup,\,3,\,(2-b)-mod}$ corresponds to a rank-$2$ stable bundle of degree $d$ and speciality $h^1(\mathcal F) = 3$, which is obtained by natural modification of bundles of {\em second type} as in case $(2-b)$ of Lemma \ref{lem:barjvero}. Indeed, such a bundle $\Ff$ fits in an exact sequence of the form  
\[
0\to \omega_C(-D_s) \to\mathcal F\to \omega_C \otimes A^{\vee} \to 0, 
\] where $A \in {\rm Pic}^{\nu}(C)$ is the unique line bundle on $C$ such that $|A|= g^1_{\nu}$ and where $D_s\in C^{(s)}$ is general.

\medskip

\noindent
$(ii)$  If either $(ii_1)$  $g \geq 9$, $ 3 \leq \nu \leq \frac{g+3}{4}$ and $2g-7+3\nu \leq d \leq 3g-4-\nu$  or  $(ii_2)$ $g \geq 10$, $ 3 \leq \nu \leq \frac{g+5}{5}$ and $ 2g-7+2\nu \leq d \leq 2g-8 + 3 \nu$, there exists an irreducible component of $B_d^{k_3} \cap U_C^s(d)$, denoted by $\mathcal B_{\rm 3,\,(2-b)-mod}$, which is {\em uniruled, generically smooth} and  of dimension
$$\dim \left(\mathcal B_{\rm 3,\,(2-b)-mod}\right) = 8g-11-2\nu - 2d \geq \rho_d^{k_3},$$whose general point $[\mathcal F] \in \mathcal B_{\rm 3,\,(2-b)-mod}$ corresponds to a rank-$2$ stable bundle of degree $d$ and speciality $h^1(\mathcal F) = 3$, which fits in an exact sequence of the form  
\[
0\to N \to\mathcal F\to \omega_C \otimes A^{\vee} \to 0\, 
\] where $A \in {\rm Pic}^{\nu}(C)$ is the unique line bundle on $C$ such that $|A|= g^1_{\nu}$ whereas $N \in Pic^{d-2g+2+\nu}(C)$ is  general of its degree, so non-effective and special. 

Moreover, the component $\mathcal B_{\rm 3,\,(2-b)-mod}$ is {\em superabundant}, i.e. $\mathcal B_{\rm 3,\,(2-b)-mod} = \mathcal B_{\rm sup, 3,\,(2-b)-mod} $, if $2g-6+2\nu \leq d \leq 3g-4-\nu$,  whereas it is {\em regular}, i.e. $\mathcal B_{\rm 3,\,(2-b)-mod} = \mathcal B_{\rm reg, 3,\,(2-b)-mod} $, when $d=2g-7+2\nu$. 
\end{proposition}

\begin{proof} Let $\mathcal{F}$ be a rank-2 vector bundle of degree $d$ as in Lemma \ref{lem:barjvero}, Case (2-b); it is defined by:
\begin{equation}\label{eq:case2b}
0\to K_C-(A +B_b)\to\mathcal F\to K_C-D_s\to 0\, ,
\end{equation} where $A$ is the unique line bundle of the gonality $|A| = g^1_{\nu}$, $D_s$ is a general effective divisor of degree $s$ ($h^0(D_s) = 1$), and $B_b$ is an effective divisor of degree $b = 4g - 4 - \nu - s - d \geq 0$. This $B_b$ represents the base locus of the linear pencil $|A + B_b|$.

For $\mathcal{F}$ to be a bundle of the \textit{second type}, the kernel line bundle $N$ must satisfy the condition:
\begin{equation}
\omega_C \otimes N^{\vee} \in \overline{W^{\vec{w}{1,1}}} \subseteq W^1_{\nu + b}(C)\end{equation} This ensures the bundle corresponds to the specific locus described in Lemma \ref{lem:Lar}-(i).

Applying the necessary condition for stability, we derive the following numerical constraints:\begin{equation}
\nu + b > s \implies b > s - \nu \geq 1
\end{equation} Since $s - \nu \geq 1$ by assumption, it follows that $B_b$ is non-zero ($B_b \neq 0$). Thus, the existence of a non-trivial base locus is a fundamental requirement for the stability of $\mathcal{F}$ in this configuration.

\begin{claim}\label{cl:case2b-1} If $\mathcal F$ is a bundle of degree  $d \geq 3g-5$ and of type $(2-b)$ as in Lemma \ref{lem:barjvero}, i.e.  fitting in \eqref{eq:case2b},  then $\mathcal F$ fits also in:
\begin{equation}\label{eq:case2b-2}
0\  \to \ K_C - (B_b+D_s) \  \to \  \mathcal F\   \to  \  K_C-A\  \to \ 0.
\end{equation}
\end{claim}
\begin{proof}[Proof of Claim \ref{cl:case2b-1}] Since $\mathcal F$  fits in \eqref{eq:case2b}, we have $d = \deg(\Ff) = 4g-4 - \nu - b - s$ and, by Lemma \ref{reg_seq_1_3}-(i), $\Ff$ fits also in 
\begin{equation}\label{eq:case2b-1}
0\to K_C-(A +B_b+D_s)\to\mathcal F\to K_C\to 0\, .
\end{equation} Let $\Gamma$ be a canonical section of $F= \mathbb P(\mathcal F)$, corresponding to $\mathcal F \to\!\!\!\!\! \to  \omega_C$. Then $\dim(|\Gamma|)=2$ by  \eqref{eq:isom2} and by $h^0(\mathcal F_v-(K_C-(A+B_b+D_s)))=h^1(\mathcal F_v)=3$. Similarly as in the proof of Proposition \ref{ChoiA}, 
for $q \in F$ general, the dimension of the family of unisecants in $|\Gamma|$ passing through $q \in F$ is $1$ so $|\Gamma|$ certainly contains reducible unisecants. The section contained in any such a reducible unisecant corresponds to a quotient line bundle of $\Ff$ of the form $K_C-E_f$, where $E_f$ an effective divisor containing the point $p = \rho(q) \in C$ projection of $q$ onto the base curve $C$. 
For determinantal reasons, $\Ff$ fits also in:
\[
0\to K_C-(A+B_b+D_s-E_f)\to\mathcal F\to K_C-E_f \to 0.
\]From the assumption  $3g-5\leq d= \deg(\mathcal F)=4g-4-\nu-b-s$, it follows that $s = \deg(D_s) \leq g+1 - \nu - b$. Thus, since $D_s$ is general of its degree and since $h^0(K_C- (A +B_b))= h^1(A+B_b) = g+1-\nu-b$, by Serre duality one has $h^1(A+B_b+D_s) = h^0(K_C-(A+B_b +D_s)) = g+1-\nu-b-s$ so, by Riemann-Roch, 
$h^0(A +B_b+D_s)=2$. Since $p \in C$ is general such that $p \in E_f$, we have therefore $h^0(A+B_b+D_s-E_f)\leq 1$. 

Assume by contradiction that $h^0(A+B_b+D_s-E_f)=0$; in such a case the coboundary map \linebreak $H^0(K_C-E_f)\stackrel{\partial}{\longrightarrow} H^1(K_C-(A+B_b+D_s-E_f))$ arising from the above exact sequence would be surjective, so its dual map 
$H^1(K_C-(A+B_b+D_s-E_f))^{\vee} \stackrel{\partial^{\vee}}{\longrightarrow} H^0(K_C-E_f)^{\vee}$ would be injective. Since 
$$H^1(K_C-(A+B_b+D_s-E_f))^{\vee}\simeq H^0(A+B_b+D_s-E_f)\simeq
{\rm Hom}(K_C-(A+B_b+D_s), K_C-E_f)$$ 
\normalsize
and $$H^0(K_C-E_f)^{\vee} \simeq H^1(E_f)\simeq {\rm Ext}^1(K_C-(A+B_b+D_s), \, K_C-(A+B_b+D_s-E_f)),$$ 
we would therefore deduce that $\mathcal F$ does not admit $K_C-(A+B_b+D_s)$ as a sub-line bundle, namely that $\mathcal F$ would not fit in \eqref{eq:case2b-1} 
contradicting Lemma \ref{reg_seq_1_3}. 

Thus, $h^0(A+B_b+D_s-E_f)=1$ and $A-E_f$ is effective. By the way, since $h^1(\mathcal F)=3$, $h^1(K_C-E_f) \geq 2$, 
hence $f:=\deg (E_f)\geq \nu$ which forces $\mathcal O_C(E_f)=A$. Thus $\mathcal F$ fits in \eqref{eq:case2b-2}, as stated. 
\end{proof}

Following Claim \ref{cl:case2b-1}, when $d \geq 3g-5$, any type (2-b) bundle $\mathcal{F}$ can be realized as an extension $\mathcal{F}_u \in \text{Ext}^1(K_C-A, \, K_C-(B_b+D_s))$. While stability for Case (2-b) requires a non-trivial base locus $B_b \neq 0$ — which complicates the construction — the results of Claim \ref{cl:case2b-1} suggest a shift in strategy. Specifically, we will show that presentations of the form \eqref{eq:case2b-2} with $b=0$ generate the desired components of $B_d^{k_3} \cap U_C^s(d)$.

\begin{claim}\label{KX-A} Let $g \geq 5$, $3 \leq \nu \leq \frac{g}{2}$, $3g-5 \leq d \le 4g-5-2\nu$ and $\nu+1 \leq s \leq g-\nu+1 $ be integers and let $C$ be a general $\nu$-gonal curve. Then there exists an irreducible component of $B_d^{k_3} \cap U_C^s(d)$, denoted by $\mathcal B_{\rm sup,\,3,\,(2-b)-mod}$, which is {\em uniruled, generically smooth} and {\em superabundant}, being of dimension  \begin{equation}\label{choiD}
\dim \left(\mathcal B_{\rm sup,\,3,\,(2-b)-mod}\right) = 8g-11-2\nu - 2d > \rho_d^{k_3},
\end{equation}whose general point $[\mathcal F] \in \mathcal B_{\rm sup,\,3,\,(2-b)-mod}$ corresponds to a rank-$2$ stable bundle of degree $d$ and speciality $h^1(\mathcal F) = 3$, fitting in an exact sequence of the form 
\[
0\to \omega_C(-D_s) \to\mathcal F\to \omega_C \otimes A^{\vee} \to 0\, 
\] where $A \in {\rm Pic}^{\nu}(C)$ is the unique line bundle on $C$ such that $|A|= g^1_{\nu}$ and where $D_s\in C^{(s)}$ is general, where as above $\nu+1 \leq s \leq g-\nu+1 $. 
\end{claim}

\begin{proof}[Proof of Claim \ref{KX-A}] Set $L:= \omega_C \otimes A^{\vee}$ and $N:= \omega_C(-D_s)$, with $s$ any integer as in the statement. Therefore 
$d = \deg(\mathcal F) = 4g-4-s-\nu$ and $m:= \dim({\rm Ext}^1(L,N)) =h^1(N-L)=h^0(K_C-A+D_s)=2g-2-\nu+s-g+1=g-1+s-\nu$, since $\deg(K_C-A+D_s)=2g-2+s-\nu>2g-2$. Moreover $\ell:= h^0(L)=2g-2-\nu-g+3=g+1-\nu > 1=h^1(K_C-D_s):=r$. We see further that $m=g-1+s-\nu\geq \ell + 1 = h^0(L) + 1 =g+2-\nu$ iff $s\geq 3$. Thus, from Theorem \ref{CF5.8}, $\mathcal W_1 \subset {\rm Ext}^1(L,N) $, defined as in \eqref{W1}, is irreducible and of $\dim(\mathcal W_1)=g-1+s-\nu-(g+1-\nu)=s-2$. We also remark that, for $u \in \mathcal W_1$ general, the corresponding rank-$2$ vector bundle $\Ff_u$ has speciality $3$. Posing 
$\widehat{\mathcal W}_1 := \mathbb P(\mathcal W_1)$, we have therefore 
$\dim (\widehat{\mathcal W}_1)=4g-7-d-\nu = s-3$, the latter equality by definition of $d$. More precisely, since $r=h^1(K_C-D_s)=1$, it turns out that for $u \in \mathcal W_1 \subset \ext^1(K_C-A, K_C-D_s)$, the associated coboundary map $H^0(L) \stackrel{\partial_u}{\longrightarrow} H^1(N)$ must be the zero-map. As in the proof of Lemma \ref{Lemma1Sgen}, recall that $\partial_u$ is induced by the natural cup-product $\cup : H^0(L) \otimes H^1(N-L) \to H^1(N)$; this means that $\mathcal W_1$ is simply 
\begin{equation}\label{eq:w1_(2-b)}
{\small {\mathcal W}_1 := \left(\coker \left(H^0(K_C+D_{s}-A) \leftarrow H^0(K_C-A) \otimes H^0(D_{s})\right)\right)^{\vee},}
\end{equation} i.e. $\mathcal W_1$ turns out to be a sub-vector space of $\ext^1(K_C-A, K_C-D_s)$ hence $\widehat{\mathcal W}_1$ is a linear space of $\mathbb P:= \mathbb P({\rm Ext}^1(\omega_C \otimes A^{\vee},\; \omega_C(-D_s)))$. Since $D_s$ varies in a suitable open, dense subset $S \subseteq  C^{(s)} $, whereas $L = \omega_C \otimes A^{\vee}$ is uniquely determined, and since each corresponding linear space $\widehat{\mathcal W}_1$ is of (constant) dimension $4g-7-d-\nu$, we get an irreducible scheme, denoted by $\widehat{\mathcal W}_1^{Tot}$, which is ruled over $S$ by the linear spaces $\widehat{\mathcal W}_1$'s as $D_s$ varies. Thus $\dim(\widehat{\mathcal W}_1^{Tot}) = \dim 
(C^{(s)}) + \dim (\widehat{\mathcal W}_1) = 2s-3 = 8g-11-2d-2\nu$, where the latter equality follows from $d=4g-4-s-\nu$. Similarly, we get $\rho^{k_3}_d=10g-3d-18=-2g-6+3s+3\nu$ so from assumptions on $d$ we notice that 
\begin{equation}\label{eq:superab30dic}
\dim(\widehat{\mathcal W}_1^{Tot}) - \rho^{k_3}_d = (2s-3)-(-2g-6+3s+3\nu)=2g+3-s-3\nu = d -2g -2\nu+7 > 0.
\end{equation}For a general $[u] \in \widehat{\mathcal{W}}_1^{Tot}$, let $\mathcal{F}_u$ be the associated bundle.

If $s \geq \nu + 2$, we may apply the Segre invariant properties as in Proposition \ref{LN}. We let $\sigma= 1 ({\rm or} \; 2) \equiv 2\delta -d=s-\nu {\mbox{ (mod 2) }}$; then $\sigma \leq 2\delta-d$ because $2\delta -d \geq 2$; moreover $\sigma \geq 4+d-2\delta =4+\nu-s$ since $\nu+2\leq s$. Therefore, 
$\sigma =1$ means that $1\geq 4+\nu-s$, hence $s\geq \nu+3$ (being $s$ and $\nu$ with different parity); if otherwise $\sigma =2$, this means that $2\geq 4+\nu-s$, hence $s\geq \nu+2$ which is in our assumptions. If we set $X := \varphi_{|K_C+L-N|}(C)$ we can apply  Proposition \ref{LN}: if $\sigma =1$, then $2\delta-d=s-\nu$ is odd hence $\Sec_{\frac{1}{2}(s-\nu+1-2)}(X)=\Sec_{\frac{1}{2}(s-\nu-1)}(X)$ so
\[
\dim \left(\Sec_{\frac{1}{2}(s-\nu-1)}(X)\right)=s-\nu-2<\dim (\widehat{\mathcal W}_1)=s-3; 
\]if otherwise $\sigma =2$, then $2\delta-d=s-\nu$ is even hence $\Sec_{\frac{1}{2}(s-\nu+2-2)}(X)=\Sec_{\frac{1}{2}(s-\nu)}(X)$ so
\[
\dim \left(\Sec_{\frac{1}{2}(s-\nu)}(X)\right)=s-\nu-1<\dim (\widehat{\mathcal W}_1)=s-3.
\]If otherwise $d$ is odd and $2\delta - d = 1$, any destabilizing subbundle $\mathcal{M}$ would force the extension to split, contradicting the choice of a non-zero $u \in \mathcal{W}_1$. 

Thus $ \widehat{\mathcal W}_1^{Tot}$ is endowed with a natural (rational) modular map 
\[
 \begin{array}{ccc}
 \widehat{\mathcal W}_1^{Tot}& \stackrel{\pi_{d,\nu}}{\dashrightarrow} & U^s_C(d) \\
 (s, [u])  & \longrightarrow &[\mathcal F_u]
 \end{array}
 \] for which $\im(\pi_{d,\nu}) \subseteq B^{k_3}_d \cap U_C^s(d)$. Since $\widehat{\mathcal W}_1^{Tot}$ is ruled over $S$, it then follows that $\im(\pi_{d,\nu})$ is {\em uniruled}. 

To show that the modular map $\pi_{d,\nu}: \widehat{\mathcal{W}}_1^{Tot} \dashrightarrow U^s_C(d)$ is birational onto its image, we analyze its fibers. Assume there exists an isomorphism $\varphi: \mathcal{F}_{u'} \to \mathcal{F}_u$ for $(s', [u']) \neq (s, [u])$. Determinantal constraints force $D_{s'} = D_s$, meaning both extensions lie in the same projectivized space $\widehat{\mathcal{W}}_1$. The isomorphism $\varphi$ induces a commutative diagram:$$\begin{array}{ccccccc}
0 \to & K_C - D_s & \stackrel{\iota_1}{\longrightarrow} & \mathcal{F}_{u'} & \to & K_C- A & \to 0 \\ 
 & & & \downarrow^{\varphi} & & &  \\
0 \to & K_C - D_s & \stackrel{\iota_2}{\longrightarrow} & \mathcal{F}_u & \to & K_C- A  & \to 0.
\end{array}$$The maps $\varphi \circ \iota_1$ and $\iota_2$ define global sections $s_1, s_2 \in H^0(\mathcal{F}_u \otimes (K_C - D_s)^{\vee})$. Since $h^0(D_s - A) = 0$, the section $\Gamma \subset \mathbb{P}(\mathcal{F}_u)$ corresponding to the quotient $K_C - A$ is algebraically isolated, which implies $h^0(\mathcal{F}_u \otimes (K_C - D_s)^{\vee}) = 1$. Consequently, $s_1$ and $s_2$ must be linearly dependent. This dependency forces $[u] = [u']$, proving that the general fiber is a singleton and $\pi_{d,\nu}$ is birational.

Let $\mathcal{B}_{\text{sup, 3, (2-b)-mod}}$ be the closure of $\text{im}(\pi_{d,\nu})$. By the birationality of $\pi_{d,\nu}$ and the irreducibility of $\widehat{\mathcal{W}}_1^{Tot}$, this locus is an irreducible, uniruled, and superabundant scheme of dimension:$$ \dim(\mathcal{B}_{\text{sup, 3, (2-b)-mod}}) = 8g-11-2d-2\nu. $$To show this is a generically smooth component of $B_d^{k_3} \cap U_C^s(d)$, it suffices to prove that the dimension of the tangent space at a general $[\mathcal{F}]$ matches the dimension of the locus. This is equivalent to showing that the kernel of the Petri map $\mu_{\mathcal{F}}$ has dimension $d-2g+7-2\nu$.

Since the coboundary map $\partial_u$ vanishes for $\mathcal{F} \in \mathcal{W}_1$, we have:$$ H^0(\mathcal{F}) \simeq H^0(K_C-A) \oplus H^0(K_C-D_s) \quad \text{and} \quad H^0(\omega_C \otimes \mathcal{F}^{\vee}) \simeq H^0(A) \oplus H^0(D_s). $$The Petri map $\mu_{\mathcal{F}}$ deforms continuously from the map $\mu_{\mathcal{F}_0}$ of the splitting bundle $\mathcal{F}_0 = (K_C-D_s) \oplus (K_C-A)$. Decomposing $\mu_{\mathcal{F}_0}$ into line bundle multiplication maps:$$ \mu_{\mathcal{F}_0} = \mu_{D_s} \oplus \mu_{K_C-D_s,A} \oplus \mu_{K_C-A,D_s} \oplus \mu_{A}. $$As $h^0(D_s)=1$, the maps $\mu_{D_s}$ and $\mu_{K_C-A, D_s}$ are injective. Using the base-point-free pencil trick on the remaining components:$$ \ker(\mu_{\mathcal{F}_0}) \simeq \ker(\mu_{K_C-D_s,A}) \oplus \ker(\mu_{A}) \simeq H^0(K_C-A-D_s) \oplus H^0(K_C-2A). $$Substituting the dimensions $h^0(K_C-A-D_s) = g+1-\nu-s$ and $h^0(K_C-2A) = g+2-2\nu$:$$ \dim(\ker(\mu_{\mathcal{F}_0})) = 2g+3-3\nu-s = d-2g+7-2\nu. $$This confirms generic smoothness, which completes the proof of Claim \ref{KX-A}. 
\end{proof} 

\noindent
The above arguments complete the proof of Case $(i_1)$. 

Concerning Case $(i_2)$, the strategy from Claim \ref{cl:case2b-1} to obtain presentation \eqref{eq:case2b-2} fails because the assumption $h^0(A + B_b + D_s) = 2$ no longer holds under the range $g-\nu + 2 \leq s \leq g-1$. However, inspired by Claim \ref{KX-A}, we consider bundles $\mathcal{F}_u$ defined by extensions:
\begin{equation}\label{(cl:1-b)}
0 \to K_C-D_s \to \mathcal{F}u \to K_C-A \to 0,
\end{equation}where $u \in \mathcal{W}_1$ and $D_s$ is general. Although these bundles are neither of first nor of second type, the construction of the irreducible, uniruled component $\mathcal{B}{\mathrm{sup}, 3, (2\text{-}b)\text{-mod}}$ and its dimension estimate proceed exactly as in $(i_1)$. The only specific requirement for $(i_2)$ is verifying that the Petri map kernel maintains the dimension $\dim(\ker(\mu_{\mathcal{F}})) = d - 2g + 7 - 2\nu$, for $g-\nu + 2 \leq s \leq g-1$.

As in $(i_1)$, the vanishing of the coboundary map $\partial_u$ implies $H^0(\omega_C \otimes \mathcal{F}^\vee) \simeq H^0(A) \oplus H^0(D_s)$. This yields a diagram of multiplication maps where the injectivity of $\mu_{D_s}$ (since $h^0(D_s)=1$) leads to:
$$ \dim(\ker(\mu_{\mathcal{F}})) = \dim(\ker(\mu_{\mathcal{F}, A})) = h^0(\mathcal{F} \otimes A^\vee). $$To compute $h^0(\mathcal{F} \otimes A^\vee)$, we analyze the long exact sequence associated with the extension \eqref{(cl:1-b)} tensored by $A$ and $A^\vee$. Let $\mathcal{F}_v := \mathcal{F} \otimes A^\vee$. Tensoring by $A^\vee$ gives:
\begin{equation}\label{eq:ext_v}
0 \to K_C - D_s - A \to \mathcal{F}_v \to K_C - 2A \to 0.
\end{equation} From the speciality $h^1(\mathcal{F})=3$ and $h^0(A)=2$, we deduce $h^1(\mathcal{F} \otimes A^\vee) \geq 5$. Since $h^1(K_C - 2A) = 3$, the extension class $v$ must lie in a degeneracy locus $\widetilde{\mathcal{W}}_2 \subset \text{Ext}^1(K_C - 2A, K_C - D_s - A)$. This $v$ corresponds to a hyperplane containing the image of the restricted multiplication map:$$ \mu_2 : H^0(K_C - 2A) \otimes V_2 \to H^0(K_C + D_s - A), $$where $V_2 \subset H^0(A + D_s)$ is a 2-dimensional subspace. By applying Riemann-Roch to $\mathcal{F} \otimes A^\vee$, the computed dimensions satisfy:$$ h^0(\mathcal{F} \otimes A^\vee) = \chi(\mathcal{F} \otimes A^\vee) + h^1(\mathcal{F} \otimes A^\vee) = d - 2g + 7 - 2\nu. $$This confirms the generic smoothness of the component for the extended range of $s$, completing the proof of part ($i$).

\medskip

\noindent ({\it{ii}}) We consider vector bundles $\mathcal{F}_u$ defined by extensions
\begin{equation*}
0 \to N \to \mathcal{F}_u \to K_C-A \to 0,
\end{equation*} where $N \in \text{Pic}^{d-2g+2+\nu}(C)$ is general and $u \in \mathcal{W}_1 \subset \text{Ext}^1(K_C-A, N)$. Although these bundles are neither of first nor second type, the construction of the component $\mathcal{B}_{3,(2-b)\text{-mod}}$, its dimension estimate, and the proof of uniruledness proceed similarly as in cases $(i_1)$ and $(i_2)$ above. 

Set $L := K_C-A$; from the bounds on $\nu$, we have $\deg(2K_C-N-A) > 2g-2$, hence 
$m := \dim \text{Ext}^1(L,N) = h^0(2K_C-N-A) = 5g-5-d-2\nu$, $\ell := h^0(L) = g+1-\nu$ and $r := h^1(N) = 3g-3-\nu-d$. Therefore conditions of Theorem \ref{CF5.8} are satisfied: $\ell > r$ and $m \geq \ell+1$ follow from the bounds $2g-4 \leq d \leq 4g-7-\nu$. Consequently, $\mathcal{W}_1$ is irreducible of dimension $m - c(\ell,r,1)$, where $c(\ell,r,1) = \ell-r+1 = d-2g+5$. Thus, $\dim \mathcal{W}_1 = 7g-10-2d-2\nu$. As in the proof of Claim \ref{KX-A}, by varying $N$ over an open dense subset $S \subseteq \text{Pic}^{d-2g+2+\nu}(C)$, we obtain an irreducible scheme $\widehat{\mathcal{W}}_1^{Tot}$ dominating $S$ with fibers $\widehat{\mathcal{W}}_1 := \mathbb{P}(\mathcal{W}_1)$. Its dimension is $\dim \left( \widehat{\mathcal{W}}_1^{Tot} \right) = g + (7g-11-2d-2\nu) = 8g-11-2d-2\nu$. Comparing this with the expected dimension $\rho_d^{k_3} = 10g-3d-18$, we obtain:
\begin{equation}\label{eq:str2gen}
\dim \left( \widehat{\mathcal{W}}_1^{Tot} \right) - \rho_d^{k_3} = d - 2g - 2\nu + 7 \geq 0,
\end{equation}with equality holding iff $d = 2g-7+2\nu$.

To prove that the bundle $\Ff_u$ arising from a general pair $(N,[u]) \in \widehat{\mathcal W}_1^{Tot}$ is stable we observe that, since $d  \leq 3g-4-\nu\leq 3g-6 $, 
setting $\delta:= \deg(K_C-A) = 2g-2-\nu$ we have $2\delta-d=4g-4-2\nu -d \geq g +2-2\nu \geq 2$so, as in Proposition \ref{LN}, we may take $\sigma:= 1 \; ({\rm or} \; 2) \equiv 2\delta -d {\mbox{ (mod 2) }}$, which are compatible with $\sigma \leq 2\delta-d$ (because $2\delta -d \geq 2$) and with $\sigma \geq 4+d-2\delta =8+2\nu+d-4g$. Applying Proposition \ref{LN} to $X := \varphi_{|K_C+L-N|}(C)$ we have that when $\sigma =1$, i.e. $d$ odd, one has  
\[
\dim \left(\Sec_{\frac{1}{2}(2\delta-d-1)}(X)\right)= (4g-4-2\nu -d -1)-1<\dim (\widehat{\mathcal W}_1)  =-2d+7g-2\nu-11
\]iff $d  < 3g-5$, which certainly holds in our bounds on $d$, when otherwise $\sigma =2$, i.e. $d$ even, 
\[
\dim\left(\Sec_{\frac{1}{2}(2\delta-d)}(X)\right)=(4g-4-2\nu -d )-1 <\dim (\widehat{\mathcal W}_1)  =-2d+7g-2\nu-11 
\]iff $d < 3g-6$  which is compatible with the bounds on $d$. This shows that the general bundle $\mathcal F_u$ is stable.

The above arguments ensure that $\widehat{\mathcal W}_1^{Tot}$ is endowed with a natural (rational) modular map 
\[
 \begin{array}{ccc}
 \widehat{\mathcal W}_1^{Tot}& \stackrel{\pi_{d,\nu}}{\dashrightarrow} & U^s_C(d) \\
 (N, [u])  & \longrightarrow &[\mathcal F_u]
 \end{array}
 \] for which $\im(\pi_{d,\nu}) \subseteq B^{k_3}_d \cap U_C^s(d)$.

 \begin{claim}\label{cl:uniratnov25} In the above set-up, $\im(\pi_{d,\nu})$ is {\em uniruled} and such that 
\begin{equation}\label{claim I3-1}
\dim(\im(\pi_{d,\nu})) = 8g-11-2d-2\nu.
\end{equation}
 \end{claim}
 
 \begin{proof}[Proof of Claim \ref{cl:uniratnov25}] Since $\text{im}(\pi_{d,\nu})$ is dominated by $\widehat{\mathcal{W}}_1^{Tot}$, it suffices to show that $\widehat{\mathcal{W}}_1^{Tot}$ is uniruled over a dense open $S \subseteq \text{Pic}^{d-2g+2+\nu}(C)$. This follows if each fiber $\widehat{\mathcal{W}}_1 \subset \mathbb{P}(\text{Ext}^1(K_C-A, N))$ is unirational for $N \in S$. Recall that the coboundary maps defining $\mathcal{W}_1$ are given by the cup-product $\cup: H^0(K_C-A) \otimes H^1(N-K_C+A) \to H^1(N)$. By Serre duality, this is equivalent to the multiplication map: $$ \mu : H^0(K_C-A) \otimes H^0(K_C-N) \to H^0(2K_C-A-N). $$For any $1$-dimensional subspace $V_1 \subset H^0(K_C-N)$, the restriction $\mu|_{V_1} : H^0(K_C-A) \otimes V_1 \to H^0(2K_C-N-A)$ is injective via the sequence $0 \to K_C-A \to 2K_C-A-N$ induced by $V_1$. Since $m > \ell + 1$, the image $\text{im}(\mu|_{V_1})$ has dimension $\ell < m-1$ and therefore is contained in a hyperplane of $H^0(2K_C-A-N)$. Let $\Lambda \cong \mathbb{P}(H^0(K_C-N))$ be the variety of such subspaces $V_1^\lambda$, and let $\mathbb{P} = \mathbb{P}(H^0(2K_C-A-N))$. We define the incidence variety:$$ \mathfrak{I} := \{ (\lambda, \pi) \in \Lambda \times \mathbb{P} \mid \mu(\lambda \otimes V_1^\lambda) \subseteq \pi \}. $$The projection $pr_1: \mathfrak{I} \to \Lambda$ is surjective with fibers $pr_1^{-1}(\lambda) \cong \mathbb{P}^{m-1-\ell}$. Thus, $\mathfrak{I}$ is a projective bundle over the rational variety $\Lambda$, making $\mathfrak{I}$ rational. Since $\widehat{\mathcal{W}}_1 = \overline{pr_2(\mathfrak{I})}$, it is unirational. Consequently, $\widehat{\mathcal{W}}_1^{Tot}$ is uniruled over $S$, and so is $\text{im}(\pi_{d,\nu})$.

To establish \eqref{claim I3-1}, the birationality of $\pi_{d,\nu}$ onto its image follows the same strategy as in the proof of Claim \ref{cl:piddeltanugenfin_3_1-a}, yielding $\dim(\text{im}(\pi_{d,\nu})) = \dim(\widehat{\mathcal{W}}_1^{Tot}) = 8g-11-2d-2\nu$.\end{proof}

\begin{claim}\label{claimI32} With above assumptions, let $N \in \pic^{d-2g+2+\nu}(C)$ be general and let $\Ff_u$ be such that  it corresponds to a general element $u \in \mathcal W_1 \subset {\rm Ext}^1(K_C-A, N)$. Let $F_u$ denote the surface scroll associated to $\Ff_u$ and let $\Gamma \subset F_u$ 
be the section corresponding to the quotient line bundle $K_C-A$ as in 
\begin{equation}\label{eq:19-12cogl}
0 \to N \to \Ff_u \to K_C-A \to 0.
\end{equation} Then $\Gamma$ is linearly isolated (li) on $F_u$ in the sense of Definition \ref{def:ass0}, i.e. $h^0(\Ff_u \otimes N^{\vee})=1$. 
\end{claim}

\begin{proof}[Proof of Claim \ref{claimI32}]The two assertions are equivalent by \eqref{eq:isom2}; we thus show $h^0(\mathcal{F}_u \otimes N^{\vee})=1$. Tensoring \eqref{eq:19-12cogl} by $A$ yields:\begin{equation}\label{eq:22-12cogl}0 \to N + A \to \mathcal{F}_u \otimes A \to K_C \to 0.\end{equation}The extension space $\text{Ext}^1(K_C, N + A)$ is isomorphic to $\text{Ext}^1(K_C-A, N) \cong H^1(N+A-K_C)$. Let$$ \widetilde{\mathcal{W}}_1 := \{ w \in \text{Ext}^1(K_C, N + A) \mid \text{corank}(\partial_w) \geq 1 \}. $$By Theorem \ref{CF5.8}, $\widetilde{\mathcal{W}}_1$ is irreducible with expected codimension $\widetilde{c}(1) = h^0(K_C) - h^1(N+A) + 1$. A direct calculation shows that $\text{codim}(\widetilde{\mathcal{W}}_1) = \widetilde{c}(1) > c(1) = \text{codim}(\mathcal{W}_1)$. Consequently, $\dim(\widetilde{\mathcal{W}}_1) < \dim(\mathcal{W}_1)$. 

For $u \in \mathcal{W}_1$ general, $\mathcal{F}_v := \mathcal{F}_u \otimes A$ corresponds to an element $v \notin \widetilde{\mathcal{W}}_1$. Thus, the coboundary map $\partial_v: H^0(K_C) \to H^1(N+A)$ is surjective, implying $h^1(\mathcal{F}_u \otimes A) = h^1(K_C) = 1$. Since $\text{rank}(\mathcal{F}_u)=2$, $\det(\mathcal{F}_u) = K_C + N - A$, Serre duality gives $ 1 = h^1(\mathcal{F}_u \otimes A) = h^0(\omega_C \otimes A^{\vee} \otimes \mathcal{F}_u^{\vee}) = h^0(\mathcal{F}_u \otimes N^{\vee}).$
\end{proof}

Let $\mathcal{Z}_{d,\nu}$ denote the closure of $\text{im}(\pi_{d,\nu})$ in $B_d^{k_3} \cap U^s_C(d)$. By the previous construction and the birationality of $\pi_{d,\nu}$, $\mathcal{Z}_{d,\nu}$ is a uniruled variety of dimension $8g-11-2d-2\nu$. It follows from \eqref{eq:str2gen} that $\dim \mathcal{Z}_{d,\nu} > \rho_d^{k_3}$ for $d \geq 2g-6+2\nu$, and $\dim \mathcal{Z}_{d,\nu} = \rho_d^{k_3}$ iff $d = 2g-7+2\nu$. Consequently, for $d \geq 2g-6+2\nu$, any irreducible component $\mathcal{B}$ containing $\mathcal{Z}_{d,\nu}$ is necessarily superabundant. 

To establish that $\mathcal{Z}_{d,\nu}$ is itself a component, we must show that for a general $[\mathcal{F}] \in \mathcal{Z}_{d,\nu}$, the dimension of the Zariski tangent space satisfies:\begin{equation*}
\dim\left(T_{[\mathcal{F}]}(B^{k_3}d \cap U^s_C(d)) \right) = 8g-11-2d-2\nu.
\end{equation*}This would identify $\mathcal{Z}{d,\nu}$ as a generically smooth component (and a regular one when $d=2g-7+2\nu$). 

Recalling Remark \ref{rem:BNloci}, the tangent space is isomorphic to $\text{coker}(\mu_{\mathcal{F}})$, where $\mu_{\mathcal{F}}$ is the Petri map of $\Ff$. Thus, since $\Ff$ is stable
\begin{equation*}
\dim\left(T_{[\mathcal{F}]}(B^{k_3}d \cap U^s_C(d)) \right) = \rho_d^{k_3} + \dim \ker \mu{\mathcal{F}}.
\end{equation*} Comparing this with \eqref{eq:str2gen}, the claim reduces to proving:\begin{equation*}
\dim( \ker (\mu_{\mathcal{F}})) = d-2g+7-2\nu.
\end{equation*} Verifying this dimension is where the treatments of cases $(ii_1)$ and $(ii_2)$ diverge. We begin by addressing case $(ii_1)$.

\begin{claim}\label{claimI5} Let  $g \geq 9$, $3 \leq \nu \leq \frac{g+3}{4}$ and $2g-7+3\nu \leq d \leq 3g-4-\nu$ be integers and let $C$ be a general $\nu$-gonal curve. Then, for $[\Ff] \in \mathcal{Z}_{d,\nu}$ general as above, one has $h^1(\Ff \otimes A^{\vee}) = 5$.

In particular, $\dim(\ker(\mu_{\Ff}) ) = d-2g+7-2\nu$ so $\mathcal{Z}_{d,\nu}$ is an irreducible component of $B_d^{k_3} \cap U_C^s(d)$, which will be denoted by $\mathcal B_{\rm sup, 3, (2-b)-mod}$ and which is {\em generically smooth}, {\em uniruled} and {\em superabundant}. 
\end{claim}

\begin{proof}[Proof of Claim \ref{claimI5}] A general $[\mathcal{F}] \in \mathcal{Z}_{d,\nu}$ is given by $\mathcal{F} = \mathcal{F}_u$ for $u \in \mathcal{W}_1 \subset \text{Ext}^1(K_C-A, N)$ and $N \in \text{Pic}^{d-2g+2+\nu}(C)$ general. By Claim \ref{claimI32}, $h^1(\mathcal{F} \otimes A) = 1$. Consider the exact sequence:
{\footnotesize
\begin{equation*}
0\to H^0(\mathcal{F}\otimes A^\vee) \to H^0(\mathcal{F})\otimes H^0(A) \xrightarrow{\mu_{\mathcal{F},A}} H^0(\mathcal{F}\otimes A) \to H^1(\mathcal{F}\otimes A^\vee) \to H^1(\mathcal{F})\otimes H^0(A) \to H^1(\mathcal{F}\otimes A)\to 0.
\end{equation*}} With $h^1(\mathcal{F})=3$, $h^0(A)=2$, and $h^1(\mathcal{F}\otimes A)=1$, we have $h^1(\mathcal{F} \otimes A^\vee) \geq 5$. Tensoring \eqref{eq:19-12cogl} by $A^\vee$ gives:
\begin{equation}\label{(**)}
0 \to N-A \to \mathcal{F} \otimes A^\vee \to K_C-2A \to 0.
\end{equation} The bundle $\mathcal{F} \otimes A^\vee$ corresponds to $v \in \widetilde{\mathcal{W}}_2 = \{ w \in \text{Ext}^1(K_C-2A, N-A) \mid \text{corank}(\partial_w) \geq 2 \}$ so to a hyperplane in $H^0(2K_C-N-A)$ containing the image of:
\begin{equation}\label{eq:mu29gencogl}
\mu_2: H^0(K_C -2A) \otimes V_2 \to H^0 (2K_C-N-A),
\end{equation}where $V_2 = H^0(A) \otimes V_1 \subseteq H^0(A + K_C-N)$ and $V_1 \subset H^0(K_C-N)$ a $1$-dimensional subspace. We will show that $\text{coker}(\mu_2) \cong H^0(A)$. 

If by contradiction $\dim(\text{coker}(\mu_2)) > 2$, let $V_3$ be a $3$-dimensional subspace of $\text{coker}(\mu_2)$ containing $V_2$. For such a general $V_3$, the associated linear series $|V_3|$ is base-point-free since $|A|$ is. Let $V'_2 \subset V_3$ correspond to a general pencil. The restriction $\mu_{V'_2}$ is injective, and $\dim(\text{Im}(\mu_{V_3})) \geq \dim(\text{Im}(\mu_{V'_2})) = 2g-4\nu+4$. The dimension of the family of hyperplanes containing $\text{Im}(\mu_{V_3})$ is:$$ h^0(2K_C-N-A) - \dim(\text{Im}(\mu_{V_3})) + \dim(\{V_3\}) \leq h^0(2K_C-N-A) - (2g-4\nu+4) + (3g-5-d). $$This is strictly less than $h^0(2K_C-N-A) - (g-\nu+1)$ whenever $d > 2g+3\nu-8$. Under this assumption, $\text{corank}(\mu_2)=2$, which implies $h^1(\mathcal{F} \otimes A^\vee) = 5$. The surjectivity of $\mu_{\mathcal{F},A}$ then implies:$$ h^0(\mathcal{F} \otimes A^\vee) = 2h^0(\mathcal{F}) - h^0(\mathcal{F} \otimes A). $$Replacing $h^0(\mathcal{F}) = d-2g+5$ and $h^0(\mathcal{F} \otimes A) = d+2\nu-2g+3$ (via Riemann-Roch), we obtain:$$ \dim \ker(\mu_{\mathcal{F}}) = h^0(\mathcal{F} \otimes A^\vee) = 2(d-2g+5) - (d+2\nu-2g+3) = d-2g+7-2\nu. $$Thus, $[\mathcal{F}]$ is a smooth point of an irreducible, generically smooth component of $B_d^{k_3} \cap U_C^s(d)$. Uniruledness and superabundance follow from \eqref{eq:str2gen} and Claim \ref{claimI32}.
\end{proof}

\noindent
The above arguments complete the proof of $(ii_1)$. 

To conclude the proof of part $(ii_2)$, it remains to establish an analogous result to Claim \ref{claimI5}, but now in the range $2g-7+2\nu \leq d \leq 2g-8+3\nu$.

\begin{claim}\label{claimFgen10} Let $g \geq 10$, $3 \leq \nu \leq \frac{g+5}{5}$ and $2g-7+2\nu \leq d \leq 2g-8+3\nu$ be integers and let $C$ be a general $\nu$-gonal curve of genus $g$. 
Then, for $[\Ff] \in \mathcal{Z}_{d,\nu}$ general as above, one has $h^1(\Ff \otimes A^{\vee}) = 5$. 

In particular, $\dim(\ker(\mu_{\Ff}) ) = d-2g+7-2\nu$ so $\mathcal{Z}_{d,\nu}$ is an irreducible component of $B_d^{k_3} \cap U_C^s(d)$, which will be denoted by $\mathcal B_{\rm 3, (2-b)}$ and which is moreover {\em generically smooth}, {\em uniruled} and of dimension 
$$\dim(\mathcal B_{\rm 3, (2-b)-mod}) = 8g-11-2d-2\nu \geq \rho_d^{k_3}.$$
Furthermore, such a component is {\em superabundant}, i.e. $\mathcal B_{\rm 3, (2-b)-mod} = \mathcal B_{\rm sup, 3, (2-b)-mod}$, when $d$ is such that $2g-6+2\nu \leq d \leq 2g-8 + 3 \nu$,  whereas it is {\em regular}, i.e.  
$\mathcal B_{\rm 3, (2-b)-mod} = \mathcal B_{\rm reg, 3, (2-b)-mod}$, if $d = 2g-7+2\nu$.  
\end{claim}

\begin{proof}[Proof of Claim \ref{claimFgen10}] As in the previous case, we will show $h^1(\mathcal{F} \otimes A^\vee) = 5$ for a general $[\mathcal{F}] \in \mathcal{Z}_{d,\nu}$. 

Setting $\mathcal{E} := \omega_C \otimes A \otimes \mathcal{F}^\vee$, Serre duality implies $h^1(\mathcal{F} \otimes A^\vee) = h^0(\mathcal{E})$. From \eqref{eq:utiligen10}, we have $h^0(\mathcal{E} \otimes A^\vee) = 3$ and $h^0(\mathcal{E} \otimes \mathcal{O}_C(-2A)) = 1$. The multiplication map
\begin{equation}\label{eq:rankmuFgen}
0 \to H^0(\mathcal{E} \otimes \mathcal{O}_C(-2A)) \to H^0(A) \otimes H^0(\mathcal{E} \otimes A^\vee) \stackrel{\mu}{\longrightarrow} H^0(\mathcal{E})
\end{equation} has rank $h^0(A)h^0(\mathcal{E} \otimes A^\vee) - h^0(\mathcal{E} \otimes \mathcal{O}_C(-2A)) = 2(3) - 1 = 5$. Thus, proving $h^0(\mathcal{E}) = 5$ reduces to showing $\mu$ is surjective.

Consider the following commutative diagram of multiplication maps:

{\footnotesize
 \begin{equation*}
\begin{array}{ccccccccccccccccccccccc}
&&0&&0&&0&&0&&\\[1ex]
  &     &\downarrow                           &                    &\downarrow                          &                                                    &\downarrow                  &&\downarrow && \\[1ex]
0&\rightarrow& H^0(\mathcal O_C)  & \rightarrow & H^0(\mathcal E \otimes \mathcal O_C(-2A))  & \stackrel{\zeta_1}{\longrightarrow} &  H^0(K_C-N-A) & \rightarrow  & \coker({\zeta_1})  &  \rightarrow &  0 \\ [1ex]
&&\downarrow && \downarrow && \downarrow&  &\downarrow&&\\[1ex]
0&\lra &H^0(A) \otimes H^0(A)& \rightarrow & H^0(A)\otimes H^0(\mathcal E \otimes A^{\vee}) & \stackrel{\zeta_2}{\longrightarrow} &
H^0(A) \otimes H^0(K_C-N) &  \rightarrow &   \coker({\zeta_2}) & \rightarrow &  0 \\[1ex]
&&\;\;\; \;\;\; \downarrow^{\mu_{A,A}} &&\;\; \downarrow^{\mu} &&\;\;\;\; \;\;\; \;\;\;\;\;\; \downarrow^{\mu_{A,K_C-N}}  && \downarrow^{\mu_{\coker}} &&\\[1ex]
0&\lra &  H^0(2A) & \lra & H^0(\mathcal E) & \stackrel{\zeta_3}{\longrightarrow}  & H^0(K_C+A-N) &  \rightarrow & \coker({\zeta_3}) & \rightarrow  &  0
\end{array}
\end{equation*}
} Since $h^0(2A)=3$ and $\dim \ker \mu_{A,A}=1$, then $\mu_{A,A}$ is surjective. By the assumptions on $d$ and $\nu$, $K_C-N-A$ and its tensor products by $A$ and $2A$ are non-special. Consequently, $\text{rk}(\mu_{A,K_C-N}) = 2h^0(K_C-N) - h^0(K_C-N-A) = h^0(K_C-N+A)$, making $\mu_{A,K_C-N}$ and $\mu_{\coker}$ surjective. A dimension count on the cokernels yields $\text{rk}(\zeta_3)=2$, which confirms $h^0(\mathcal{E})=5$.

It follows that $\mu_{\mathcal{F},A}$ is surjective; thus, using $h^0(\mathcal{F}) = d-2g+5$ and $h^0(\mathcal{F} \otimes A) = d+2\nu-2g+3$, one can compute:
\begin{equation*}
\dim(\ker(\mu_{\mathcal{F}})) = h^0(\mathcal{F} \otimes A^{\vee}) = 2(d-2g+5) - (d+2\nu-2g+3) = d-2g+7-2\nu.
\end{equation*} This identifies $[\mathcal{F}]$ as a smooth point of a generically smooth irreducible component of $B_d^{k_3} \cap U_C^s(d)$ which, by construction, is uniruled moreover superabundant for $d \geq 2g-6+2\nu$ whereas regular for $d=2g-7+2\nu$.
\end{proof}

The above arguments conclude also the proof of $(ii_2)$ and so of the Proposition. 
\end{proof}

\subsection{First-type superabundant loci in components from second-type modifications} Regarding potential components consisting of first-type bundles from Case $(1-b)$ of Lemma \ref{lem:barjvero}, recall that Propositions \ref{prop:nocases}-$(i)$ and \ref{prop:17mar12.52}-(1) show these do not form irreducible components of $B_d^{k_3} \cap U^s_C(d)$ in the ranges: $3 \leq \nu \leq \frac{g}{3}$ and $\frac{10}{3} g - 5 \leq d \leq 4g-5-2\nu$; $3 \leq \nu \leq d \leq \frac{g}{2}$ and $3g-6-\nu \leq d \leq \frac{10}{3} g - 6$. For smaller values of $d$, we have the following preliminary result:

\begin{lemma}\label{lem:sup3case1-b}  Let $C$ be a general $\nu$-gonal curve of genus $g$ with $\nu\geq 3$. If $\mathcal B \subseteq B_d^{k_3} \cap U^s_C(d)$ is an irreducible locus such that 
$\dim(\mathcal B) \geq \rho_d^{k_3} = 10g-18-3d$ and whose general point $[\mathcal F] \in \mathcal B$ is assumed to correspond to a rank-$2$, degree $d$ stable bundle $\mathcal F$ of speciality $i = h^1(\Ff) =3$ as in case $(1-b)$ of Proposition \ref{prop:17mar12.52}-$(1-b)$, then necessarily one has
\begin{equation}\label{eq:19-12cogl2}
3 \leq \nu \leq \frac{g-4}{3}\ {\mbox{ and }} \ \frac{5g-10+\nu}{2} \leq d \leq 3g-7-\nu.
\end{equation}

Furthermore, for any $d$ as in \eqref{eq:19-12cogl2} and for any integer 
$\frac{g+2-\nu}{2} \leq n \leq g-2\nu-1$, there actually exist such irreducible loci $\mathcal B_{\rm 3,n}:= \mathcal B \subset B^{k_3}_{d} \cap U_C^s(d)$ whose general point $[\mathcal F] \in \mathcal B_{\rm 3,n}$ corresponds to a bundle $\mathcal F$ presented as a general extension of the form $$0\to N\to \mathcal F\to  K_C - (A+M) \to 0,$$where $A$ the unique line bundle on $C$ such that $|A| = g^1_{\nu}$, $M$ an effective divisor of movable points on $C$ of degree $n:= \deg(M)$, for some $\frac{g+2-\nu}{2} \leq n \leq g-2\nu-1$ as above,  and such that $A+M \in \overline{W^{\vec{w}_{2,1}}} \subset W^2_{\nu+n}(C)$ is general in the component $ \overline{W^{\vec{w}_{2,1}}}$ as in Lemma \ref{lem:w2tLar}, i.e. $|A+M|= g^2_{\nu+n}$ is base-point-free and birationally very ample, and where furthermore $N$ is general of its degree, such that $g+\nu-1\leq \deg(N) \leq g+n-5$. 
\end{lemma}

\begin{proof} We consider case $(1\text{-}b)$ from Proposition \ref{prop:17mar12.52}, with $2g-2 \leq d \leq 3g-7-\nu$. We construct irreducible loci $\mathcal{B} = \mathcal{B}_{3,n} \subset B^{k_3}_d \cap U^s_C (d)$ whose general point $[\mathcal{F}_u]$ is defined by:
$$(u):\;\; 0 \to N \to \mathcal{F}_u \to K_C - (A + M) \to 0,$$where $\delta := \deg(K_C-(A+M))$, $|A| = g^1_{\nu}$, and $M$ is an effective divisor of degree $n$ such that $|A+M|$ is a base-point-free $g^2_{\nu+n}$ on $C$. 

We take $A+M$ general in the component $\overline{W^{\vec{w}_{2,1}}} \subset W^2_{\nu+n}(C)$. Setting $n = \deg(M) = 2g-2-\delta-\nu$, the condition $n \geq \frac{g+2-\nu}{2}$ follows from Lemma \ref{lem:barjvero}-(1-b), which requires $\delta \leq \frac{3g-6-\nu}{2}$ for the non-emptiness and expected dimension of the locus $\overline{W^{\vec{w}_{2,1}}}$. Furthermore, $h^1(K_C - A - M) = h^0(A+M) = 3$, so $h^0(K_C-A-M) = \delta - g + 4 = g+2-\nu-n$. For this to be effective, we must have $n \leq g+1-\nu$. Summarizing the bounds:
$$\frac{g+2-\nu}{2} \leq n \leq g+1-\nu\;\; {\rm and} \;\; g \leq \delta \leq \frac{3g-6-\nu}{2}.$$Under such numerical constraints, any locus of dimension at least $\rho_d^{k_3}$ necessarily satisfies \eqref{eq:19-12cogl2}. 

To explicitly construct these stable bundles, we proceed as follows:

\begin{claim}\label{lem:i=2.2.3_1-b} In the above assumptions for $d$, $n$, $\nu$ and $C$, let $\mathcal F_u$ be a rank-$2$ vector bundle arising as a general extension $u\in {\rm Ext}^1(K_C-(A+M), N)$, where $N \in {\rm Pic}^{d-2g+2 +\nu+n}(C)$ is general. Then  $\mathcal F_u$ is stable and such that $h^1(\mathcal F_u) = h^1(K_C-(A+M)) = 3$. 
\end{claim}
\begin{proof}[Proof of Claim \ref{lem:i=2.2.3_1-b}] Set $L := K_C - (A + M)$, with $\deg(L) = \delta$. Let $N \in \text{Pic}^{d-\delta}(C)$ be general. The stability condition $2\delta - d > 0$ implies $\deg(N) < \deg(L)$, so $N \not\cong L$. Thus, by \eqref{eq:m}, $ m := \dim (\text{Ext}^1(L,N)) = 2\delta - d + g - 1$. Let $\ell := h^0(L) = \delta - g + 4 = g+2-\nu-n$. Under the stated bounds on $\nu$ and $n$, the condition $m \geq \ell+1$ from Theorem \ref{CF5.8} is satisfied since $2\delta - d \geq 4 - \nu - n$. Let $r := h^1(N)$; if $\deg(N) < g-1$, then $h^0(N)=0$ and $r = 3g-3-\nu-n-d > 0$. In this case, $\ell \geq r$ holds because $d \geq 2g-2$. If otherwise $\deg(N) \geq g-1$, then $r=0$ and $\ell \geq r$ holds trivially. In both cases, Theorem \ref{CF5.8} ensures the surjectivity of the coboundary map $\partial_u$ for a general $u \in \text{Ext}^1(L,N)$, hence $h^1(\mathcal{F}_u) = 3$. 

Concerning the stability of $\mathcal{F}_u$, we distinguish two cases. If $2\delta - d \geq 2$, we may apply Proposition \ref{LN} to $X \subset \mathbb{P}(\text{Ext}^1(L,N))$ being the image of $C$ via the morphism associated to $|K_C+L-N|$. Since $\dim \text{Sec}_{\frac{1}{2}(2\delta-d+\sigma-2)}(X) < m-1 = \dim(\mathbb{P}(\text{Ext}^1(L,N)))$ for $\sigma \leq g$, general $\mathcal{F}_u$ is certainly stable.

If $2\delta - d = 1$, hence $d$ is odd, $\delta = (d+1)/2$, and $m=g$; in such a case, if $\mathcal{F}_u$ were unstable, there would exist a sub-line bundle $\mathcal{M} \subset \mathcal{F}_u$ with $\deg(\mathcal{M}) \geq (d+1)/2$. Since $\deg(\mathcal{M}) > \deg(N)$, the composition $\mathcal{M} \to \mathcal{F}_u \to L$ is non-zero. As $\deg(\mathcal{M}) \geq \deg(L)$, this map would be an isomorphism, implying the sequence giving rise to $\Ff_u$ being spliting. This contradicts the generality of $u \in \text{Ext}^1(L,N) \cong H^1(N-L)$. 
Thus, $\mathcal{F}_u$ is stable also in this case, completing the proof of the Claim.
\end{proof}

From the above arguments, for any $n$ and $\nu$ as in the statement of the Lemma, one can consider a vector bundle $\mathcal E_{d,\nu,n}$ on a suitable open, dense subset $S \subseteq 
\pic^{d-2g+2+\nu+n}(C) \times  \overline{W^{\vec{w}_{2,1}}} \subset \pic^{d-2g+2+\nu+n}(C)\times  W^2_{2g-2-\delta}(C)$, filled-up by pairs $(N , A(M))$ varying, whose rank is  
$ m= \dim (\ext^1(K_C-A-M,\, N))= 2\delta - d + g - 1 = 5g-5- d - 2 \nu - 2n $ as in \eqref{eq:m}. Taking the associated projective bundle $\mathbb P(\mathcal E_{d,\nu,n})\to S$ (consisting of family of $\mathbb P\left(\ext^1(K_C-(A+M), N)\right)$'s as $(N ,\, A+M)$ varies in $S$) and since $\dim(\overline{W^{\vec{w}_{2,1}}}) = 3g-6-2\delta-\nu = 2n+\nu - g - 2$ by Lemma \ref{lem:w2tLar}, one has$$ \dim (\mathbb P(\mathcal E_{d,\nu, n})) = g+ (5g-6- d - 2 \nu - 2n) + (2n + \nu - g -2) = 5g-8-d - \nu,$$the latter being bigger than $2g-2$ since, in any case, from above we must have $d \leq 3g-7-\nu$. From Claim \ref{lem:i=2.2.3_1-b}, for any $d$, $\nu$ and $n$ as in Lemma \ref{lem:sup3case1-b}, one has therefore a natural (rational) modular map
 \begin{eqnarray*}
 &\mathbb P(\mathcal E_{d,\nu,n})\stackrel{\pi_{d,\nu,n}}{\dashrightarrow} &U_C(d) \\
 &(N, A+M, [u])\to &\mathcal [F_u],
 \end{eqnarray*} which  gives $ \mathcal Z_{d,\nu,n} := \im (\pi_{d,\nu, n})\subseteq B^{k_3}_d \cap  U^s_C(d)$. If in particular $\dim(\overline{W^{\vec{w}_{2,1}}}) = 3g-6-2\delta-\nu = 2n+\nu - g - 2=0$, then $ \dim \mathbb P(\mathcal E_{d,\nu, n})$
has several components.

\begin{claim}\label{cl:piddeltanugenfin_3_1-b} For any integers $d$, $\nu$ and $n$ as in Lemma \ref{lem:sup3case1-b} and for $C$ a general $\nu$-gonal curve, the map $\pi_{d,\nu,n}$ is {\em generically finite} onto its image. In particular, $\mathcal Z_{d,\nu,n}$ is {\em uniruled} and $\dim(\mathcal Z_{d,\nu,n}) = \dim(\mathbb P(\mathcal E_{d,\nu, n})) = 5g-8-d - \nu$. 
\end{claim}

\begin{proof}[Proof of Claim \ref{cl:piddeltanugenfin_3_1-b}] The uniruledness of $\mathcal{Z}_{d,\nu,n}$ follows from the fact that it is dominated by the projective bundle $\mathbb{P}(\mathcal{E}_{d,\nu,n})$ over the base $S$. To determine its dimension, we prove that the modular map $\pi_{d,\nu,n}$ is generically finite. 

Let $[\mathcal{F}] = [\mathcal{F}_u]$ be a general point in $\mathcal Z_{d,\nu,n}$. We examine the fiber:$$ \pi^{-1}([\mathcal{F}_u]) = \{ (N', A+M', [u']) \in \mathbb{P}(\mathcal{E}_{d,\nu,n}) \mid \mathcal{F}_{u'} \cong \mathcal{F}_u \}. $$ If $A+M' \cong A+M$, then $N' \cong N$ because the determinants are isomorphic. If $[u'] \neq [u]$, the stability of $\mathcal{F}_u$ implies any isomorphism $\varphi: \mathcal{F}_{u'} \to \mathcal{F}_u$ yields two sections $s_1, s_2 \in H^0(\mathcal{F}_u \otimes N^\vee)$. Since $\deg(L-N) = 2\delta - d \leq g-1$ and $N$ is general, $h^0(L-N)=0$. By \eqref{eq:Ciro410}, the section $\Gamma \subset \mathbb{P}(\mathcal{F}_u)$ is algebraically isolated, implying $h^0(\mathcal{F}_u \otimes N^\vee) = 1$. Thus, $s_1$ and $s_2$ must be linearly dependent, which leads to the contradiction $[u'] = [u]$ as in Claim \ref{cl:piddeltanugenfin_3_1-a}.

If otherwise $A+M' \not\cong A+M$, let $L' := K_C - (A+M')$. The determinant condition requires $N' \cong N + M' - M$. We have therefore two distinct saturated sub-line bundles $N$ and $N'$ of $\mathcal{F}_u$. By the generality of $N$, the degrees $\deg(L-N)$ and $\deg(L'-N')$ are such that the corresponding sections in the ruled surface $F_u= \mathbb{P}(\mathcal{F}_u)$ are algebraically isolated. Consequently, the set of such sub-line bundles, denoted by $\text{Div}^{1,\delta}_{F_u}$, consists of only finitely many reduced points. 

In both cases, the fiber $\pi^{-1}([\mathcal{F}_u])$ is finite, proving that $\pi_{d,\nu,n}$ is generically finite onto its image.
\end{proof}

Let $\mathcal B_{\rm 3,n}$ denote the closure in $B_d^{k_3}\cap U^s_C(d)$ of $\mathcal{Z}_{d,\nu,n}$. Then  $\mathcal B_{\rm 3,n}$ is {\em uniruled} of dimension $\dim (\mathcal B_{\rm 3,n}) = 5g-8-d - \nu$, as $\mathcal{Z}_{d,\nu,n}$ is dense in it. Notice first that, for any $\frac{g+2-\nu}{2} \leq n \leq g+1-\nu$ as above, the dimension of the locus $\mathcal B_{\rm 3,n}$ does not depend on $n = \deg(M)$; furthermore notice that  
$$\dim (\mathcal B_{\rm 3,n}) = 5g-8-d-\nu \geq \rho_d^{k_3} = 10g-18-3d \;\; \Leftrightarrow \;\; d \geq \frac{5g-10+\nu}{2},$$ proving the lower-bound on $d$ as in \eqref{eq:19-12cogl2} (recall that the upper-bound is given by what proved for case $(1-b)$ in Proposition \ref{prop:nocases} and \ref{prop:17mar12.52}). Therefore we get $3\leq \nu \leq \frac{g-4}{3}$ in \eqref{eq:19-12cogl2} which proves the Lemma.
\end{proof}

\medskip

\noindent
{$\boxed{\mbox{Further obstructions: no component from Case $(1-b)$ as} \;  \mathcal B_{\rm 3,n} \subsetneq \mathcal B_{\rm 3, (2-b)-mod}}$}  Once the irreducible loci $\mathcal B_{\rm 3,n}$ have been constructed as in Lemma \ref{lem:sup3case1-b}, we need to control the Zariski tangent space of $B_d^{k_3} \cap U_C^s(d)$ at a general point $[\Ff] \in \mathcal B_{\rm 3,n}$. As we will see below, although any $\mathcal B_{\rm 3,n}$ is a superabundant, irreducible uniruled locus in  $B_d^{k_3} \cap U_C^s(d)$ whose dimension is independent of $n$, we will show that any such locus is contained, as a locally closed subset, in an irreducible component of the form $\mathcal B_{\rm 3, (2-b)-mod}$. To do this, we start with:

\begin{claim}\label{ChoiJ} For $[\Ff] \in \mathcal B_{\rm 3,n}$ general as above, the Petri map $\mu_{\Ff}$ is such that $\dim(\ker(\mu_{\Ff}) ) = d-2\nu-2g+7$. In particular, we have
\[
\dim (T_{[\mathcal F]}(B^{k_3}_d \cap U^s_C(d))) =8g-2d-2\nu-11 > \dim (\mathcal B_{\rm 3,n}) = 5g-8-d-\nu . 
\]
\end{claim}

\begin{proof}[Proof of Claim \ref{ChoiJ}] Recall that $[\Ff] \in \mathcal B_{\rm 3,n}$ general arises as an extension by $N \in {\rm Pic}^{d-2g+2+2\nu}(C)$ general of $L=K_C - (A+M)$. In particular, since $h^1(N) = 0$, then 
$$H^0(\Ff) \simeq H^0(N) \oplus H^0(K_C - A-M) \; {\rm and} \; 
H^0(\omega_C \otimes \Ff^{\vee}) \simeq H^0(A+M).$$Thus the Petri map $\mu_\mathcal F$ reads also as 
\begin{equation*}\begin{CD}
\left[H^0(N)\oplus H^0(K_C - A-M)\right]&\;\otimes\;& \;H^0(A+M)&\;\; \stackrel
{\mu_\mathcal F }{\longrightarrow}\;\; &H^0(\omega_C\otimes \mathcal F\otimes \mathcal F^{\vee}).\\\end{CD}
 \end{equation*} By Claim \ref{lem:i=2.2.3_1-b}, we have $h^1(\mathcal F)=h^1(K_C-(A+M)) = h^0(A+M)= 3$. Consider the following diagram:

{\footnotesize
\begin{equation*}
\begin{array}{ccccccccccccccccccccccc}
&&0&&0&&0&&\\[1ex]
&&\downarrow &&\downarrow&&\downarrow&&\\[1ex]
0&\lra& N + A +M - K_C & \rightarrow & \mathcal F\otimes (A+M -K_C)& \rightarrow & \mathcal O_C &\rightarrow & 0 \\[1ex]
&&\downarrow && \downarrow && \downarrow& \\[1ex]
0&\lra & N \otimes  \mathcal F^{\vee}& \rightarrow & \mathcal F\otimes \mathcal F^{\vee} & \rightarrow &
\mathcal  (K_C-A-M) \otimes \mathcal F^{\vee}&\rightarrow & 0 \\[1ex]
&&\downarrow &&\downarrow &&\downarrow  &&\\[1ex]
0&\lra & \mathcal O_{C}& \lra & \mathcal F\otimes N^{\vee}&\lra& (K_C-A-M)\otimes N^{\vee}&\rightarrow& 0\\[1ex]
&&\downarrow &&\downarrow&&\downarrow&&\\[1ex]
&&0&&0&&0&&
\end{array}
\end{equation*}}

\noindent 
which arises from the exact sequence defining $\Ff$ tensored by its dual sequence. If we tensor the column in the middle by  $\omega_C$, we get $H^0(\mathcal F\otimes (A+M))\hookrightarrow  H^0(\omega_C\otimes \mathcal F\otimes \mathcal F^{\vee})$. From the isomorphism $H^0(\omega_C\otimes\mathcal F^{\vee}) \simeq H^0(A+M)$, the Petri map of $\Ff$ simply reads
$\mu_{\Ff} : H^0(\mathcal F)\otimes H^0(A+M)\to  H^0(\omega_C\otimes \mathcal F\otimes \mathcal F^{\vee})$. By the previous injection  
$H^0(\Ff \otimes (A+M)) \hookrightarrow H^0(\omega_C\otimes \mathcal F\otimes \mathcal F^{\vee})$ we deduce that the kernel of $\mu_{\Ff}$ is the same as the kernel of the natural multiplication map $\mu' := \mu_{\mathcal F, A+M}: H^0(\mathcal F)\otimes H^0(A+M)\to H^0(\mathcal F\otimes (A+M))$.

Since $N$ is general with $h^1(N) = 0$, so it is $N+ A +M$ and thus $h^1(N + A +M) = 0$. In particular, from the above diagram one has 
$$H^0(N + A +M) \oplus H^0(\omega_C) \simeq H^0( \mathcal F\otimes (A +M)) \hookrightarrow H^0(\omega_C \otimes \mathcal F\otimes \mathcal F^{\vee}),$$where the isomorphism follows from the first row whereas the inclusion comes from the middle column of the previous diagram. 

Considering the natural multiplication map 
$\mu_{N, A +M}: H^0(N) \otimes  H^0(A +M)\to H^0(N+A +M)$ and the Petri map of $A +M$, namely $\mu_{A +M} : H^0(K_C- A -M) \otimes H^0(A +M) \to H^0(K_C)$,  from above we have that $\mu_{\Ff} = \mu_{N, A +M} \oplus \mu_{A +M}$ and we can consider the following exact  diagram:

{\footnotesize
\begin{equation*}
\begin{array}{cclcccl}
& 0&&0&&0 & \\[1ex]
& \downarrow&&\downarrow &&\downarrow& \\[1ex]
0 \to & \ker (\mu_{N, A +M})& \to  &  H^0(N) \otimes H^0(A +M) & \stackrel{\mu_{N, A +M}}{\longrightarrow} &H^0(N+A +M) & \\[1ex]
& \downarrow && \downarrow && \downarrow  & \\[1ex]
0 \to &\ker (\mu_{\Ff}) &\to & [H^0(N) \oplus H^0(K_C - A -M)] \otimes H^0(A +M) & \stackrel{\mu_\mathcal F}{\longrightarrow} & H^0(\mathcal F \otimes (A+M)) & 
\\[1ex]
& \downarrow && \downarrow &&\downarrow & \\[1ex]
0 \to &\ker(\mu_{A +M}) &\to & H^0(K_C -A -M)\otimes H^0(A +M) & \stackrel{\mu_{A +M}}{\longrightarrow}   & H^0(K_C)& \\[1ex]
&\downarrow && \downarrow &&\downarrow & \\[1ex]
& 0 && 0&&0 &
\end{array}
\end{equation*}} To compute the kernel of the multiplication map $\mu_{N, A +M}$,
recall that $A+M$ is globally generated (cf. Lemma \ref{lem:w2tLar}-$(b)$), therefore we have the exact sequence
$$0 \to {\mathcal K} \to H^0(A+M) \otimes \mathcal O_C \to A+M \to 0,$$where $\mathcal K$ denotes the 
rank-$2$ kernel vector bundle of the previous surjection. Tensoring by $\mathcal F$, one can complete to the following commutative diagram: 
{\footnotesize 
\begin{equation*}
\begin{array}{ccccccccccccccccccccccc}
&&0&&0&&0&&\\[1ex]
&&\downarrow &&\downarrow&&\downarrow&&\\[1ex]
0&\lra&  (-A) \otimes N & \rightarrow & H^0(A) \otimes N& \rightarrow & A \otimes N &\rightarrow & 0 \\[1ex]
&&\downarrow && \downarrow && \downarrow& \\[1ex]
0&\lra & \mathcal K \otimes N & \rightarrow & H^0(A+M) \otimes N & \rightarrow &
N \otimes (A+M)   &\rightarrow & 0 \\[1ex]
&&\downarrow &&\downarrow &&\downarrow  &&\\[1ex]
0&\lra & N -M & \lra & N  &\lra& N\otimes \mathcal O_M&\rightarrow& 0\\[1ex]
&&\downarrow &&\downarrow&&\downarrow&&\\[1ex]
&&0&&0&&0&&.
\end{array}
\end{equation*}}
From $\deg(N) \geq g-1+\nu$ and from the generality of $N$, we have $h^1(N-A)=0$. Since $h^0(N-M)=0$, we get 
$$\dim \left(\ker (\mu_{N, A +M})\right)=h^0(N-A)=(d-2g+2+\nu+n-g+1)-\nu=d-3g+3+n.$$
On the other hand the Petri map $\mu_{A+M}$ has a corank $2n+\nu-g-2$, since $\dim(\overline{W^{\vec{w}_{2,1}}}) = 3g-6-2\delta-\nu = 2n+\nu - g - 2$, thus
 \begin{eqnarray*}
 \dim(\ker(\mu_{A+M})) & = & 3 \, h^0(K_C-A-M) - g + (2n+\nu-g-2)  \\
                                       & = & 3 (g+2-\nu-n) - g  + (2n+\nu-g-2) = g+4-n-2\nu.
                                       \end{eqnarray*}
\end{proof}

We want to show that the irreducible loci $\mathcal B_{\rm 3,n}$, arising from bundles of type $(1-b)$ in  Lemma \ref{lem:barjvero} and which actually exist under conditions as in Lemma \ref{lem:sup3case1-b}, are locally closed subschemes strictly contained (according to the chosen $d$) in irreducible components of the form $\mathcal B_{\rm sup, 3,  (2-b)-mod}$ as in Proposition \ref{prop:case2b}, which instead have been constructed therein for any $ 3 \leq \nu \leq  \frac{g+3}{3}$ and $2g-7+2\nu \leq d \leq 3g-4-\nu$. 

Notice that when $3 \leq \nu \leq \frac{g-4}{3}$, as requested for the existence of $\mathcal B_{\rm 3,n}$ in Lemma \ref{lem:sup3case1-b}, one clearly has $3g-7 -\nu < 3g-4-\nu$ and $2g-7+2\nu < \frac{5g-10+\nu}{2}$; therefore, for any irreducible $\mathcal B_{\rm 3,n}$ there certainly exists an irreducible component of type $\mathcal B_{\rm 3, (2-b)-mod}$ as in Proposition \ref{prop:case2b}-$(ii)$. We moreover recall that 
$$\dim(\mathcal B_{\rm 3, (2-b)-mod}) = 8g-11-2\nu - 2d > 5g-8-d-\nu = \dim(\mathcal B_{\rm 3,n})$$and that for 
$[\Ff_w] \in \mathcal B_{\rm 3,n}$ general we have also 
\[
\dim (T_{[\mathcal F_w]}(B^{k_3}_d \cap U^s_C(d))) = 8g-2d-2\nu-11 = \dim(\mathcal B_{\rm 3, (2-b)-mod}), 
\] as proved in Claim \ref{ChoiJ}. Recall further that a general point 
$[\Ff_u] \in \mathcal B_{\rm 3, (2-b)-mod}$ as in Proposition \ref{prop:case2b}-$(ii)$ arises from an exact sequence of the form 
\begin{equation}\label{eq:7gencogl1}
0 \to N \to \Ff_u \to K_C-A \to 0,
\end{equation}where $N \in {\rm Pic}^{d-2g+2+\nu}(C)$ general, $h^1(K_C-A) = h^0(A) =2$  and $u \in \mathcal W_1 \subset \ext^1(K_C-A,\,N)$ general, namely the associated coboundary map $\partial_u$ is such that ${\rm corank} (\partial_u)=1$, whereas $[\Ff_w] \in  \mathcal B_{\rm 3,n}$ general as in  Lemma \ref{lem:sup3case1-b} arises from a general extension
\begin{equation}\label{eq:7gencogl2}
0 \to N' \to \Ff_w \to K_C-(A +M) \to 0,
\end{equation}where $N' \in {\rm Pic}^{d-2g+2+\nu+n}(C)$ general, $M$ is an effective divisor of movable points of degree $\deg(M) = n$ on $C$ such that $A + M \in \overline{W^{\vec{w}_{2,1}}} \subset W^2_{\nu + n}(C)$ is general in the component $\overline{W^{\vec{w}_{2,1}}} $ as in Lemma \ref{lem:w2tLar}, i.e. $|A+M| = g^2_{\nu+n}$ is base-point-free and birationally very ample, so $h^1(\Ff_w) = h^1(K_C-(A+M)) = h^0(A+M)=3$ and the corresponding coboundary map $\partial_w$ is surjective.

\begin{proposition}\label{prop:fanc19-12} Let $3 \leq \nu \leq \frac{g-4}{3}$, $\frac{5g-10+\nu}{2} \leq d \leq 3g-7-\nu$ and $\frac{g+2-\nu}{2} \leq n \leq g-2\nu-1$ be integers and let $C$ be a general $\nu$-gonal curve of genus $g$. Then any irreducible locus $\mathcal B_{\rm 3,n}$ as in Lemma \ref{lem:sup3case1-b} is strictly contained in a component $\mathcal B_{\rm 3, (2-b)-mod}$ as in Proposition \ref{prop:case2b}, in such a way that $[\Ff_w] \in \mathcal B_{\rm 3,n}$ general is a smooth point of the corresponding component $\mathcal B_{\rm 3, (2-b)-mod}$ containing it. 

In particular, the loci $\mathcal B_{\rm 3,n}$ arising from case $(1-b)$ as in Lemma \ref{lem:barjvero} can never fill-up an irreducible component of $B_d^{k_3} \cap U^s_C(d)$. 
\end{proposition}

\begin{proof} Let $[\mathcal{F}_w] \in \mathcal{B}_{3,n}$ be a general point. By construction, $\mathcal{F}_w$ fits into the extension \eqref{eq:7gencogl2}, where $N' \in \text{Pic}^{d-2g+2+\nu+n}(C)$ is general, $A+M \in \overline{W^{\vec{w}{2,1}}} \subset W^2_{\nu+n}(C)$ is a general $g^2_{\nu+n}$, and $w \in \text{Ext}^1(K_C-A-M, N')$ is general. In particular, the coboundary map $\partial_w$ is surjective. The surface scroll $F_w = \mathbb{P}(\mathcal{F}_w)$ contains a special section $\Gamma$ of degree $2g-2-\nu-n$ and speciality $i(\Gamma) = 3$, corresponding to the quotient $\mathcal{F}_w \twoheadrightarrow K_C-A-M$. This section is algebraically isolated (ai); indeed, its normal bundle $\mathcal{N}_{\Gamma/F_w} \simeq K_C-A-M-N'$, the latter being general of degree $\deg(K_C-A-M-N') \leq \frac{g-2-3\nu}{2} < 0$, hence $h^0(\mathcal{N}_{\Gamma/F_w})=0$. Composing the injection $N'(-M) \hookrightarrow N'$ with \eqref{eq:7gencogl2} identifies $N'(-M)$ as a sub-line bundle, yielding:
\begin{equation}\label{eq:quotientQ0}
0 \to N'(-M) \xrightarrow{\tilde{\sigma}} \mathcal{F}_w \to \mathcal{Q} \to 0.
\end{equation} The morphism $\tilde{\sigma}$ induces a global section $\sigma$ of $\mathcal{F}_w \otimes (N')^{\vee} \otimes \mathcal{O}_C(M)$. Tensoring \eqref{eq:7gencogl2} by $(N')^{\vee} \otimes \mathcal{O}_C(M)$ gives:
\begin{equation*}
0 \to M \to \mathcal{F}_w \otimes (N')^{\vee} \otimes \mathcal{O}_C(M) \to K_C-A-N' \to 0.
\end{equation*} Since $\deg(K_C-A-N') \leq g-5$, the generality of $N'$ implies $h^0(K_C-A-N') = 0$. Consequently
\begin{equation}\label{eq:unasez}
h^0(\mathcal{F}_w \otimes (N')^{\vee} \otimes \mathcal{O}_C(M)) = h^0(M).
\end{equation}

\begin{claim}\label{cl:Seonja8gen25}  For  $A + M \in \overline{W^{\vec{w}_{2,1}}} \subseteq W^2_{\nu + n}(C)$ general in the component $\overline{W^{\vec{w}_{2,1}}} $ as in Lemma \ref{lem:w2tLar}-$(b)$, one has $h^0(M)=1$. 
\end{claim}
\begin{proof}[Proof of Claim \ref{cl:Seonja8gen25}] Since $M$ is effective, then $h^0(M) \geq 1$. Moreover $h^0(M) \leq 2$, because $h^0(A+M) =3$ and 
$|A+M|$ is base-point-free. We want to show that $h^0(M) =2$ cannot occur. 

If by contradiction we had $h^0(M) =2$, considering the natural multiplication map
$$\mu_{A,M}:  H^0 (A) \otimes H^0 (M) \to H^0 (A+M)$$ and the fact that $|A|$ is base-point-free, by the base-point-free pencil trick we would get ${\rm ker} (\mu_{A,M}) = h^0(M-A) \geq 1$ because $h^0(A+M)=3$. Thus there would exist an effective divisor $D_{n-\nu}$ of degree $\deg(D_{n-\nu}) = n-\nu$ such that $M \sim A+D_{n-\nu}$ so that 
$A+M \sim 2A + D_{n-\nu}$, contradicting that $A+M$ is general in  $\overline{W^{\vec{w}_{2,1}}} $. 
\end{proof}

From \eqref{eq:unasez} and Claim \ref{cl:Seonja8gen25}, we get $h^0(\Ff_w \otimes (N')^{\vee} \otimes \mathcal O_C(M)) = 1$. Similarly, tensoring \eqref{eq:7gencogl2} by $(N')^{\vee}$ we get $$0 \to \mathcal O_C \to \Ff_w \otimes (N')^{\vee} \to K_C-A- M- N' \to 0.$$Since $K_C-A-M-N'$ is general, by generality of $N'$, and since $\deg(K_C-A-M-N') = \deg(K_C-A-N') - n < \deg(K_C-A-N')$, reasoning as above we have $h^0(K_C-A-M-N') = 0$, so $h^0(\Ff_w \otimes (N')^{\vee}) = h^0(\mathcal O_C)=1$. In other words, setting $\mathcal E:= \Ff_w  \otimes (N')^{\vee} \otimes \mathcal O_C(M)$, from \eqref{eq:unasez} and from above we have that 
$$h^0(\mathcal E) = h^0(\mathcal E(-M)) =1.$$Therefore, the unique global section $\sigma$ of $\mathcal E = \Ff_w  \otimes (N')^{\vee} \otimes \mathcal O_C(M)$ induced by \eqref{eq:quotientQ0} coincides with the unique global section of $\mathcal E(-M) =  \Ff_w  \otimes (N')^{\vee}$ which, from \eqref{eq:7gencogl2}, does not vanish at any point of $C$ as the cokernel of the injection $\mathcal O_C \hookrightarrow \mathcal E(-M) =  \Ff_w  \otimes (N')^{\vee}$ is locally free, being isomorphic to $K_C-A-M-N'$. In other words, $\sigma$ vanishes along $M$ and we have 
$$0 \to \mathcal O_C \stackrel{\sigma}{\longrightarrow}  \mathcal E= \Ff_w  \otimes (N')^{\vee} \otimes \mathcal O_C(M) \to (K_C-A-N') \oplus \mathcal O_M \to 0.$$Comparing with the definition of $\sigma$, we have $\mathcal Q \simeq (K_C-A-M) \oplus \mathcal O_M$ and, replacing it into \eqref{eq:quotientQ0}, we get that $\Ff_w$ fits also in 
\begin{equation}\label{eq:9gencogl}
0 \to N'(-M) \to \Ff_w \to (K_C-(A+M)) \oplus \mathcal O_M \to 0,
\end{equation}i.e. $F_w$ contains the reducible unisecant $\widetilde{\Gamma} : = \Gamma + f_M$ as in \eqref{eq:Fund2} of degree $\deg(\widetilde{\Gamma}) = 2g-2-\nu$ and of speciality $i(\widetilde{\Gamma}) = i(\Gamma) =3$ as in \eqref{eq:iLa}, being $\Gamma$ the section corresponding to the quotient line bundle $K_C-(A+M)$. 

From \eqref{eq:9gencogl}, $\Ff_w$ arises therefore also as an element $w$ of the extension space

\begin{equation}\label{eq:9gen1} 
  \ext^1((K_C-(A+M))\oplus \mathcal O_M,\;N'(-M)) \simeq   \ext^1(\mathcal O_M,\;N'(-M)) \oplus \ext^1(K_C-(A+M),\;N'(-M)).
\end{equation}
\normalsize We have $ \ext^1(K_C-(A+M),\;N'(-M)) \simeq H^1(N'+A-K_C)$ whereas, from \cite[III.\,Proposition\,6.7,\,p.\,235]{H} and from the fact that 
$\omega_C$ and $N'(-M)$ are locally free, we have that 
\begin{equation*}
\ext^1(\mathcal O_M,\;N'(-M)) \simeq  \ext^1(\mathcal O_M \otimes \omega_C,\;N'(-M) \otimes \omega_C)   \simeq  \ext^1(\mathcal O_M \otimes \omega_C \otimes (N')^{\vee} \otimes \mathcal O_C(M),\;\omega_C);
\end{equation*}
\normalsize
the latter is such that 
\begin{equation*}
\ext^1(\mathcal O_M \otimes \omega_C \otimes (N')^{\vee} \otimes \mathcal O_C(M),\;\omega_C)  \simeq  H^0(\mathcal O_M \otimes \omega_C \otimes (N')^{\vee} \otimes \mathcal O_C(M))^{\vee}  \simeq  H^0(\mathcal O_M)^{\vee}  \simeq  H^0(\mathcal O_M),
\end{equation*}
\normalsize
where the first isomorphism follows from \cite[III.\,Theorem\,7.6,\,p.\,243]{H}, the second and the third isomorphism from the fact that $\omega_C \otimes (N')^{\vee} \otimes \mathcal O_C(M)$ is locally free whereas $\mathcal O_M$ is a sky-scraper sheaf. Thus, replacing in \eqref{eq:9gen1}, we get that 
\begin{equation*}\label{eq:9gen2} 
\ext^1((K_C-(A+M))\oplus \mathcal O_M,\;N'(-M)) \simeq  H^1(N'+A-K_C) \oplus H^0(\mathcal O_M);
\end{equation*}this means we have the following exact sequence
\begin{equation}\label{eq:9gen3}
0 \to H^0(\mathcal O_M) \to \ext^1((K_C-(A+M))\oplus \mathcal O_M,\;N'(-M)) \to H^1(N'+A-K_C) \to 0.
\end{equation} 

Similarly, since $N \in {\rm Pic}^{d-2g+2+\nu}(C)$ and $N' \in {\rm Pic}^{d-2g+2+\nu+n}(C)$ are both general, we may consider $N = N'-M$, therefore for $[\Ff_u] \in \mathcal B_{\rm 3,\,(2-b)-mod}$ general, we have $$u \in \mathcal W_1 \subset  
\ext^1(K_C-A,\,N'-M) \simeq H^1(N'+A-K_C-M).$$If we consider the exact sequence 
$0 \to \mathcal O_C(-M) \to \mathcal O_C \to \mathcal O_M \to 0$ and tensor it by 
$N' + A - K_C$, we get 
$$0 \to N' + A - K_C-M \to N'+A-K_C \to \mathcal O_M \otimes (N'+A-K_C) \simeq \mathcal O_M \to 0.$$Notice that, from the bounds on $d$ and $n$, $\deg(N'+A-K_C) = 
(d-2g + 2 + n + \nu) + \nu -(2g-2)= d - 4g + 4 + n + 2\nu \leq (3g-7-\nu) - 4g+4 + (g-2\nu-1) + 2\nu= -4 -\nu <0$, therefore $h^0( N'+A-K_C) = 0$ and the previous exact sequence gives in cohomology 
\begin{equation}\label{eq:9gen4}
0 \to H^0(\mathcal O_M) \to H^1(N'+A-K_C-M)  \to H^1(N'+A-K_C) \to 0.
\end{equation}

Comparing \eqref{eq:9gen3} and \eqref{eq:9gen4} we deduce the following geometric interpretation via {\em elementary  transformations} of rank-$2$ vector bundles, which will complete the proof:  notice first that  $ H^1(N'+A-K_C) \simeq \ext^1(K_C-A, N')$ and 
that any  $z \in \ext^1(K_C-A, N')$ gives rise to a rank-$2$ vector bundle $\Ff_z$ of degree $\deg(\Ff_z) = \deg(N') + \deg(K_C-A) = d+n$ (moreover, for  $z \in  \ext^1(K_C-A, N')$ general, one has also $h^1(\Ff_z) = h^1(K_C-A) = 2$). Let $F_z := \mathbb P(\Ff_z)$; hence $F_z$ contains a special section $\Gamma'$ of speciality $2$ and of degree $\deg(\Gamma') = \deg(K_C-A) = 2g-2-\nu$, corresponding to the quotient line bundle $K_C-A$ of $\Ff_z$. Moreover, since from above $h^0(K_C-N'-A) = 0$, then $\Gamma'$ is algebraically isolated (ai) (so fixed, cf. Definition \ref{def:ass0}) on $F_z$. 

Looking at \eqref{eq:9gen4}, any $u \in H^1(N'+A-K_C-M) \simeq \ext^1(K_C-A,\:N'-M)$ gives rise to a surface scroll $F_u = \mathbb P(\Ff_u)$ obtained 
by an elementary transformation of $F_z$ as above, for  $z \in  H^1(N'+A-K_C) \simeq \ext^1(K_C-A,\, N')$ where such an elementary
 transformation is performed by considering $n$ points $q_1, \ldots, q_n \in F_z$, which sit on the (distinct) fibers $f_M$ of $F_z$ over some $M \in C^{(n)}$ and which do not belong to the special section $\Gamma'$; such an elementary transformation is obtained by first blowing-up $F_z$ at each point of intersection between the fiber $f_{q_i}$ and the 
special section $\Gamma'$ and then by contracting the proper transform of the fiber $f_{q_i}$, for any $1 \leq i \leq n$. This elementary transformation gives rise to a surface scroll $F_u$ arising from $u \in H^1( N'+A-K_C -M) \simeq \ext^1(K_C-A, N'-M)$; from the fact that $H^0(\mathcal O_M)$ is the kernel of the surjection \eqref{eq:9gen4}, the same surface scroll $F_z$ gives rise, by moving points $M$ on $C$ and by performing the elementary transformation as above, to positive dimensional families of surface scrolls $F_u$. Since for both $u$ and $z$ general the speciality of the corresponding bundles is $2$, by speciality reasons, it is clear that the same arguments apply for  $z \in \widetilde{\mathcal W_1} \subset \ext^1(K_C-A,\,N')$ and $u \in \mathcal W_1 \subset \ext^1(K_C-A, N'-M)$ general, which both give rise instead to bundles of speciality  $3$. 
By moving along the fibers $f_M$ of $F_z$ the points $q_1, \ldots, q_n$ of the elementary transformations described above and letting them specialize to the points of intersection between the fibers $f_M$ with the fixed section $\Gamma'$ of $F_z$, performing the elementary transformation centered along the specialized points gives rise to a surface scroll $F_w$ arising from a bundle $\Ff_w$ fitting in \eqref{eq:7gencogl2}, whose presentation is also equivalent to that in \eqref{eq:9gencogl}. This implies that $F_u$ 
degenerates to $F_w$ so also $\Ff_u$ degenerates to $\Ff_w$, which shows that $\mathcal B_{\rm 3,n} \subsetneq \mathcal B_{\rm 3, (2-b)-mod}$, according to the chosen degree $d$. 

Finally, the fact that $[\Ff_w] \in \mathcal B_{\rm 3,n}$ general is a smooth point for $\mathcal B_{\rm 3, (2-b)-mod}$ follows from the computations in Claim \ref{ChoiJ} and the proof of the Proposition is complete. 
\end{proof}

\subsection{Regular components from bundles of second type} Aside from the specific cases $d= 2g-6+2\nu$ in Proposition \ref{prop:case1a}-$(ii)$ and $d=2g-7+2\nu$ Proposition \ref{prop:case2b}-$(ii)$ previously discussed, our constructions have yielded {\em superabundant} components. We now seek to construct {\em regular} components of $B_d^{k_3}\cap U_C^s(d)$, which necessitates restricting to the range $$2g-2 \leq d \leq \frac{10}{3}g-6,$$ as such components cannot exist for $d \geq \frac{10}{3}g-5$ (cf. Remark \ref{rem:29dicnew}). 

Since, from Propositions \ref{lem:youngooksup3} and \ref{prop:fanc19-12}, we know that bundles of type $(1\text{-}b)$ and $(1\text{-}c)$ cannot fill-up irreducible components and bundles of type $(1\text{-}a)$ have been fully investigated in Proposition \ref{prop:case1a} in any possible degree $d$ (and therein, the only case for a {\em regular} component whose general point arises as a bundle of type $(1-a)$ is for $d= 2g-6 + 2\nu$), we focus on bundles of {\em second type} or, if needed, on their natural {\em modifications}.

Since we want to produce components $\mathcal{B}$ whose general point $[\mathcal{F}]$ is a {\em second type} bundle of speciality $h^1(\Ff) = 3$, the most natural case to investigate is for stable bundles fitting into
\begin{equation}\label{seq_omega-1}
0 \to N \to \mathcal{F} \to \omega_C \to 0,
\end{equation}where therefore $N$ is forced to have speciality at least $2$. Specifically, we will prove the existence of a {\em regular} component $\mathcal{B}_{\text{reg},3}$ where:
\begin{itemize}
\item for $2g-2 \leq d \leq 3g-7$, $N \in \text{Pic}^{d-2g+4}(C)$ is general; 
\item for $3g-6 \leq d \leq \frac{10g}{3}-6$, $N$ is chosen in such a way that $\omega_C \otimes N^{\vee}$ is a general $g^2_t \in \overline{W^{\vec{w}_{2,0}}} \subseteq W^2_t(C)$ (cf. Lemma \ref{lem:w2tLar}).
\end{itemize}

\medskip

\noindent
$\boxed{2g-2 \leq d \leq 3g-7:\; \mbox{Existence of a regular component}}$ As it will be clear from the proof of Proposition \ref{i=3first-1} below, to verify stability of bundles $\mathcal{F}$ as in \eqref{seq_omega-1}, we partition the range $2g-2 \leq d \leq 3g-7$ into:
$$2g-2 \leq d \leq \lfloor\tfrac{5g}{2}\rfloor -6 \quad \text{and} \quad \lfloor\tfrac{5g}{2}\rfloor -5 \leq d \leq 3g-7.$$In the second range, the stability of a general $[\mathcal{F}] \in \mathcal{B}_{\text{reg},3}$ will follow from the stability of a specific bundle $\mathcal{F}_w$ constructed in Claim \ref{cl:stabilityqui-1} (Case 2). Since stability is an open condition in irreducible flat families, the existence of this special stable point will imply stability also for the general point of the component.

\begin{proposition}\label{i=3first-1} Let $C$ be a general $\nu$-gonal curve of genus $g \geqslant 7$, with $3 \leqslant \nu < \frac{g}{2}$. For any integer $2g-2 \leqslant d \leqslant 3g-7$, $B_d^{k_3} \cap U_C^s(d)$ contains a {\em regular} component $\mathcal B_{\rm reg,3}$ which is {\em uniruled} and, for $[\Ff]\in  \mathcal B_{\rm reg,3}$ general, $\Ff$ fits into an exact sequence  of the form $$0 \to N \to \Ff \to \omega_C  \to 0,$$ where
\begin{itemize}
\item[$\bullet$] $N \in \Pic^{d-2g+2}(C)$ is general (special, non-effective), 
\item[$\bullet$] $\Ff=\Ff_v$, with $v$ general in an irreducible component $\mathcal W_2^*$ of $\mathcal W_2 \subsetneq \ext^1(\omega_C, N)$ as in \eqref{W1}, where $\mathcal W_2^*$ has  {\em expected codimension} $c(2) := 2d-10-4g$ as in \eqref{eq:clrt} and where $v \in \mathcal W_2^*$ is such that ${\rm corank}(\partial_v) =2$;  
\end{itemize} Furthermore $\Ff$ also  fits in an exact sequence of the form 
$$0 \to N(p) \to \Ff \to \omega_C(-p) \to 0,$$where $p \in C$ and $N(p) \in {\rm Pic}^{d-2g+3}(C)$ are general. Moreover, if $\Gamma$ (resp. $\Gamma_p$) denotes a section in the surface scroll $F:=\Pp(\Ff)$ which corresponds to $\Ff \to \!\! \to \omega_C$ (resp. $\Ff \to \!\! \to \omega_C(-p)$ with $p \in C$ general), then:

\smallskip

\noindent
$(i)$ $\Gamma$ is not linearly isolated (not li) on $F$ in the sense of Definition \ref{def:ass0}, as it moves in a $2$-dimensional linear system, and $\omega_C$ is not of minimal degree among possible special, effective quotients of $\Ff$. 

\smallskip 

\noindent
$(ii)$ $\Gamma_p$ is neither algebraically isolated (not ai) in the sense of Definition \ref{def:ass0} nor {\em algebraically specially isolated (not asi)} on $F$, i.e. 
$\dim_{\Gamma_p} \left(\mathcal S_F^{1,2g-3}\right) >0$, in the sense of \eqref{eq:aga}. Moreover, the line bundle $\omega_C(-p)$ is not of minimal degree among possible special quotients of $\Ff$. 
\end{proposition}

\begin{proof} In the notation of \eqref{eq:barj-1}, we set $L := \omega_C$ and let $N \in \text{Pic}^{d-2g+2}(C)$ be a general line bundle. From the given bounds on $d$, $N$ is non-effective with speciality $r := h^1(N) = 3g-3-d \geq 2$. We have $\ell := h^0(\omega_C) = g$ and $m := \dim \text{Ext}^1(\omega_C, N) = 5g-5-d$. The conditions $\ell \geq r$ and $m \geq \ell + 1$ are satisfied; thus, by Theorem \ref{CF5.8}, for a general $u \in \text{Ext}^1(\omega_C, N)$, $\partial_u$ is surjective. The degeneracy locus $\mathcal{W}_1 \subsetneq \text{Ext}^1(\omega_C, N)$ is irreducible and of expected codimension $c(1) = \ell - r + 1$. Its general point corresponds to a bundle $\mathcal{F}_z$ such that $\text{corank}(\partial_z) = 1$. By definition, we therefore have a strict inclusion $\mathcal{W}_2 \subsetneq \mathcal{W}_1$ since $\text{corank}(\partial_v) \geq 2$ for $v \in \mathcal{W}_2$ by definition. 

Note that $\mathcal{W}_2$ is non-empty, as it contains the trivial extension $\mathcal{F}_0 = N \oplus \omega_C$, whose coboundary map is the zero-map and $r \geq 4$ in this range. The expected codimension of $\mathcal{W}_2$ is $c(2) = 2(\ell - r + 2) = 2d + 10 - 4g$. Our goal is to show the existence of an irreducible component $\mathcal{W}_2^* \subseteq \mathcal{W}_2$ of expected dimension $m - c(2) = 9g - 15 - 3d$, whose general point $v$ satisfies $\text{corank}(\partial_v) = 2$. To establish this, we require the following preliminary result:

\begin{claim}\label{cl:Grass-1} With assumptions as above, let  $N \in \Pic^{d-2g+2}(C)$ be general and let $\mathbb G := \mathbb{G} (2, H^0(K_C-N))$ be the Grassmannian parametrizing $2$-dimensional sub-vector spaces in $H^0(K_C-N)$, the latter of dimension $r = 3g-3-d >2$. Let $\mu := \mu_{K_C-N,\; K_C} : H^0(K_C-N) \otimes H^0(K_C) \to H^0(2K_C  -N)$ be the natural multiplication map and, for any $[V_2] \in \mathbb{G}$, let $\mu_{V_2}$ denote the restriction of the map $\mu$ to the sub-vector space $V_2 \otimes H^0(K_C) \subset H^0(K_C-N) \otimes H^0(K_C)$, namely $\mu_{V_2} : V_2 \otimes H^0(K_C) \to H^0(2K_C -N)$. Then, for $V_2 \in \mathbb G$ general, the map $\mu_{V_2}$ is injective. 
\end{claim}
\begin{proof}[Proof of Claim \ref{cl:Grass-1}] The proof is straightforward; in fact, a general $[V_2] \in \mathbb{G}$ determines a $2$-dimensional sub-vector space of $H^0(K_C-N)$ giving rise to a base-point-free linear pencil on $C$. Indeed, since $h^0(K_C-N) = 3g-3-d \geqslant 4$, we can  take $\sigma_1, \sigma_2 \in H^0(K_C-N)$ general sections; if $p \in C$ is such that $\sigma_i(p)=0$, for $i=1,2$, by generality of the sections, we would have $p \in {\rm Bs}(|K_C-N|)$ so $h^0(N+p)=1$ a contradiction because $N$ is general of $\deg(N) = d-2g+2 \leq g-5$, from the bounds on $d$. 

The injectivity of $\mu_{V_2}$ follows at once from the base-point-free-pencil-trick since $\ker(\mu_{V_2}) \cong H^0(N)$, which is zero, being $N$ non-effective.
\end{proof}

\begin{claim}\label{cl:good-1} With above assumptions, let $N \in \Pic^{d-2g+2}(C)$ be general. Then, there exists a unique irreducible  component $\mathcal W_2^*$ of  $\mathcal W_2 \subsetneq \ext^1(\omega_C, N)$ whose general $v \in \mathcal W_2^*$ is such that ${\rm Coker}(\partial_v)^{\vee}$ corresponds to a general point in the Grassmannian $\mathbb{G}$. Moreover $\mathcal W_2^*$ is {\em unirational}, of the {\em expected dimension} as in \eqref{eq:clrt}, namely $\dim(\mathcal W_2^*) = 9g-15-3d$. 
\end{claim} 

\begin{proof}[Proof of Claim \ref{cl:good-1}] Let $\mathbb{P} := \mathbb{P}(\text{Ext}^1(\omega_C, N)) \cong \mathbb{P}(H^0(2K_C - N)^\vee)$. We define the incidence variety: $$ \mathfrak{I} := \{(V_2, \pi) \in \mathbb{G} \times \mathbb{P} \mid \text{im}(\mu_{V_2}) \subseteq \pi \}, $$where $\mu_{V_2} : V_2 \otimes H^0(K_C) \to H^0(2K_C - N)$ is the multiplication map as in Claim \ref{cl:Grass-1}. Given $m = 5g-5-d$, $\ell = g$, and $r = 3g-3-d$, we have $\ell \geq r > 2$ and $m \geq 2\ell + 1$. Thus, by Claim \ref{cl:Grass-1}, for a general $[V_2] \in \mathbb{G}$, the map $\mu_{V_2}$ is injective with $\dim(\text{im}(\mu_{V_2})) = 2\ell \leq m-1$; in particular $\text{im}(\mu_{V_2})$ is contained in at least one hyperplane $\pi$, ensuring that $\mathfrak{I}$ is non-empty and that dominates $\mathbb{G}$. 

Consider the natural projections $\mathbb{G} \xleftarrow{pr_1} \mathfrak{I} \xrightarrow{pr_2} \mathbb{P}$. For general $[V_2] \in \mathbb G$, the fiber $pr_1^{-1}([V_2])$ is a linear system of hyperplanes of dimension $m - 1 - 2\ell = 3g-6-d > 0$. Since $\mathbb{G}$ is irreducible and all the $pr_1$-fibers are irreducible of the same dimension, $\mathfrak{I}$ is irreducible with $$ \dim(\mathfrak{I}) = \dim(\mathbb{G}) + \dim(pr_1^{-1}(V_2)) = (6g-10-2d) + (3g-6-d) = 9g-16-3d.$$ Let $\widehat{\mathcal{W}^*}_2 := pr_2(\mathfrak{I})$; by construction, $\widehat{\mathcal{W}^*}_2 \subseteq \mathbb{P}(\mathcal{W}_2)$. For a general $[v] \in \widehat{\mathcal{W}^*}_2$, $\text{corank}(\partial_v)=2$ with $\text{coker}(\partial_v)^\vee \simeq V_2 \in \mathbb{G}$. Note that $\dim(\mathfrak{I})$ matches the expected dimension $m-1-c(2)$. Since $\mathfrak{I}$ dominates $\mathbb{G}$ and is irreducible, $\widehat{\mathcal{W}^*}_2$ must be an irreducible component of $\mathbb{P}(\mathcal{W}_2)$. It must satisfy $\dim(\widehat{\mathcal{W}^*}_2) \geq \text{expdim}(\widehat{\mathcal{W}_2})$, forcing the restriction $pr_2$ to be generically finite so $\dim(\widehat{\mathcal{W}^*}_2) = 9g-16-3d$. 

Finally, $\widehat{\mathcal{W}^*}_2$ is unirational (as a projective bundle over a rational Grassmannian), and its uniqueness follows from the irreducibility of $\mathfrak{I}$.\end{proof}

\smallskip

To construct $\mathcal B_{\rm reg,3}$, we take into account a projective bundle $\mathbb P(\mathcal E_{d})\to S$ over a suitable open dense subset $S \subseteq {\rm Pic}^{d-2g+2}(C)$, which parametrizes  the family of $\mathbb P(\ext^1(K_C, N))$'s  as  $N \in S$ varies. Since from Claim \ref{cl:good-1}, for any such $N \in S$, $\widehat{\mathcal W^*}_2$ is irreducible, {\em unirational} and of (constant) dimension $9g-16-3d$, one has an irreducible projective variety 
\begin{equation}\label{eq:bordersonja}
(\widehat{\mathcal W^*}_2)^{Tot}:= \left\{ (N,[v]) \in \mathbb P(\mathcal E_{d}) \; | \; [v] \in \widehat{\mathcal W^*}_2\right\},
\end{equation} which is  therefore {\em uniruled} over $S$, of dimension 
\begin{equation} \label{eq;non-special-dim}
\dim((\widehat{\mathcal W^*}_2)^{Tot}) = \dim(S) + \dim(\widehat{\mathcal W^*}_2) =  g  + (9g-16 - 3d) = 10 g - 3d - 16 = \rho_d^{k_3} +2.
\end{equation} Our next aim is to construct a (rational) modular map $(\mathcal W_2^*)^{Tot} \stackrel{\pi}{\dasharrow} B_d^{k_3} \cap U_C^s(d)$. To define this map, we therefore need to show that the bundle $\Ff_v$ arising from a general pair $(N, [v]) \in (\widehat{\mathcal W^*}_2)^{Tot}$ is stable. Using Theorem \ref{LN}, this will be done in the result below.

\begin{claim}\label{cl:stabilityqui-1} For $(N, [v]) \in (\widehat{\mathcal W^*}_2)^{Tot}$ general, the corresponding rank-$2$, degree $d$ vector bundle $\Ff_v$ is stable, of speciality $h^1(\Ff_v) = 3$.   
\end{claim}

\begin{proof}[Proof of Claim \ref{cl:stabilityqui-1}] Since $\text{corank}(\partial_v) = 2$ and $h^1(\omega_C) = 1$, it follows immediately that $h^1(\mathcal{F}_v) = 3$. We prove stability of the general bundle $\mathcal{F}_v$ by considering two ranges for $d$:
\begin{center} 
(Case 1): \;\; $2g-2 \leq d\leq \lfloor\frac{5g}{2}\rfloor-6$ \; and \; (Case 2): \;\; $\lfloor\frac{5g}{2}\rfloor -5\leq d \leq 3g-7$. 
\end{center}

\smallskip

\noindent
(Case 1): assume $2g-2 \leq d \leq \lfloor\frac{5g}{2}\rfloor-6$ and take $\delta = \deg(\omega_C) = 2g-2$, so $2\delta - d = 4g-4-d > 2$, allowing us to apply Theorem \ref{LN}. Consider the morphism $\varphi: C \to \mathbb{P}$ defined by the linear system $|2K_C-N|$, and let $X = \varphi(C)$. For an integer $\sigma \equiv 2\delta-d \pmod 2$ with $0 \leq \sigma \leq 2\delta-d$, one has
$$ \dim (\text{Sec}_{\frac{1}{2}(2\delta-d+\sigma-2)}(X)) = 4g-7-d+\sigma. $$This is strictly less than $\dim(\widehat{\mathcal{W}^*}_2) = 9g-16-3d$ provided $\sigma < 5g-9-2d$. Given the upper bound on $d$, $5g-9-2d \geq 3$. Thus, we can choose $\sigma \in \{1, 2\}$, which ensures that the general extension in $\widehat{\mathcal{W}^*}_2$ is stable.

\smallskip

\noindent
(Case 2): assume  $\lfloor\frac{5g}{2}\rfloor -5 \leq d \leq 3g-7$; in this range, we use the fact that stability is an open condition and we prove the existence of at least one peculiar stable bundle $\mathcal{F}_w$ in the total family $(\widehat{\mathcal{W}^*}_2)^{Tot}$. Such a bundle will arise from the construction in Claim \ref{lem;i=3second} below. 

Recall that, by Lemma \ref{lem:Lar}, the component $\overline{W^{\vec{w}_{1,0}}} \subset W^1_{4g-4-d-t}(C)$ is non-empty if $\rho(g,1,4g-4-d-t) \geq 0$, which is equivalent to $7g-10-2d-2t \geq 0$. This condition holds for $d \leq 3g-7$ and appropriately chosen $t \leq \frac{g+2}{2}$.  Based on this, we have the following:

 \begin{claim}\label{lem;i=3second} Let $C$ be a general $\nu$-gonal curve of genus $g \geqslant 7$, with 
 $3 \leqslant \nu < \frac{g}{2}$ and let $d$ be an integer with $ \lfloor\frac{5g}{2}\rfloor -5 \leqslant d \leqslant 3g-7$, and either $t:=\frac{g+2}{2}$, when both $g$ and $d$ are even, or $t := \lceil \frac{g}{2} \rceil$ otherwise.  Assume that $D_t \in C^{(t)}$  is general and that $N' \in \Pic^{d-2g+2+t}(C)$  is such that $\omega_C \otimes (N')^{\vee}$ is general in the component $\overline{W^{\vec{w}_{1,0}}} \subseteq W^1_{4g-4-d-t}(C)$ as in Lemma \ref{lem:Lar}.  Then there exists a stable, rank-$2$ vector bundle $\Ff_w$ of degree $d$ and speciality $h^1 (\Ff_w) =3$ which fits into an exact sequence  of the form 
 \begin{equation} \label{sequence-w}
 (w): \ \ \  0 \to N' \to \Ff _w \to \omega_C (-D_t) \to 0,
 \end{equation}
\end{claim}

\begin{proof}[Proof of Claim \ref{lem;i=3second}] Since $K_C-N' \in \overline{W^{\vec{w}_{1,0}}}$ is general, $h^1(N') = 2$ and $h^0(N') = d - 3g + 5 + t \geq 0$. We define the degeneracy locus $\mathcal{W}_2 \subset \text{Ext}^1(K_C-D_t, N')$ as in \eqref{W1}. 
Setting $\ell := h^0(K_C-D_t) = g-t$, $r := 2$, and $m := 5g-5-2t-d$, we have $\ell \geq r$ and $m \geq \ell+1$ for $d \leq 3g-7$ and $t \leq \frac{g+2}{2}$. The locus $\mathcal{W}_2$ is the dual of the cokernel of the multiplication map:
$$ \Phi: H^0(K_C-D_t) \otimes H^0(K_C-N') \to H^0(2K_C-N'-D_t). $$ Since $|K_C-N'|$ is a base-point-free pencil, the base-point-free pencil trick implies $\ker(\Phi) \simeq H^0(N'-D_t) = 0$. Thus, $\Phi$ is injective and:$$ \dim(\mathcal{W}_2) = m - 2\ell = 3g-5-d \geq 2. $$Any bundle $\mathcal{F}_w$ corresponding to $w \in \mathcal{W}_2$ has speciality $h^1(\mathcal{F}_w)=3$, as $\partial_w = 0$ and $h^1(N') + h^1(K_C-D_t) = 2+1=3$.

Concerning stability, let $\delta = \deg(\omega_C(-D_t)) = 2g-2-t$. Since $d \leq 4g-6-2t$, we have $2\delta-d \geq 2$. We apply Theorem \ref{LN} using the morphism $\varphi: C \to \mathbb{P}(H^0(2K_C-N'-D_t)^\vee)$ and let $X = \varphi(C)$. For $\sigma \equiv 2\delta-d \pmod 2$, the condition for stability $\mathcal{W}_2 \not\subseteq \text{Sec}_{\frac{1}{2}(2\delta-d+\sigma-2)}(X)$ reduces to $\dim(\text{Sec}) \leq \dim(\mathcal{W}_2)$, which is satisfied if $\sigma \leq -g+1+2t$. Based on the parity of $g$ and $d$:

\smallskip

\noindent
$\bullet$ If $g$ is odd ($t = \frac{g+1}{2}$) or $g$ and $d$ are even ($2t = g+2$), we take $\sigma \in \{1, 2\}$. The secant variety is either of lower dimension than $\mathbb{P}(\mathcal{W}_2)$ or it is a non-linear variety of the same dimension; in either case, the general $w \in \mathcal{W}_2$ yields a stable bundle $\mathcal{F}_w$.

\smallskip

\noindent
$\bullet$ If otherwise $g$ is even and $d$ is odd ($t = \frac{g}{2}$), we take $\sigma = 1$. Again, $\text{Sec}(X)$ is a non-linear variety of the same dimension as the linear space $\mathbb{P}(\mathcal{W}_2)$, ensuring the general bundle is stable.

This concludes the proof of the stability of the bundle $\Ff_w$.
\end{proof}

We can therefore go on with the proof of (Case 2) of Claim \ref{cl:stabilityqui-1}. A further key step will be the following:

\begin{claim}\label{cl:wspecialptgen} For $D_t \in C^{(t)}$, $\omega_C \otimes (N')^{\vee} \in \overline{W^{\vec{w}_{1,0}}} \subseteq W^1_{4g-4-d-t}(C)$ and 
$w \in \mathcal W_2 \subset \ext^1(K_C-D_t, N')$ general elements as in (the statement and the proof of) Claim \ref{lem;i=3second}, the pair $(N'-D_t, [w])$ belongs to the closure of $(\widehat{\mathcal W^*}_2)^{Tot}$ as in \eqref{eq:bordersonja}. 
\end{claim}

\begin{proof}[Proof of Claim \ref{cl:wspecialptgen}] Let $\mathcal{F}_w$ be the bundle constructed in Claim \ref{lem;i=3second} with $h^1(\mathcal{F}_w)=3$ and $\partial_w = 0$. Since $N' = K_C - g^1_{4g-4-d-t}$, Lemma \ref{reg_seq_1_3}-(ii) implies that $\mathcal{F}_w$ admits a second presentation:
\begin{equation}\label{Reg3_seq;Dt-1}
0 \to N'-D_t \to \mathcal{F}_w \to \omega_C \to 0.
\end{equation} The corresponding coboundary map $\widetilde{\partial_w}$ satisfies $\text{corank}(\widetilde{\partial_w})=2$. Note that $\deg(N'-D_t) = d-2g+2$, matching the degree of a general $N$ in Claim \ref{cl:good-1}. Since $d \leq 3g-7$, we have $h^0(N'-D_t)=0$ and $h^1(N'-D_t) = 3g-3-d$, ensuring that $N'-D_t$ has the same cohomology as a general $N$. 

Let $\widetilde{\mathbb{G}} := \mathbb{G}(2, H^0(K_C-(N'-D_t)))$. Presentation \eqref{Reg3_seq;Dt-1} gives a point $[\text{coker}(\widetilde{\partial_w})^\vee] \in \widetilde{\mathbb{G}}$. To determine the properties of such a point, we analyze the dual sequences:$$ \text{coker}(\widetilde{\partial_w})^\vee \oplus H^0(\mathcal{O}_C) \simeq H^0(\omega_C \otimes \mathcal{F}_w^\vee) \simeq H^0(g^1_{4g-4-d-t}) \oplus H^0(\mathcal{O}_C(D_t)). $$Let $\{1, f\}$ be a basis for $H^0(g^1_{4g-4-d-t})$ and $h \in H^0(\mathcal{O}_C(D_t))$ be a non-zero global section. Then $\text{coker}(\widetilde{\partial_w})^\vee$ is generated by $\{f, h\}$. By the base-point-free pencil trick, the restricted multiplication map $\mu_{\text{coker}(\widetilde{\partial_w})^\vee}$ is injective since 
$\dim ({\rm ker}(\mu _{{\rm coker}(\widetilde{\partial_w})^{\vee}})) \leq h^0(K_C - g^1 - (D_t - G \cap D_t)) = 0$. By semicontinuity, the injectivity of the multiplication map holds for a general $V_2 \in \widetilde{\mathbb{G}}$. Taking into account what proved above, namely that: 

\noindent
$\bullet$ $N'-D_t = K_C - g^1_{4g-4-d-t} - D_t$ has the same degree and cohomology of $N \in {\rm Pic}^{d-2g+2}(C)$ general as in Claims \ref{cl:Grass-1}, \ref{cl:good-1} and that the same occurs for $\ext^1(K_C, N' - D_t) \simeq H^1(N'-D_t-K_C) \simeq H^0(2K_C-N' + D_t)^{\vee} = 
H^0(K_C+ D_t + g^1_{4g-4-d-t})^{\vee}$ and $\ext^1(K_C, N) \simeq H^1(N-K_C) = 
H^0(2K_C-N)^{\vee}$, 

\noindent
$\bullet$  $ [V_2] \in \widetilde{\mathbb G}$ general and $[{\rm coker}(\widetilde{\partial_w})^{\vee}] \in \Sigma \subset \widetilde{\mathbb G}$ both define sub-linear series of $|K_C-N'+D_t|=|g^1_{4g-4-d-t} + D_t|$ for which both maps 
$\mu_{V_2}$ and $\mu_{{\rm coker}(\widetilde{\partial_w})^{\vee}}$ are injective, 

\noindent
then one can reproduce verbatim the same proof as in Claim \ref{cl:good-1} and deduces that the irreducible locus $\Sigma \subseteq \widetilde{\mathbb G}$ filled-up by points of the form $[{\rm coker}(\widetilde{\partial_w})^{\vee}]$ is contained in the image, via the map $pr_1$, of the relevant incidence variety $\widetilde{\mathcal J}$  dominating $\widetilde{\mathbb G}$, as in the proof therein. This implies that  $w$ belongs to the unique irreducible component $\widetilde{\mathcal W_2^*} \subset \ext^1(\omega_C, N'-D_t)$ arising as the image via the generically finite map $pr_2$ of $\widetilde{\mathcal J}$. Therefore  when $N   \in {\rm Pic}^{d-2g+2}(C)$ general specializes to $N'-D_t = K_C - g^1_{4g-4-d-t} -D_t \in {\rm Pic}^{d-2g+2}(C)$, by uniqueness reasons the component $\mathcal W_2^* \subset \ext^1(\omega_C, N)$ specializes to   the component $\widetilde{\mathcal W_2^*} \subset \ext^1(\omega_C, N'-D_t)$ which proves that the pair $(N'-D_t, [w])$ belongs to the closure of $(\widehat{\mathcal W^*}_2)^{Tot}$ as in \eqref{eq:bordersonja}, so the proof of Claim \ref{cl:wspecialptgen} is completed. 
\end{proof}

In summary, Claim \ref{cl:wspecialptgen} establishes that the bundle $\mathcal{F}_w$ as in Claim \ref{lem;i=3second} represents a peculiar point within the family $(\widehat{\mathcal{W}^*}_2)^{Tot}$. Given its stability, as proved in Claim \ref{lem;i=3second}, and the fact that stability is an open condition it follows that, for general $v \in (\widehat{\mathcal{W}^*}_2)^{Tot}$, also the bundle $\mathcal{F}_v$ is stable. This completes the proof of Claim \ref{cl:stabilityqui-1}.
\end{proof}

Based on such stability results, we may define a rational modular map:
$$\begin{array}{ccc}
(\widehat{\mathcal{W}^*}_2)^{Tot} & \stackrel{\pi}{\dashrightarrow} & U_C(d) \\
(N, [v]) & \longrightarrow & [\mathcal{F}_v],
\end{array}$$where the image $\text{im}(\pi)$ is contained within $B_d^{k_3} \cap U_C^s(d)$. Our subsequent analysis will prove that the general fiber of $\pi$ has dimension $2$, which is related to the dimension of the linear system of canonical sections in the surface scroll 
$F = \mathbb P(\Ff)$, for $[\Ff] \in \text{im}(\pi)$ general, as in Proposition \ref{i=3first-1}-$(i)$. Consequently, we will have $\dim(\text{im}(\pi)) = \rho_d^{k_3}$ so the closure of this image in $B_d^{k_3} \cap U_C^s(d)$ is the candidate to yield the regular component $\mathcal{B}_{\text{reg},3}$.

\begin{claim}\label{cl:4-1cogl} Let  $\mathcal B_{\rm reg,3}$ be the closure of the locus ${\rm im} (\pi)$ as above in $B_d^{k_3} \cap U_C^s(d)$. Then $\mathcal B_{\rm reg,3}$ is {\em uniruled}, of dimension 
$\dim(\mathcal B_{\rm reg,3})= \rho^{k_3}_d = 10g-18-3d$. 
\end{claim}

\begin{proof}[Proof of Claim \ref{cl:4-1cogl}] The uniruledness of $\mathcal B_{\rm reg,3}$ follows from the fact that its dense subset ${\rm im}(\pi)$ is dominated by $(\widehat{\mathcal W^*}_2)^{Tot}$, which is uniruled over $S$. 

Since $\dim((\widehat{\mathcal W^*}_2)^{Tot}) = \rho_d^{k_3} + 2$, to prove $\dim(\mathcal B_{\rm reg,3}) = \rho_d^{k_3}$ it suffices to show $\dim(\pi^{-1}([\Ff_v])) = 2$ for general $[\Ff_v] \in {\rm im}(\pi)$. Let $[\Ff_v]$ be such a general point, so arising from $[v] \in \widehat{\mathcal W^*}_2 \subset \mathbb P(\text{Ext}^1(\omega_C, N))$ with $N$ general. Since $\mathcal F_{v'} \simeq \mathcal F_v$ implies $\det(\mathcal F_{v'}) \simeq \det(\mathcal F_v)$, we must have $N' \simeq N$, hence the fiber $\pi^{-1}([\Ff_v])$ lies within the same $\mathbb P(\text{Ext}^1(\omega_C, N))$. Any isomorphism $\varphi: \Ff_{v'} \to \Ff_v$ induces the diagram:
\[
\begin{array}{ccccccl}
0 \to & N & \stackrel{\iota_1}{\longrightarrow} & \Ff_{v'} & \to & \omega_C & \to 0 \\ 
 & & & \downarrow^{\varphi} & & &  \\
0 \to & N & \stackrel{\iota_2}{\longrightarrow} & \Ff_v & \to & \omega_C & \to 0.
\end{array}
\]If $[v] = [v']$, then $\varphi \in \mathbb C^*$ (by stability of $\Ff_v$) so the maps $\varphi \circ \iota_1$ and $\iota_2$ determine sections in $H^0(\Ff_v \otimes N^{\vee})$. Since $\det(\Ff_v) = \omega_C \otimes N$ and $h^1(\Ff_v) = 3$, Serre duality yields $h^0(\Ff_v \otimes N^{\vee}) = h^1(\Ff_v) = 3$. Geometrically, the fiber $\pi^{-1}([\Ff_v])$ contains the linear system $|\Gamma| \simeq \mathbb P^2$, where $\Gamma \subset \mathbb P(\mathcal F_v)$ is the canonical section corresponding to the quotient $\Ff_v \to \omega_C$. If otherwise $[v] \neq [v']$, we would find linearly independent global sections in $H^0(\Ff_v \otimes N^{\vee})$ that do not arise from a homothety, implying $\dim(|\mathcal O_{F_u}(\Gamma)|) > 2$, a contradiction. Thus $\pi^{-1}([\Ff_v]) \simeq \mathbb P^2$, and $\dim(\pi^{-1}([\Ff_v])) = 2$.
\end{proof}

\begin{claim}\label{i=3first} For $[\Ff]\in  \mathcal B_{\rm reg,3}$ general as above, the corresponding bundle  $\Ff$ admits as quotient line bundles also $\omega_C(-p)$, for $p \in C$ general. If $\Gamma$ (resp., $\Gamma_p$) denote any section $F:=\Pp(\Ff)$ corresponding to the quotient line bundle of $\Ff$ given by $\omega_C$ (resp., $\omega_C(-p)$, with $p \in C$ general), then:

\smallskip

\noindent
$(i)$ $\Gamma$ is not linearly isolated (not li) on $F$ in the sense of Definition \ref{def:ass0}, and $\omega_C$ is not of minimal degree among possible special quotients of $\Ff$. 

\smallskip

\noindent
$(ii)$ $\Gamma_p$ is neither algebraically isolated (not ai) in the sense of Definition \ref{def:ass0} nor algebraically specially isolated (not asi) on $F$, i.e. 
$\dim_{\Gamma_p} \left(\mathcal S_F^{1,2g-3}\right) >0$, in the sense of \eqref{eq:aga}. Moreover, the line bundle $\omega_C(-p)$ is not of minimal degree among possible special quotients of $\Ff$. 
\end{claim}
 
\begin{proof}[Proof of Claim \ref{i=3first}] For $[\mathcal{F}] \in \mathcal{B}_{\text{reg},3}$ general, we already know that the canonical linear system $|\Gamma|$ on $F$ satisfies $\dim(|\Gamma|) = 2$. Consequently, $\Gamma$ is clearly not linearly isolated (not li), according to Definition \ref{def:ass0}. By arguments analogous to the proof of Claim \ref{cl:case2b-1}, the bundle $\mathcal{F}$ admits $\omega_C(-p)$ as a quotient for a general $p \in C$. This implies that $\omega_C$ is not of minimal degree among the special quotients of $\mathcal{F}$, proving part $(i)$. 

Determinantal considerations show that $\mathcal{F}$ also fits into the extension:
$$ 0 \to N(p) \to \mathcal{F} \to \omega_C(-p) \to 0, $$where $N(p) \in \text{Pic}^{d-2g+3}(C)$ is general. The normal bundle of the corresponding section $\mathcal{N}_{\Gamma_p/F} \simeq K_C-2p-N$ is thus general, of degree $4g-6-d \geq g+1$. It then follows that $\mathcal{N}_{\Gamma_p/F}$ is non-special and effective with $h^0(\mathcal{N}_{\Gamma_p/F}) \geq 2$. Thus, $[\Gamma_p]$ is a smooth point of the Hilbert scheme $\text{Div}_F^{1,2g-3}$, with $\dim_{[\Gamma_p]}(\text{Div}_F^{1,2g-3}) \geq 2$. In this context, $\Gamma_p$ is not algebraically isolated (not ai). 

Since these properties hold for a general $p \in C$, there exists an open dense subset $U \subseteq C$ that maps non-constantly to $\text{Div}_F^{1,2g-3}$, parameterizing the algebraic family of sections $\{\Gamma_p\}_{p \in U}$. This proves that $\Gamma_p$, as a section corresponding to $\omega_C(-p)$, is not algebraically specially isolated (not asi); specifically $\dim_{\Gamma_p} \left(\mathcal S_F^{1,\delta}\right) > 0$ as in \eqref{eq:aga}. Given $h^0(\mathcal{N}_{\Gamma_p/F}) \geq 2$, the proof of $(ii)$ is completed with the use of Proposition \ref{prop:CFliasi}-(2-ii).
\end{proof} To conclude the proof of Proposition \ref{i=3first-1}, we are left with showing the following result.

\begin{claim}\label{cl:Petrimap3.2-a-1} For $[\Ff] \in \mathcal B_{\rm reg,3}$ general as above, the Petri map $\mu_{\Ff}$ is injective. In particular,  $\mathcal B_{\rm reg, 3}$ is an irreducible component of  $B_d^{k_3}\cap U^s_C(d)$ which is {\em regular} and {\em uniruled}. 
\end{claim}

\begin{proof}[Proof of Claim \ref{cl:Petrimap3.2-a-1}] The uniruledness and the dimension of $\mathcal{B}_{\rm reg, 3}$ were established in Claim \ref{cl:4-1cogl}. To prove regularity, we need to show the injectivity of the Petri map $\mu_{\mathcal{F}}$ for a general $[\mathcal{F}] \in \mathcal{B}_{\rm reg, 3}$. 

Such a bundle fits into the extension:
\begin{equation}\label{seq_Nv}
0 \to N \to \mathcal{F} \to \omega_C \to 0,
\end{equation} where $N \in \text{Pic}^{d-2g+2}(C)$ is general and $h^1(\mathcal{F})=3$. Let 
$\mathcal{E}_v := \omega_C \otimes \mathcal{F}^\vee$ and $g^r_{g+r} := K_C - N$. Since $K_C - N$ is general of degree at least $g+3$, it is non-special and base-point-free. The dual of \eqref{seq_Nv} yields:
\begin{equation}\label{seq-dual-nonspecial}
0 \to \mathcal{O}_C \to \mathcal{E}_v \to g^r_{g+r} \to 0,
\end{equation}
with $h^0(\mathcal{E}_v) = 3$. Since $\Ff$ arises from $(N, [v])\in (\widehat{\mathcal W^*}_2)^{Tot}$ general, one has that 
\[
 0 \to H^0 (\mathcal O_C) \to H^ 0(\mathcal E_v)\to V_2 \to 0 \ \ \ \mbox{ with  } h^0(\mathcal E_v)=3.
 \]for  a general $[V_2] \in \mathbb G := \mathbb{G} (2, H^0(K_C-N))= \mathbb{G} (2, H^0(g^r_{g+r}))$. Considering that a non-special line bundle is base-point-free, one has that $K_C -N$ is base-point-free and  so it is $V_2$ general as above.

 This fact allows us to show that $\mathcal{E}_v$ is globally generated. Indeed, considering the evaluation at an arbitrary $p \in C$ gives 
 {\footnotesize
 \begin{equation*}
\begin{array}{ccccccccccccccccccccccc}
&&0&&0&&0&&\\[1ex]
&&\downarrow &&\downarrow&&\downarrow&&\\[1ex]
0&\lra&H^0 (\mathcal O_C (-p))& \rightarrow & H^0 (\mathcal E_v (-p))& \stackrel{\alpha} \longrightarrow &  {\rm im}(\alpha) &\rightarrow & 0 \\[1ex]
&&\downarrow && \downarrow && \downarrow & \\[1ex]
0&\lra & H^0 (\mathcal O_C) & \rightarrow & H^0 (\mathcal E_v )&   \longrightarrow &
V_2   &\rightarrow & 0 ,
\end{array}
\end{equation*}} where ${\rm im}(\alpha) \subset H^0 (g^r_{g+r} (-p))$.
One notices that  ${\rm im}(\alpha)$ is a subspace of  $V_2$ and so $\dim({\rm im}(\alpha)) \leq 2$.  On the one hand,  $V_2$ gives rise to a base-point-free linear pencil of $g^r_{g+r}$, whereas ${\rm im}(\alpha)$ gives rise to a sub-series of $g^r_{g+r}(-p)\subset g^r_{g+r}$. This implies that  $ {\rm im}(\alpha)$ is  a proper subspace of $V_2$, whence $\dim ({\rm im}(\alpha)) \leq 1$. Therefore one gets $h^0 (\mathcal E_v (-p))\leq 1$, equivalently $h^0 (\mathcal E_v (-p))\leq h^0 (\mathcal E_v ) -2$ so the bundle $\mathcal E_v =\omega _C \otimes \Ff_v^\vee$ is globally generated. The evaluation sequence for $\mathcal{E}_v$ is 
$$0 \to \det(\mathcal{E}_v)^\vee = -g^r_{g+r} \to \mathcal{O}_C \otimes H^0(\mathcal{E}_v) \to \mathcal{E}_v \to 0;$$tensoring it by $\mathcal{F}$ and taking cohomology gives the {\em Petri sequence}:\begin{equation}\label{eq;seqPetri-1}
0 \to H^0(\mathcal{F}(-g^r_{g+r})) \to H^0(\mathcal{F}) \otimes H^0(\omega_C \otimes \mathcal{F}^\vee) \xrightarrow{\mu_{\mathcal{F}}} H^0(\omega_C \otimes \mathcal{F} \otimes \mathcal{F}^\vee).
\end{equation} Thus, $\ker(\mu_{\mathcal{F}}) \simeq H^0(\mathcal{F}(-g^r_{g+r}))$. From \eqref{seq_Nv}, we have the sequence $$ 0 \to N - g^r_{g+r} \to \mathcal{F}(-g^r_{g+r}) \to K_C - g^r_{g+r} \to 0. $$By the generality of $N$ and the definition of $g^r_{g+r}$, we have $h^0(N) = 0$ and $h^0(K_C - g^r_{g+r}) = h^0(N) = 0$. This implies $H^0(\mathcal{F}(-g^r_{g+r})) = 0$, proving that $\mu_{\mathcal{F}}$ is injective and so the claim. 
\end{proof}

The previous arguments conclude the proof of  Proposition \ref{i=3first-1} . 
\end{proof}
 
\bigskip

\noindent
$\boxed{3g-6 \leq d \leq \frac{10}{3}g-6:\; \mbox{Existence of a regular component}}$ \, We now address the remaining cases in which the existence of a {\em regular} component is still possible, focusing on the range $3g-6 \leq d \leq \frac{10}{3}g-6$. Similarly as in Proposition \ref{prop:case2b}, in order to ease constructions, we slightly modify the presentations of {\em second type} bundles. Specifically, we will consider kernel line bundles $N$ arising from a suitable refinement of case $(2\text{-}c)$ in Proposition \ref{prop:17mar12.52}. Before proceeding with the construction of  $\mathcal B_{\rm reg,3}$, we state the following preliminary result.

\begin{lemma}\label{lem:KC-P} Let $g \geq 9$, $4 \leq \nu < \frac{g}{2}$ and $t \leq g+2$ be positive integers and let $C$ be a general $\nu$-gonal curve of genus $g$. Let $g^2_t \in \overline{W^{\vec{w}_{2,0}}} \subseteq W^2_t (C)$ be general in such a component as in Lemma \ref{lem:w2tLar}, i.e. $h^0(g^2_t) =3$ and $\frac{2}{3} g +2 \leq t\leq g+2$ by $\rho(g,2,t) = 3t-2g-6 \geq 0$. Consider $\mathcal W_2 \subsetneq {\rm Ext} ^1 (K_C  , K_C-g^2_{t})$ as in \eqref{W1} and assume there exists an irreducible, positive dimensional locus $\mathcal U_2 \subseteq \mathcal W_2$, whose general point $u \in \mathcal U_2$ is such that ${\rm corank}(\partial_u) = 2$. Let  $\mathcal F_u$ be the rank-$2$ vector bundle arising from such a general $ u \in \mathcal U_2$, so fitting in an exact sequence of the form: 
\begin{equation}\label{seqg2}
(u) \ \ \  0\to K_C-g^2_{t}  \to \mathcal F_ u\to  K_C \to 0.
\end{equation} Let $F_u:= \Pp(\Ff_u)$ denote the surface scroll associated to $\Ff_u$. Then: 

\smallskip

\noindent
$(i)$  the bundle $\mathcal F_u$ fits also in an exact sequence of the form:$$(u_p) \ \ \ \ \  0\to K_C-g^2_{t} +p  \to \mathcal F_ u\to  K_C -p\to 0,$$where $p \in C$ is general; 
  \smallskip

\noindent
$(ii)$ for $p \in C$ general, the section $\Gamma _p$  corresponding to $ \mathcal F_ u\to \!\! \to  K_C -p$ is linearly isolated in $F_u$, in the sense of Definition \ref{def:ass0}; 
\smallskip

\noindent
$(iii)$  for $p \in C$ general, the sections $\Gamma _p$ are all algebraically equivalent on $F_u$, belonging to a $1$-dimensional, integral component of ${\rm Div}^{1,2g-3}_{F_u}$; 

\smallskip

\noindent
$(iv)$ for $p \in C$ general, $K_C -p$ is  of minimal degree among effective and special quotients of $\Ff_u$; 

\smallskip

\noindent
$(v)$ $\mathcal F_u$ is stable for either  $d \geq 3g-6$, when $\frac{g+3}{3} \leq \nu < \frac{g}{2}$, or $d \geq 4g-3-4\nu$, when $4 \leq \nu \leq \frac{g+2}{3}$. 
\end{lemma}
  
\begin{proof} $(i)$ The proof is similar to that in Proposition \ref{ChoiA}; since $\mathcal{F}_u$ fits in \eqref{seqg2} and is of speciality $3$, the coboundary map $\partial_u$ has corank $2$. Since $\det(\mathcal{F}_u) \simeq 2K_C - g^2_t$, duality implies $h^0(\mathcal{F}_u - (K_C - g^2_t))=3$. Setting $F_u := \mathbb{P}(\mathcal{F}_u)$, the canonical section $\Gamma$ of $F_u$ satisfies $\dim(|\Gamma|) = 2$ by \eqref{eq:isom2}. Thus, $F_u$ contains a section $\Gamma_p$ corresponding to $\omega_C(-p)$, and $\mathcal{F}_u$ fits in extension $(u_p)$.

\smallskip

\noindent
$(ii)$ For $q \in F_u$ general point with projection $p \in C$, the linear pencil $|\mathcal{I}_{q/F_u}(\Gamma)|$ contains a unique section $\Gamma_p$ after splitting the fiber $f_p$. Thus $\Gamma_p$ is linearly isolated, meaning $h^0(\mathcal{F}_u(-K_C-p+g^2_t))=1$, which is equivalent to $h^1(\mathcal{F}_u(p))=1$.

\smallskip

\noindent
$(iii)$ By \eqref{eq:Ciro410}, $\mathcal{N}_{\Gamma_p/F_u} \simeq g^2_t - 2p$. Since general $g^2_t \in \overline{W^{\vec{w}_{2,0}}}$ is birationally very ample by Lemma \ref{lem:w2tLar}, one has $h^0(g^2_t - 2p)=1$, so $\dim_{[\Gamma_p]}(\text{Div}^{1,2g-3}_{F_u})=1$. The locus $\mathfrak{G}_U := \{\Gamma_p\}_{p \in U}$ is a 1-dimensional, generically smooth component of the Hilbert scheme of unisecants of $F_u$ of degree $\delta=2g-3$ so sections $\Gamma_p$ are algebraically (but not linearly) equivalent and all of speciality $1$.

\smallskip

\noindent
$(iv)$ If $\omega_C(-p)$ were not of minimal degree among effective and special quotients  of $\Ff_u$, then there would exist $q \in C$ for which $\Ff_u$ would also fit in 
$$0 \to K_C - g^2_t + p + q \to \mathcal{F}_u \to K_C - p - q \to 0.$$This would imply $h^1( \mathcal F_ u ) \leq h^0 (g^2_{t} -p -q) +h^0 (p+q) = h^0 (g^2_{t} -p -q)  + 1$, as $C$ is not hyperelliptic. Since a general $g^2_t \in \overline{W^{\vec{w}_{2,0}}} \subseteq W^2_t (C)$ is birationally very ample by Lemma \ref{lem:w2tLar}, we haave $h^0 (g^2_{t} -p -q)=1$, for $p \in C$ general, hence, from above,  $h^1( \mathcal F_ u )  \leq 2$. This would contradict $h^1( \mathcal F_ u ) =3$.

\smallskip

\noindent
$(v)$ Concerning stability, assume first by contradiction that $\mathcal{F}_u$ is unstable with maximal subbundle $M$, $\deg(M) > 2g - 2 - \frac{t}{2}$. Since from $(i)$ $\Ff_u$ fits also in $(u_p)$, for $p\in C$ general, by determinantal reasons one has the following diagram: 
{\footnotesize
 \begin{equation}\label{eq:1b}
\begin{array}{ccccccccccccccccccccccc}
&&&&0&&&&\\[1ex]
&& &&\downarrow&&&&\\[1ex]
&&  &  & M &  &  & &  \\[1ex]
&& && \downarrow && & \\[1ex]
0&\rightarrow & K_C-g^2_{t} +p & \rightarrow & \mathcal F_u  & \rightarrow &
K_C-p &\rightarrow & 0 \\[1ex]
&& &&\downarrow &&   &&\\[1ex]
& & &  & 2K_C-g^2_{t} -M&& && \\[1ex]
&& &&\downarrow&&&&\\[1ex]
&&&&0&&&&
\end{array}
\end{equation}} Since $\deg(M) > \frac{d}{2} = 2g-2 - \frac{t}{2}$ whereas 
$\deg(K_C-g^2_t + p) = 2g-1 - t$ and since $t >2$, one has $\deg(M) > \deg( K_C-g^2_t + p)$. Hence there is no line-bundle map from $M$ to $K_C-g^2_t + p$, so the composition of the two maps $M \hookrightarrow \Ff_u$ and $\Ff_u \to\!\!\!\to  K_C-p$ from \eqref{eq:1b} gives a non-zero line-bundle map $M \to K_C-p$ which, necessarily, has to be injective. Therefore $M= K_C -p -E_{e, p}$ for some effective divisor $E_{e, p}$ of degree $e>0$, otherwise $M$ would be isomorphic to $K_C-p$ and, in such a case, $\Ff_u$ would correspond to the zero-vector $\Ff_0$ of the extension space, contradicting that $\Ff_u$ arises from a general element of a positive dimensional irreducible locus $\mathcal U_2$. For determinantal reasons, we have therefore:
\begin{equation} \label{eq:exactM}
0 \to M= K_C -p -E_{e, p} \to \mathcal F_ u \to K_C-g^2_{t} +p +E_{e, p}\to 0,
\end{equation}which yields $h^1 ( \mathcal F_ u) \leq h^0 (p+ E_{e, p} ) +h^0 ( g^2_{t} -p -E_{e, p})$. On the one hand, a general $g^2_t \in \overline{W^{\vec{w}_{2,0}}} \subseteq W^2_t (C)$ is birationally very ample by Lemma \ref{lem:w2tLar}, so $h^0 ( g^2_{t} -p -E_{e, p} )\leq 1$, being $p\in C$ general and $E_{e, p}$ effective. From \eqref{eq:exactM} one gets $ h^0 (p+ E_{e, p} ) \geq 2$ since $h^1 ( \mathcal F_ u)=3$. In other words, $\mathcal O_C(p+E_{e, p}) \in W^1_{e+1}(C)$. 

We remark that   $e+1:= \deg (p+E_{e, p}) < \frac{t}{2}$, due to $\deg (M)  > \frac{d}{2} = 2g-2- \frac{t}{2}$. Moreover, from the assumptions, $t \leq g+2$ so $\frac{t}{2} \leq \frac{g+2}{2}$. Since $e+1 < \frac{t}{2} \leq \frac{g+2}{2}$, from Lemma \ref{lem:Lar} it follows that $W^1_{e+1}(C) = \overline{W^{\vec{w}_{1,1}}}$, namely $|p+E_{e, p}| = g^1_{\nu} + B_{e+1-\nu}$, where $B_{e+1-\nu}$ denotes an effective divisor of base points of $|p+E_{e, p}|$. This means that $h^0(p+E_{e, p} - g^1_{\nu}) \geq 1$. From Lemma \ref{lem:w2tLar}, if we consider the maximal splitting vector associated to $g^2_t \in \overline{W^{\vec{w}_{2,0}}} \subseteq W^2_t (C)$ general, it is $\mathcal O_{\Pp^1}(\vec{w}_{2,0}) = B(\nu-3, t+1-g-\nu) \oplus B(3,0)$. Since $g^1_{\nu}$ corresponds to $\mathcal O_{\Pp^1}(1)$, one sees that 
$$h^0 (g^2_t -g^1_\nu) = h^0(B(\nu-3, t-g-\nu) \oplus B(3,-1)) = 0$$so, a fortiori, 
$h^0 ( g^2_t -(p +E_{e, p}) )=0$, since $|p +E_{e, p}| = g^1_{\nu} + B_{e+1-\nu}$. Using this, combined with $h^1 (\mathcal F_u)=3$ in the previous exact sequence, we get that $h^0 (p+ E_{e, p}) \geq 3$, i.e. $\mathcal O_C(p+ E_{e, p}) \in W^2_{e+1}(C)$. 

Since $\deg(p+ E_{e, p}) < \frac{g+2}{2}$ and since $\nu \geq 4$ from the assumptions, by classification in Lemma \ref{lem:w2tLar} we have that 
$W^2_{e+1}(C) = \emptyset$ when $\frac{g+3}{3} \leq \nu < \frac{g}{2}$  (this is because $e+1 < \frac{g+2}{2}$ and the latter integer is strictly less than all the upper-bounds appearing in the first line of the lists of cases occurring in $(iii)$, $(iv)$, $(v)$, $(vi)$ and $(vii)$ of Lemma \ref{lem:w2tLar}), this gives a contradiction when
$$d \geq 3g-6 \;\;\; {\rm and} \;\;\; \frac{g+3}{3} \leq \nu < \frac{g}{2},$$as, from $(u)$, one has $d = 4g-4 - t$ with $t \leq g+2$. 

When otherwise $ 4 \leq \nu \leq \frac{g+2}{3}$, from Lemma \ref{lem:w2tLar}, $W^2_{e+1}(C) $ is not empty and irreducible consisting only of $\overline{W^{\vec{w}_{2,2}}}$. In particular, one must have 
$e+1 - 2\nu \geq 0$, so one has $2\nu \leq e+1 < \frac{t}{2}$, which forces $t >  4\nu$ and so 
$d = 4g-4-t < 4g-4-4\nu$. Therefore, we get contradiction for $d \geq 4g-4-4\nu \;\;\; {\rm and} \;\;\;  4 \leq \nu \leq \frac{g+2}{3}$.

Assume now by contradiction that $\Ff_u$ is strictly semi-stable, which may occur only when $d$ is even. This means that there exists a maximal sub-line bundle $M \subset \Ff_u$, with $\deg( M) = \frac{d}{2} = 2g-2-\frac{t}{2}$.  As above, from diagram \eqref{eq:1b} and from $t >2$ , we get $\deg(M) > \deg( K_C-g^2_t + p)$ so  
there exists some effective divisor $E_{e, p}$ of degree $e>0$ for which $M= K_C -p -E_{e, p}$, because $\Ff_u$ arises from a general element of a positive dimensional irreducible locus $\mathcal U_2 \subseteq \mathcal W_2$. 
One gets \eqref{eq:exactM} from which, with same reasoning as above, one deduces  $\mathcal O_C(p+E_{e, p}) \in W^2_{e+1}(C)$. 

For $e+1 = \frac{t}{2} < \frac{g+2}{2}$, as above, we get contradiction for $d \geq 3g-6 \;\;\; {\rm and} \;\;\; \frac{g+3}{3} \leq \nu < \frac{g}{2}$ since, from Lemma \ref{lem:w2tLar}, one has $ W^2_{e+1}(C) = \emptyset$. When otherwise $ 4 \leq \nu \leq \frac{g+2}{3}$ and $e+1 = \frac{t}{2} < \frac{g+2}{2}$, from Lemma \ref{lem:w2tLar} $W^2_{e+1}(C) = \overline{W^{\vec{w}_{2,2}}}$ so one has $2\nu \leq e+1= \frac{t}{2}$, which forces $t \geq 4\nu$ and so $d = 4g-4-t \leq 4g-4-4\nu$. Therefore, we get contradiction for 
$d \geq 4g-3-4\nu \;\;\; {\rm and} \;\;\;  4 \leq \nu \leq \frac{g+2}{3}$. 

If otherwise $e+1=\frac{t}{2} = \frac{g+2}{2}$, then $d=3g-6$ which must be even, so also $g$ has to be even. From the first part of the above arguments, one gets $\mathcal O_C(p+E_{e, p}) \in W^1_{\frac{g+2}{2}}(C)$ which, from Lemma \ref{lem:Lar}, is such that $W^1_{\frac{g+2}{2}}(C) = \overline{W^{\vec{w}_{1,1}}} \cup \overline{W^{\vec{w}_{1,0}}}$. If we are in the first component, i.e. $\mathcal O_C(p+E_{e, p}) \in \overline{W^{\vec{w}_{1,1}}}$, then $| p+E_{e, p}| = g^1_{\nu} + B_{\frac{g+2}{2} -\nu}$ and, with the same reasoning as above, one deduces that $\mathcal O_C(p+E_{e, p}) \in W^2_{\frac{g+2}{2}}(C)$. In such a case, when $d = 3g-6 \;\;\; {\rm and} \;\;\; \frac{g+3}{3} \leq \nu < \frac{g}{2}$ we get contradiction because, from Lemma \ref{lem:w2tLar}, $W^2_{\frac{g+2}{2}}(C) = \emptyset$. When otherwise $ 4 \leq \nu \leq \frac{g+2}{3}$, once again one has  $W^2_{\frac{g+2}{2}}(C) = \overline{W^{\vec{w}_{2,2}}}$ which forces $\frac{t}{2} = \frac{g+2}{2} \geq 2\nu$, i.e. $\nu \leq \frac{g+2}{4}$, so we get contradiction for 
$d = 3g-6\;\; {\rm and} \;\; \frac{g+3}{4} \leq \nu < \frac{g}{2}$. If otherwise $\mathcal O_C(p+E_{e, p}) \in  \overline{W^{\vec{w}_{1,0}}} \subset W^1_{\frac{g+2}{2}}(C)$, in this case one has $\rho(g,1,\frac{g+2}{2}) =0$, i.e. $\dim(\overline{W^{\vec{w}_{1,0}}})=0$ whereas $g^2_t = g^2_{g+2} \in  \overline{W^{\vec{w}_{2,0}}} \subset W^2_{g+2}(C)$ is general in the $g$-dimensional component $\overline{W^{\vec{w}_{2,0}}}$, where the statement on the dimension follows from Lemma \ref{lem:w2tLar}. From  \eqref{eq:exactM} one has 
$$3 = h^1(\Ff_u) \leq h^0(g^2_{g+2} - (p + E_{e, p})) + h^0(p + E_{e, p}).$$Since $h^0(g^2_{g+2}) =3$, then $ 0 \leq h^0(g^2_{g+2} - (p+E_{e, p})) \leq 2$ because, as reminded above, $g^2_{g+2}$ general is birationally very-ample whereas $p+E_{e, p}$ is assumed to be in $ \overline{W^{\vec{w}_{1,0}}} \subset W^1_{\frac{g+2}{2}}(C)$. Furthermore $h^0(g^2_{g+2} - (p+E_{e, p})) = 2$ cannot occur: otherwise either 
$g^2_{g+2} - (p+E_{e, p}) \in  \overline{W^{\vec{w}_{1,0}}}$ which has dimension 
$0$ so, since also $\mathcal O_C(p+E_{e, p}) \in \overline{W^{\vec{w}_{1,0}}}$ whereas 
$\dim (\overline{W^{\vec{w}_{2,0}}}) = g$, we would get contradiction, or 
otherwise $ g^2_{g+2} - (p+E_{e, p}) \in  \overline{W^{\vec{w}_{1,1}}}$ which, from Lemma \ref{lem:Lar}, is such that $\dim(\overline{W^{\vec{w}_{1,1}}}) = \frac{g+2}{2}- \nu$ so once again, taking into account that $\mathcal O_C(p+E_{e, p}) \in \dim(\overline{W^{\vec{w}_{1,0}}})$ which is $0$-dimensional whereas 
$g^2_{g+2} \in \overline{W^{\vec{w}_{2,0}}}$ varies in dimension $g$, we would get another contradiction. Notice that also $h^0(g^2_{g+2} - (p+E_{e, p}) ) = 1$ cannot occur, otherwise 
$g^2_{g+2} = (p+E_{e, p}) + B_{\frac{g+2}{2}} = g^1_{\frac{g+2}{2}}+ B_{\frac{g+2}{2}}$, where $B_{\frac{g+2}{2}}$ an effective base-divisor, another contradiction. Thus the only possibility is $h^0(g^2_{g+2} - (p+E_{e, p})) = 0$ so, from the previous inequality, we deduce that $ h^0(p + E_{e, p}) \geq 3$, i.e. $\mathcal O_C(p + E_{e, p}) \in W^2_{\frac{g+2}{2}} (C)$. From Lemma \ref{lem:w2tLar}, we get contradiction for 
$d = 3g-6 \;\;\; {\rm and} \;\;\; \frac{g+3}{3} \leq \nu < \frac{g}{2}$ When otherwise $4 \leq \nu \leq \frac{g+2}{3}$, from Lemma \ref{lem:w2tLar} 
one has $W^2_{\frac{g+2}{2}} (C) = \overline{W^{\vec{w}_{2,2}}}$, which forces 
$\frac{g+2}{2} - 2 \nu \geq 0$, i.e. $\nu \leq \frac{g+2}{4}$, therefore, once again, we get contradiction for $d= 3g-6 \; \; {\rm and} \;\;  \frac{g+3}{4} \leq \nu \leq \frac{g+2}{3}$. 

To sum-up (v) is completely proved. 
\end{proof}

\begin{proposition}\label{thm:sonja}  Let $g \geq 9$, $\frac{g+3}{4} \leq \nu < \frac{g}{2}$ and $3g-6 \leq d \leq \frac{10}{3}g-6$ be integers and let $C$ be a general $\nu$-gonal curve of genus $g$. Then,  $B^{k_3}_d  \cap U^s_C (d)$ admits a {\em regular}  component $\mathcal B_{\rm reg,3}$  whose general point $[\mathcal F]$ corresponds to a rank-$2$ stable bundle $\Ff$ of degree $d$ and speciality $h^1(\Ff)=3$, which fits into an exact sequence \eqref{seqg2} as in Lemma \ref{lem:KC-P}, where $g^2_t \in \overline{W^{\vec{w}_{2,0}}} \subseteq W^2_t (C)$ general as in Lemma \ref{lem:w2tLar}, where necessarily $\frac{2}{3}g+2 \leq t \leq g+2$ and where $\Ff$ satisfying all the properties $(i)$-$(v)$ as in Lemma \ref{lem:KC-P}. 
\end{proposition}

\begin{proof} Notice that if $\Ff$ as in the statement fits in \eqref{seqg2} of Lemma \ref{lem:KC-P}, then $d = \deg(\Ff) = 4g-4-t$. Therefore, from bounds on $d$, one has $\frac{2}{3}g+2 \leq t \leq g+2$ as in the assumptions of Lemma \ref{lem:KC-P}. Recall moreover that, from Lemma \ref{lem:KC-P}, a general $\Ff_u$ as in the assumptions therein turns out to be stable either for $d \geq 3g-6$, when $\frac{g+3}{3} \leq \nu < \frac{g}{2}$, or for $d \geq 4g-3-4\nu$, when $4 \leq \nu \leq \frac{g+2}{3}$. Notice that, when $\nu \geq \frac{g+3}{4}$ one has $4g-3-4\nu \leq 3g-6$; in particular, for any $\frac{g+3}{4} \leq \nu < \frac{g}{2}$, $\Ff_u$ as in Lemma \ref{lem:KC-P} turns out to be stable for any $3g-6 \leq d \leq \frac{10}{3}g-6$ which is the range of interest. This explains the assumptions on $\nu$ in the above statement. 

To construct the component $\mathcal B_{\rm reg,3}$, following Lemma \ref{lem:KC-P}, we therefore need first to produce a positive dimensional irreducible locus $\mathcal U_2 \subseteq \mathcal W_2$ in $\ext^1(K_C, K_C- g^2_t)$  with all properties listed therein. Consider $g^2_t \in \overline{W^{\vec{w}_{2,0}}} \subseteq W^2_t (C)$ general where, from Lemma \ref{lem:w2tLar}, $\dim(\overline{W^{\vec{w}_{2,0}}} ) = 3t-2g-6 \geq 0$, the latter inequality following from  the bounds on $t$ (arising from those on $d$). 
Then, consider $\ext^1(K_C, K_C- g^2_t)$ which is of dimension $m:= \dim(\ext^1(K_C, K_C- g^2_t)) = 
h^1(- g^2_t) = h^0(K_C+ g^2_t) = g-1+t \geq \frac{5}{3}g +1$, from the bounds on $t$. Using notation as in \eqref{eq:barj-1} and in Theorem \ref{CF5.8}, one has $\ell := h^0(K_C) = g$ and $r := h^1(K_C- g^2_t) = h^0(g^2_t)=3$, by generality of $g^2_t \in  \overline{W^{\vec{w}_{2,0}}}$. Since $\ell \geq r$ and $m \geq \ell+1$, from Theorem \ref{CF5.8}, $v \in \ext^1(K_C, K_C- g^2_t)$ general gives rise to a vector bundle $\Ff_v$ of speciality $h^1(\Ff_v)=1$, since the corresponding coboundary map $\partial_v$ is surjective, whereas the locus $\mathcal W_1 \subsetneq \ext^1(K_C, K_C- g^2_t)$ as in \eqref{W1} is not-empty, irreducible, of ({\em expected}) dimension $\dim(\mathcal W_1) = m - c(1) = m-(\ell-r+1) = 
(g-1+t) - (g-3+1) = t+1$ (recall \eqref{eq:clrt}. Moreover, for $w \in \mathcal W_1$ general, ${\rm corank} (\partial_w) =1$. Thus, considering $\mathcal W_2$ defined as in \eqref{W1}, one has necessarily $\mathcal W_2 \subsetneq \mathcal W_1$. 

From \eqref{eq:clrt}, the expected codimension of 
$\mathcal W_2$ in $\ext^1(K_C, K_C- g^2_t)$, is given by $c(2) = 2 (\ell-r+2) = 2 (g-1) = 2g-2$. In particular
$${\rm expdim}(\mathcal W_2) = {\rm max} \{0,\; m-c(2)\} = {\rm max} \{0,\; (g-1+t)- (2g-2)\} 
= {\rm max} \{0,\; t-g+1\},$$in other words $\mathcal W_2$ is expected to be 
$\mathcal W_2 = \mathcal W_3 = \{\Ff_0 = (K_C-g^2_t) \oplus K_C\}$, for any $\frac{2}{3}g +2 \leq t \leq g-1$, and to be of dimension $j \in \{1,2,3\}$ for $t = g+j-1$.  The next result shows that this is not the case.

\begin{claim}\label{cl:step1} For any $\frac{2}{3} g +2 \leq t \leq g+2$ and for $g^2_t \in \overline{W^{\vec{w}_{2,0}}} \subseteq W^2_t (C)$ general, the locus 
 $\mathcal W_2 \subsetneq \ext^1(K_C, K_C- g^2_t)$ is not-empty and it contains an irreducible 
component $\mathcal U_2 \subseteq \mathcal W_2$ which is {\em uniruled}, of $\dim(\mathcal U_2) = 2$ and whose general point $u \in \mathcal U_2$ gives rise to an exact sequence \eqref{seqg2} as in Lemma \ref{lem:KC-P}, i.e. with ${\rm corank}(\partial_u) = 2$. For $\frac{2}{3}g +2 \leq t \leq g$, the component $\mathcal U_2$ is {\em superabundant}, i.e. 
$\dim(\mathcal U_2) > {\rm expdim} (\mathcal W_2) = {\rm max}\{0, \; t-g+1\}$. 
\end{claim}

\begin{proof} The superabundance of $\mathcal{U}_2$ will follow from our previous computation $\text{expdim}(\mathcal{W}_2) = \max\{0, t-g+1\}$ and form the construction of $\mathcal{U}_2$. 

To perform such construction observe that, since $g^2_t \in \overline{W^{\vec{w}_{2,0}}} \subseteq W^2_t(C)$ is general, Lemma \ref{lem:w2tLar} ensures it is a complete, birationally very ample linear series. Thus, there exists an open dense $U \subseteq C$ such that for any $p \in U$, $|g^2_t-p|$ is a base-point-free pencil. This defines a $1$-dimensional irreducible locus in the Grassmannian $\mathbb{G} := \mathbb{G}(2, H^0(g^2_t))$ given by $\mathcal{V}_{2,U} := \{[V_{2,p} := H^0(g^2_t-p)]\}_{p \in U}$. For each such $[V_{2,p}] \in \mathcal{V}_{2,U}$, we consider the restricted multiplication map $\mu_{V_{2,p}} : V_{2,p} \otimes H^0(K_C) \to H^0(K_C+g^2_t)$. By the base-point-free pencil trick$$ \dim(\ker(\mu_{V_{2,p}})) = h^1(g^2_t-p) = g+2-t. $$Consequently, $\dim(\text{im}(\mu_{V_{2,p}})) = 2g - (g+2-t) = g+t-2$. 

Recall the definition of the degeneracy locus $\mathcal{W}_2 \subset \text{Ext}^1(K_C, K_C-g^2_t) \simeq H^0(K_C+g^2_t)^\vee$:$$ \mathcal{W}_2 = \{ u \mid \exists V \subseteq H^0(g^2_t), \dim(V) \geq 2, \text{im}(\mu_V) \subseteq \ker(u) \}. $$Given $m := \dim(H^0(K_C+g^2_t)) = g+t-1$, the subspace $\text{im}(\mu_{V_{2,p}})$ is a hyperplane in $H^0(K_C+g^2_t)$. Therefore, $\mathcal{V}_{2,U}$ induces an irreducible 1-dimensional locus $\widehat{\mathcal{U}}_2 \subset \mathbb{P}(\text{Ext}^1(K_C, K_C-g^2_t))$. For a general $[u] \in \widehat{\mathcal{U}}_2$, we have $\text{coker}(\partial_u)^\vee \simeq V_{2,p}$. The affine cone $\mathcal{U}_2$ over $\widehat{\mathcal{U}}_2$ is thus a $2$-dimensional irreducible locus in $\mathcal{W}_2$ satisfying the required properties.

We want to prove that $\mathcal U_2$ is actually an irreducible component of $\mathcal W_2$. Let $\mathcal U_2^* \subseteq \mathcal W_2$ be any irreducible component of $\mathcal W_2$ containing $\mathcal U_2$. Since, for $u \in \mathcal U_2$ general ${\rm corank}(\partial_u) = 2$, by semicontinuty on $\mathcal U_2^*$ and by the fact that $\mathcal U_2^* \subseteq \mathcal W_2$ (so for $u^* \in \mathcal U_2^*$ general one has at least ${\rm corank}(\partial_{u^*}) \geq 2$), one deduces that $u^* \in \mathcal U_2^*$ general has ${\rm corank}(\partial_{u^*}) = 2$. From the properties of the loci $\mathcal U_2$ and $\mathcal U_2^*$ and from Lemma \ref{lem:KC-P}, $u \in \mathcal U_2$ and  $u^* \in \mathcal U_2^*$ general give both rise to rank-two vector bundles $\Ff_u$ and $\Ff_{u^*}$ such that $h^1(\Ff_u) = h^1(\Ff_{u^*}) = 3$ and satisfying all the properties $(i)$-$(v)$ therein. In particular, from Lemma \ref{lem:KC-P}-$(i)$, both bundles $\Ff_u$ and $\Ff_{u^*}$ fit also in an exact sequence of the form $(u_p)$, for $p \in C$ general. 

Since $h^0(K_C-p) = g-1$ and $h^1(K_C-g^2_t + p) = h^0(g^2_t-p) = 2$, the fact that $h^1(\Ff_u) = h^1(\Ff_{u^*}) = 3$ reads, in the vector space $\ext^1(K_C-p,K_C-g^2_t + p)$, that $u, u^*$ must be elements in $$\left(\coker \left(H^0(K_C+g^2_t-2p) \stackrel{\Phi_p}{\leftarrow} H^0(K_C-p) \otimes H^0(g^2_t-p) \right)\right)^{\vee},
$$in other words they both have to belong to $\im (\Phi_p)^{\perp}$, sub-vector space of $\ext^1(K_C-p,K_C-g^2_t + p)$. Since, $|g^2_t-p|$ is a base-point-free pencil, by the base-point-free pencil trick we have 
$\dim(\ker (\Phi_p)) = h^0(K_C-g^2_t) = h^1(g^2_t) = g+2-t$ as computed above. Thus, 
$\dim(\im (\Phi_p))  = 2 (g-1) - (g+2-t) = g+t-4$, whereas 
$h^0(K_C+g^2_t-2p) = \dim(\ext^1(K_C-p, K_C-g^2_t+p)) = g+t-3$ from \eqref{eq:m} since $K_C-p \ncong K_C-g^2_t+p$. Hence $\im (\Phi_p)^{\perp} \simeq \mathbb C$ is a $1$-dimensional sub-vector space in  $\ext^1(K_C-p, K_C-g^2_t+p)$. Since this holds for $p \in C$ general and since, from Lemma \ref{lem:KC-P}, $\Ff_{u^*}$ admits all quotients $K_C-p$, for $p$ varying in a suitable open dense subset  $U \subseteq C$, this gives rise to a dominant map $\psi$ defined as:
$$U \times \mathbb C \stackrel{\psi}{\longrightarrow} \mathcal U_2^*:\;\;\; (p, u_p)\stackrel{\psi}{\longrightarrow} \Ff_{u_p},$$where $u_p$ is a vector in $\im (\Phi_p)^{\perp} \simeq \mathbb C$ and where $\Ff_{u_p}$ is the corresponding bundle in $\mathcal U_2^*$. 

Thus $\dim(\mathcal U_2^*) = 2$, so $\mathcal U_2 = \mathcal U_2^*$, i.e. $ \mathcal U_2$ is actually an irreducible component of $\mathcal W_2$ as desidered. The previous dominant map $\psi$ shows also that $ \mathcal U_2$ is uniruled. This concludes the proof of the claim.  
\end{proof}

From Lemma \ref{lem:KC-P}-(v), for general $u \in \mathcal{U}_2$, the bundle $\mathcal{F}_u$ is stable with speciality $h^1(\mathcal{F}_u) = 3$. To construct the component $\mathcal{B}_{\text{reg},3} \subseteq B_d^{k_3} \cap U_C^s(d)$, let $g^2_t$ vary in an open dense subset $S \subseteq \overline{W^{\vec{w}_{2,0}}}$. By Claim \ref{cl:step1}, $\widehat{\mathcal{U}}_2 := \mathbb{P}(\mathcal{U}_2) \simeq U \subseteq C$ since each $\mathbb{P}(\text{im}(\Phi_p^\vee)^\perp)$ is a point. This defines an irreducible variety $\widehat{\mathcal{U}}_2^{\text{Tot}}$, birational to $S \times C$, with:$$ \dim(\widehat{\mathcal{U}}_2^{\text{Tot}}) = \dim(S) + 1 = (3t-2g-6) + 1 = 3t-2g-5 = \rho_d^{k_3} + 1, $$where the dimension of $\overline{W^{\vec{w}_{2,0}}}$ follows from Lemma \ref{lem:w2tLar} and from $d = 4g-4-t$. There is a therefore a natural rational modular map:
\begin{equation}\label{eq:rationalKC-P}
 \begin{array}{ccc}
\widehat{\mathcal U}_2^{Tot}& \stackrel{\pi}{\dashrightarrow} & U^s_C(d) \\
 u= (p, g^2_t) & \longrightarrow & [\mathcal F_u],
 \end{array}
 \end{equation} where $[u] = \Pp(\im (\Phi_p)^{\perp})$ and  where 
 $\im(\pi) \subseteq B^{k_3}_d \cap U_C^s(d)$.

\begin{claim}\label{cl:step2} The general fiber of $\pi$ as in \eqref{eq:rationalKC-P} is $1$-dimensional. In particular, $\dim(\im(\pi)) = \rho_d^{k_3}$.
\end{claim}
\begin{proof} Let $[\Ff] \in \im(\pi)$ be general, so $\Ff= \Ff_u$ for $u= (g_t^2, p) \in \widehat{\mathcal U}_2^{Tot}$ general. 

Let $u'= ((g^2_t)', p') \in \pi^{-1}([\Ff_u])$, so that $\Ff_{u'} \cong \Ff_u$. Since from \eqref{seqg2} as in Lemma \ref{lem:KC-P} one has $ 2K_C - (g^2_t)' = \det(\Ff_{u'}) \simeq \det(\Ff_u) = 2K_C - g^2_t$, then one must have $g^2_t = (g^2_t)'$. Therefore the fiber has to be contained in $U \subseteq C$, but from Lemma \ref{lem:KC-P}-$(i)$, it has to consist of $U$ and the claim is proved.  
\end{proof}

To conclude that $\im(\pi)$ is a dense subset in a regular component $\mathcal B_{\rm reg,3}$ we need to compute the Petri map of the bundle arising from its general point.

\begin{claim}\label{cl:Petrimapstep3} For $[\Ff] \in \im(\pi)$ general, the Petri map $\mu_{\Ff}$ is injective. In particular the closure  of $\im(\pi)$ in $B_d^{k_3}\cap U^s_C(d)$, denoted by $\mathcal B_{\rm reg, 3}$, gives rise to an irreducible component  which is {\em regular}. 
\end{claim}

\begin{proof}[Proof of Claim \ref{cl:Petrimapstep3}] The expected dimension $\rho_d^{k_3}$ of $\mathcal{B}_{\rm reg, 3}$ follows from Claim \ref{cl:step2} and from the density of $\text{im}(\pi)$. To establish that $\mathcal{B}_{\rm reg, 3}$ is generically smooth, we show that the Petri map $\mu_{\mathcal{F}}$ is injective for a general $[\mathcal{F}] \in \text{im}(\pi)$. Recall that such a bundle fits into the presentations:
\begin{equation}\label{eq;fits 1}
(u):\;\; 0\to K_C -g^2_t \to \mathcal{F} \to K_C \to 0 \quad \text{and} \quad (u_p): \;\; 0\to K_C -g^2_t+p \to \mathcal{F} \to K_C -p\to 0,
\end{equation} where $p \in C$ and $g^2_t \in \overline{W^{\vec{w}_{2,0}}}$ are general.

Let $\mathcal{E}_u := \omega_C \otimes \mathcal{F}_u^\vee$. From $(u_p)$, we obtain the extension $0 \to \mathcal{O}_C(p) \to \mathcal{E}_u \to g^2_t - p \to 0$. For any $q \neq p$, $h^0(g^2_t - p - q) = 1$ due to the birational very ampleness of $g^2_t$ (cf. Lemma \ref{lem:w2tLar}). Thus, $h^0(\mathcal{E}_u(-q)) \leq 1 = h^0(\mathcal{E}_u) - 2$, which implies that $\mathcal{E}_u$ is globally generated.

Global generation of $\mathcal E_u$ and $h^0(\mathcal E_u) = 3$ yield the evaluation sequence $$0 \to -g^2_t = \det(\mathcal{E}_u)^{\vee} \to \mathcal{O}_C \otimes H^0(\mathcal{E}_u) \to \mathcal{E}_u \to 0.$$Thus, tensoring by $\mathcal F_u$ and taking cohomology, one gets 
\begin{equation} \label{eq;seqPetri}
0\longrightarrow H^0 (\mathcal F_u (-g^2_t)) \longrightarrow H^0 (\mathcal F_u )\otimes H^0 (\omega _C \otimes \mathcal F_u ^{\vee}) \stackrel{\mu_{\mathcal F_u}}\longrightarrow H^0 ( \omega _C \otimes \mathcal F_u  \otimes \mathcal F_u^{\vee})\, ,
\end{equation} whence $\ker (\mu_{\mathcal F_u})\simeq H^0  (\mathcal F_u (-g^2_t))$. 

Considering \eqref{eq;fits 1} and the exact sequence obtained by tensoring it with $-g^2_t$ gives rise to the following exact diagram
 {\footnotesize
\begin{equation*}
\begin{array}{ccccccccccccccccccccccc}
&&0&&0&&0&&\\[1ex]
&&\downarrow &&\downarrow&&\downarrow&&\\[1ex]
0&\lra& K_C - 2 g^2_t& \rightarrow & \mathcal F_u(-g^2_t)& \rightarrow & K_C- g^2_t &\rightarrow & 0 \\[1ex]
&&\downarrow && \downarrow && \downarrow& \\[1ex]
0&\lra & K_C- g^2_t & \rightarrow & \mathcal F_u & \rightarrow &
K_C   &\rightarrow & 0 \\[1ex]
&&\downarrow &&\downarrow &&\downarrow  &&\\[1ex]
0&\lra & \mathcal O_{g^2_t}  & \lra & \mathcal F_u |_{g^2_t}  &\lra& \mathcal O_{g^2_t} &\rightarrow& 0\\[1ex]
&&\downarrow &&\downarrow&&\downarrow&&\\[1ex]
&&0&&0&&0&&.
\end{array}
\end{equation*}}

\noindent
First notice $h^0 (K_C - 2 g^2_t)=0$; indeed we consider the following commutative diagram:

{\footnotesize   
\begin{equation*}
\begin{array}{ccccccccccccccccccc}
 H^0(g^2_t)\otimes H^0(K_C-g^2_t) &\stackrel{\mu _{g^2_t}}\longrightarrow  & H^0(K_C) & &  \\[1ex]
 \uparrow {\boldsymbol\iota}  && \parallel  & \\[1ex]
V_2\otimes  H^0 (K_C-g^2_t) &\stackrel{\mu_{V_2}}\longrightarrow &H^0 (K_C)&
\end{array}
\end{equation*}
}

\noindent
where $V_2$ is a general two-dimensional linear sub-vector space of $H^0(g^2_t)$ and where $\boldsymbol \iota$ is a natural injection. The map $\mu_{g^2_t}$ is injective, since $g^2_t$ is a general element of $\overline{W^{\vec{w}_{2,0}}}$, therefore by composition also 
$\mu_{V_2}$ is injective. Since $\mathbb P (V_2) = |V_2|$ is a general sub-pencil of the $g^2_t$, then $|V_2|$ is a base point free general pencil thus, the base-point-free pencil trick gives $h^0 (K_C -2g^2_t)= \dim(\ker(\mu_{V_2})) = 0$, as stated.

Thus, passing to cohomology in the previous diagram, we get

{\footnotesize
\begin{equation*}
\begin{array}{ccccccccccccccccccccccc}
&&&&0&&0&&\\[1ex]
&&&&\downarrow&&\downarrow&&\\[1ex]
&& 0 & \rightarrow & H^0(\mathcal F_u(-g^2_t))& \rightarrow & H^0(K_C- g^2_t) &\rightarrow & \cdots \\[1ex]
&&\downarrow && \downarrow && \downarrow& \\[1ex]
0&\lra & H^0(K_C- g^2_t) & \rightarrow & H^0(\mathcal F_u) & \stackrel{\alpha}{\longrightarrow} &
H^0(K_C)   &\rightarrow & \cdots  \\[1ex]
&&\downarrow &&\downarrow &&\downarrow  &&\\[1ex]
0&\lra & H^0(\mathcal O_{g^2_t}) \simeq \mathbb C^t  & \lra & H^0(\mathcal F_u |_{g^2_t}) \simeq \mathbb C^{2t}  &\lra& H^0(\mathcal O_{g^2_t}) \simeq \mathbb C^t &\rightarrow& \cdots 
\end{array}
\end{equation*}} Assume by contradiction $h^0(\mathcal{F}_u(-g^2_t)) \neq 0$; an element $0 \neq x \in H^0(\mathcal{F}_u(-g^2_t)) $ would map to a non-zero element in $H^0(K_C)$ through the top-right composition. However, the middle-column injection implies $x$ is also in $H^0(\mathcal{F}_u)$ for which $\alpha(x) = 0 \in H^0(K_C)$. This contradicts the fact that the image of $H^0(\mathcal{F}_u(-g^2_t))$ in $H^0(\mathcal{F}_u)$ must land in $\ker(\alpha) = H^0(K_C - g^2_t)$. Therefore, $H^0(\mathcal{F}_u(-g^2_t)) = 0$, hence $\mu_{\mathcal{F}_u}$ is injective.
\end{proof}

The previous arguments conclude the proof of Proposition \ref{thm:sonja}. 
\end{proof}


\end{document}